\documentclass[10pt,reqno]{amsart}
\usepackage[utf8]{inputenc}
\usepackage{a4wide}
\numberwithin{equation}{section}
\usepackage{mathrsfs}
\usepackage{amsfonts}
\usepackage{amsmath}
\usepackage{stmaryrd}
\usepackage{amssymb}
\usepackage{amsthm}
\usepackage{mathrsfs}
\usepackage{url}
\usepackage{amsfonts}
\usepackage{amscd}
\usepackage{indentfirst}
\usepackage{enumerate}
\usepackage{amsmath,amsfonts,amssymb,amsthm}
\usepackage{amsmath,amssymb,amsthm,amscd}
\usepackage{graphicx,mathrsfs}
\usepackage{appendix}
\usepackage{bigints}
\usepackage[numbers,sort&compress]{natbib}
\usepackage{color}
\usepackage[colorlinks, linkcolor=blue, citecolor=blue]{hyperref}

\newcommand\R{\mathbb R}

\newcommand\mbb\mathbb
\newcommand\mbf\mathbf
\newcommand\mcal\mathcal
\newcommand\mfrak\mathfrak
\newcommand\mrm\mathrm
\newcommand\msf\mathsf
\renewcommand\a\alpha
\renewcommand\b\beta
\newcommand\g\gamma
\newcommand\G\Gamma
\renewcommand\d\delta
\newcommand\D\Delta
\newcommand\e\varepsilon
\newcommand\z\zeta
\renewcommand\t\theta
\newcommand\Th\Theta
\newcommand\la\lambda
\newcommand\La\Lambda
\newcommand\s\sigma
\newcommand\si\varsigma
\newcommand\Si\Sigma
\newcommand\ups\upsilon
\newcommand\U\Upsilon
\newcommand\ph\varphi
\renewcommand\o\omega
\renewcommand\O\Omega
\newcommand\wt\widetilde
\newcommand\wh\widehat
\newcommand\ol\overline
\newcommand\ul\underline
\newcommand\mr\mathring
\newcommand\ub\underbrace
\newcommand\pa\partial
\newcommand\n\nabla
\newcommand\fa\forall
\newcommand\ex\exists
\newcommand\es\emptyset
\newcommand\wk\rightharpoonup
\newcommand\inc\hookrightarrow
\newcommand\linf\varliminf
\newcommand\lsup\varlimsup
\newcommand\os\overset
\newcommand\us\underset
\newcommand\sr\stackrel
\newcommand\Ot\Leftarrow
\newcommand\To\Rightarrow
\newcommand\map\mapsto
\newcommand\ot\leftarrow
\newcommand\lot\longleftarrow
\newcommand\lto\longrightarrow
\newcommand\tot\leftrightarrow
\newcommand\ltot\longleftrightarrow
\newcommand\sm\backslash
\renewcommand\Cup\bigcup
\renewcommand\Cap\bigcap
\newcommand\sub\subset
\newcommand\Sub\Subset
\newcommand\sne\subsetneq
\newcommand\bus\supset
\newcommand\Bus\Supset

\newcommand\eq\equiv
\newcommand\ox\otimes
\newcommand\Ox\bigotimes
\newcommand\pl\oplus
\newcommand\Pl\bigoplus
\newcommand\x\times
\renewcommand\c\circ
\newcommand\q\quad
\renewcommand\l\left
\renewcommand\r\right
\newcommand\fr\frac

\newcommand{\dd}{\,\mathrm d}

\usepackage[numbers,sort&compress]{natbib}
\usepackage[dvipsnames]{xcolor}
\usepackage{color}

\definecolor{bondiblue}{rgb}{0.0, 0.58, 0.71}
\def\sideremark#1{\ifvmode\leavevmode\fi\vadjust{\vbox to0pt{\vss
			\hbox to 0pt{\hskip\hsize\hskip1em
				\vbox{\hsize2.1cm\tiny\raggedright\pretolerance10000
					\noindent #1\hfill}\hss}\vbox to15pt{\vfil}\vss}}}%

\newtheorem{Thm}{Theorem}[section]
\newtheorem{Lem}[Thm]{Lemma}
\newtheorem{Cor}[Thm]{Corollary}
\newtheorem{Prop}[Thm]{Proposition}

\newtheorem{Rem}[Thm]{Remark}

\begin{document}

\title[Biharmonic problem]
{Refined asymptotic estimates, Morse index and nondegeneracy of solutions to a biharmonic problem with large exponents in dimension four}
\author[X. Liang, W. Peng, S. Tian, W. Wang]{Xiuda Liang, Wangwang Peng, Shuying Tian, Wenjie Wang}

\address[Xiuda Liang]{School of Mathematics and Statistics, Central China Normal University, Wuhan 430079, China}
\email{lxddd@mails.ccnu.edu.cn}

\address[Wangwang Peng]{School of Mathematics and Statistics, Central China Normal University, Wuhan 430079, China}
\email{wwpeng@mails.ccnu.edu.cn}

\address[Shuying Tian]{School of Mathematics and Statistics, Wuhan University of Technology, Wuhan 430079, China}
\email{sytian@whut.edu.cn}

\address[Wenjie Wang]{School of Mathematics and Statistics, Central China Normal University, Wuhan 430079, China}
\email{wjwang3269@mails.ccnu.edu.cn}

\begin{abstract}
We are concerned with the semilinear biharmonic problem under Dirichlet boundary conditions
\begin{equation*}
\begin{cases}
\Delta^2 u=(u^+)^{p}  &{\text{in}~\Omega},\\[0.5mm]
 u \not\equiv 0  &{\text{in}~\Omega},\\[0.5mm]
u=\dfrac{\partial u}{\partial \nu } = 0 &{\text{on}~\partial \Omega},
\end{cases}
\end{equation*}
where $\Omega\subset \R^4$ is a smooth bounded domain, \(u^+ = \max\{u,0\} \), \(\nu\) denotes the outward unit normal and $p>1$ is sufficiently large.  The basic concentration behavior of solutions satisfying the uniform energy bound is known.  In this work, we first obtain refined asymptotic estimates for the peak heights, concentration scales and rescaled profiles of multi-spike solutions.  Based on these estimates, we further analyze the generalized linearized eigenvalue problem and derive sharp bounds for the Morse and augmented Morse indices. Consequently, if the limiting $k$-spike configuration ${\mathbf{x}}_\infty$ is a nondegenerate critical point of Kirchhoff--Routh function $\Psi_k$, then
\[
m(u_p)=k+n_+\big(D^2\Psi_k({\mathbf{x}}_\infty)\big)
\]
and \(u_p\) is nondegenerate for all sufficiently large $p$ along the selected concentration subsequence. The absence of the maximum principle for the Dirichlet biharmonic operator on general domains poses main difficulties in both refined asymptotic analysis and the study of the generalized linearized eigenvalue problem. To address these difficulties, we employ Green's representation together with refined blow-up and spectral analysis.

\end{abstract}

\date{\today}
\maketitle
{\small
\smallskip
\newline
}
\section{Introduction and main results}
\setcounter{equation}{0}
In this paper, we consider the following semilinear fourth-order elliptic problem
\begin{equation}\label{1.1}
\begin{cases}
\Delta^2 u=(u^+)^{p}  &{\text{in}~\Omega},\\[0.5mm]
 u \not\equiv 0  &{\text{in}~\Omega},\\[0.5mm]
u=\frac{\partial u}{\partial \nu } = 0 &{\text{on}~\partial \Omega},
\end{cases}
\end{equation}
where $\Omega\subset \R^4$ is a smooth bounded domain, \(u^+ = \max\{u,0\} \), \(\nu\) denotes the outward unit normal and $p>1$ is sufficiently large. Our aim is to obtain refined asymptotic estimates and study the Morse index and nondegeneracy of solution families \(u_p\) satisfying the following uniform energy bound
\begin{equation}\label{1.2}
p \int_{\Omega}\big(u_p^+(x)\big)^{p+1}\,dx \le C, \qquad \text{as } p \to +\infty,
\end{equation}
where \(C>0\) is independent of \(p\).

\vskip 0.2cm

Dimension four is a natural setting for the biharmonic problem with large exponents. The embedding \(H_0^2(\Omega) \hookrightarrow L^q(\Omega)\) is compact for every \(1\le q<\infty\), and the fundamental solution of \(\Delta^2 \) in \( \mathbb{R}^4 \) is logarithmic (see \cite{GGS}). These properties parallel those of the Laplacian in dimension two. For the concentrating solutions considered here, the asymptotic analysis in \cite{santra-wei} shows that suitable rescalings around positive concentration peaks converge to a solution of the fourth-order Liouville equation \( \Delta^2 Z = e^Z\) in \(\mathbb{R}^4\). This limiting equation is also related to conformal geometry through the \(Q\)-curvature of a smooth four-dimensional Riemannian manifold (see \cite{chang,ChangYang}), defined by
\[
Q_g=\frac{1}{12}\Big(-\Delta_g S_g+S_g^2-3|\operatorname{Ric}_g|^2\Big),
\]
where \(S_g\) is the scalar curvature of the metric \(g\) and \(\operatorname{Ric}_g\) is the Ricci tensor. Under a conformal change of metric \(\widetilde g=e^{2w}g\), the \(Q\)-curvature transforms as
\[
P_gw+2Q_g=2Q_{\widetilde g}e^{4w},
\]
where the Paneitz operator \(P_g\) \cite{paneitz} is defined, for smooth \(f\), by
\[
P_g f
:=\Delta_g^2 f
- \operatorname{div}_g\left[
\left(\frac{2}{3}S_g I-2\operatorname{Ric}_g\right)\nabla_g f
\right].
\]
Here we use the convention \(\Delta_g=\operatorname{div}_g\nabla_g\), so that \(-\Delta_g\) is nonnegative, and view \(\operatorname{Ric}_g\) as an endomorphism using the metric \(g\). In the flat Euclidean setting, \(P_g=\Delta^2\) and \(Q_g=0\), so the conformal transformation law reduces to
\[
\Delta^2w=2Q_{\widetilde g}e^{4w}.
\]
This type of equation has attracted considerable and growing interest in both differential geometry and physics over the past few decades, we refer to \cite{Lin,ARS,DelPinok,M-D,LW1,LW2,HR111,rw} and references therein.


\vskip 0.2cm

It is well known that, the standard variational methods yield a least energy solution \({\tilde{u}}_p\) of problem \eqref{1.1} for any $p>1$ in any smooth bounded domain $\Omega$. The asymptotic behavior of four-dimensional biharmonic equations with large exponents was investigated for least energy solutions under Navier boundary conditions by Takahashi \cite{Takahashi2005,Takahashi2007} and Ben Ayed--El Mehdi--Grossi \cite{AMG}. For the Dirichlet boundary value problem \eqref{1.1}, Santra and Wei \cite{santra-wei}  developed the concentration analysis that underlies the present work.  In particular, for
least energy solution \({\tilde{u}}_p\), they obtained
\begin{equation*}
\lim_{p\to+\infty} p\int_{\Omega}\big({\tilde{u}}_p^+(x)\big)^{p+1}\,dx = 64\pi^{2}{e}, \quad \lim_{p\to+\infty} \|{\tilde{u}}_p\|_{\infty} = \sqrt{e},
\end{equation*}
and, up to a subsequence, for any compact subset of $\overline{\Omega} \setminus \{x_0\}$,
\[
p\,{\tilde{u}}_p(x)\to 64\pi^{2}\sqrt{e}\,G(x,x_0),
\]
where $x_0\in\Omega$ is the unique concentration point of ${\tilde{u}}_p$ and $G(x,\cdot)$ is the Green's function of $\Delta^2$ under the Dirichlet boundary condition, i.e.
\begin{equation}\label{Greenfunction}
\begin{cases}
\Delta_x^2 G(x,y)=\delta(x-y) & \text{in } \Omega,\\[2ex]
G(x,y)=\dfrac{\partial G}{\partial\nu}(x,y)=0 & \text{for } x\in\partial\Omega,
\end{cases}
\end{equation}
and $\delta(\cdot-y)$ denotes the Dirac delta function. More generally, it was also shown in \cite{santra-wei} that families of solutions satisfying the uniform energy bound \eqref{1.2} concentrate, up to a subsequence, at a finite number of distinct points in \(\Omega\). For the corresponding problem with an odd nonlinearity, Chen--Cheng--Zhao \cite{Chen-Cheng-Zhao} studied the asymptotic behavior of least energy nodal solutions to \(\Delta^2u=|u|^{p-1}u\) in dimension four. They established concentration at two distinct points, with one positive and one negative peak, extending the corresponding asymptotic results of Grossi--Grumiau--Pacella \cite{GGP} for the planar Lane–Emden equation \(-\Delta u=|u|^{p-1}u\) to the biharmonic setting. This problem differs from \eqref{1.1} in its nonlinearity, although solutions of \eqref{1.1} may also change sign on a general domain.

\vskip 0.2cm

For the classical planar Lane–Emden problem
\begin{equation}\label{lane}
\begin{cases}
-\Delta u=u^{q}  &{\text{in}~\Omega},\\[0.5mm]
 u>0  &{\text{in}~\Omega},\\[0.5mm]
u=0 &{\text{on}~\partial \Omega},
\end{cases}
\end{equation}
where $\Omega\subset \R^2$ is a smooth bounded domain and $q>1$ is sufficiently large. Ren and Wei \cite{Ren-Wei1,Ren-Wei2} established fundamental concentration results for least energy solutions, and Adimurthi and Grossi \cite{Adimurthi-Grossi} obtained sharper asymptotic estimates. Using the Lyapunov-Schmidt reduction method, Esposito--Musso--Pistoia \cite{EMP2006} constructed multi-spike solutions concentrating at critical points of a function (usually called the Kirchhoff--Routh function), which consists of the Green function \( G \) and the Robin function \( R \) associated with \( - \Delta \) in \( \Omega \) under Dirichlet boundary conditions. For general families \( u_q \) of positive solutions satisfying the following uniform  energy bound
\[
q \int_{\Omega} | \nabla u_q |^2 \,dx \le C, \qquad \text{as } q \to +\infty,
\]
where \(C>0\) is independent of \(q\). De Marchis--Ianni--Pacella \cite{DIP2017-1} described the asymptotic concentration profile, while De Marchis--Grossi--Ianni--Pacella \cite{DGIP2018} identified the sharp peak heights and energy quantization. In particular, the limiting concentration configuration of such solutions is a critical point of the Kirchhoff--Routh function.

These concentration results laid the foundation for refined asymptotic estimates of concentrating solutions to \eqref{lane} and for the analysis of the associated generalized linearized eigenvalue problem. These developments have played an important role in computing the Morse index and establishing nondegeneracy and local uniqueness of concentrating solutions for problem \eqref{lane}.  In particular, De Marchis--Grossi--Ianni--Pacella \cite{DGIP2019} computed the Morse index of single-spike solutions and obtained uniqueness of \eqref{lane} in convex domains for sufficiently large \(p\). Grossi--Ianni--Luo--Yan \cite{Grossi-Ianni-Luo-Yan} derived refined asymptotic expansions and established nondegeneracy of multi-spike solutions, together with local uniqueness of single-spike solutions. And, Ianni--Luo--Yan \cite{Ianni-Luo-Yan} subsequently computed the Morse index of multi-spike solutions and derived a degree counting formula, obtaining local uniqueness for multi-spike solutions. Further information on the first eigenvalues and the concentration of the corresponding eigenfunctions was obtained by Pan--Tian--Yang \cite{PTY}.

\vskip 0.2cm

Motivated by these developments, we investigate the Morse index of concentrating solutions to the fourth-order problem \eqref{1.1} and derive their nondegeneracy under suitable assumptions on the limiting concentration configuration. However, the absence of a general maximum principle for the Dirichlet biharmonic operator \cite{GGS,Green2} creates main difficulties in both the refined asymptotic analysis and the study of the generalized linearized eigenvalue problem. To overcome these difficulties, we adopt the Green representation formula for technical arguments and obtain the corresponding refined estimates.

\vskip 0.2cm

Firstly, we recall the known asymptotic characterization of the solutions to problem \eqref{1.1} under assumption \eqref{1.2}, following Santra and Wei \cite{santra-wei}.

\vskip 0.2cm

\newtheorem*{theoremA}{Theorem A}
\begin{theoremA}[\cite{santra-wei}]

Let $u_p$ be a family of solutions to  \eqref{1.1} satisfying \eqref{1.2}.
Then there exist finitely many distinct points $x_{\infty,j}\in \Omega$, $j=1,\cdots,k$ and a subsequence of $p$ (still denoted by $p$)  such that setting
$\mathcal{S}:=\big\{x_{\infty,1},\cdots, x_{\infty,k}\big\}$,
one has
\begin{enumerate}
\item[(i)] \begin{equation}\label{11-14-03N}
\lim_{p\rightarrow +\infty} p u_{p}=64\pi^2 \sqrt{e}\sum^k_{j=1}G(x,x_{\infty,j})\,\,~\mbox{in every compact subset of} ~ \bar{\Omega}\backslash \mathcal{S},
\end{equation}
and the energy satisfies
\begin{equation*}
\lim_{p\rightarrow +\infty} p \displaystyle\int_{\Omega}|\Delta u_{p}(x)|^2dx=64\pi^2 e\cdot k,
\end{equation*}

\item[(ii)]
 the concentrated points $x_{\infty,j}$, $j=1,\cdots,k$ satisfy
\begin{equation} \label{x_j}
\nabla_x H(x_{\infty,j}, x_{\infty,j}) + \sum_{l
\neq j}
\nabla_x G(x_{\infty,j}, x_{\infty,l}) = 0,
\end{equation}
where
\begin{equation}\label{Hfuction}
H(x,y) = G(x,y) + \frac{\ln|x-y|
}{8\pi^2} := G(x,y) - S(x,y)
\end{equation}
is the regular part of the Green's function $G$ defined by \eqref{Greenfunction}.
\end{enumerate}
Moreover, for some small fixed $r>0$, let $x_{p,j}\in \overline{B_{2r}(x_{\infty,j})}\subset\Omega$ be the sequence defined as
\begin{equation}\label{def:xpj}
u_{p}(x_{p,j})=\max_{\overline{B_{2r}(x_{\infty,j})}}u_{p}(x),
\end{equation}
then for any $j=1,\cdots,k$, it holds
\begin{equation*}
\lim_{p\rightarrow +\infty}x_{p,j}=x_{\infty,j},~~
\lim_{p\rightarrow +\infty}u_{p}(x_{p,j})=\sqrt{e}.
\end{equation*}
\end{theoremA}

Throughout the rest of the paper, we work along the concentration subsequence selected in Theorem~A, still denoted by \(p\). All limits and assertions for sufficiently large \(p\) are understood along this subsequence. When further subsequences are needed to obtain limits of eigenfunctions, we state this explicitly.

We set
\begin{equation}\label{defwpj}
\begin{aligned}
&Z_{p,j}(y):=\frac{p}{u_{p}(x_{p,j})}\Big(u_{p}(x_{p,j}+\varepsilon_{p,j}y)-
u_{p}(x_{p,j})\Big),~y\in \Omega_{p,j}:=\frac{\Omega-x_{p,j}}
{\varepsilon_{p,j}},\\
&\varepsilon_{p,j}:=\Big(p\big(u_{p}(x_{p,j})\big)^{p-1}\Big)^{-1/4} \rightarrow 0 ~~\text{as $p \rightarrow \infty$},
\end{aligned}
\end{equation}
and by blow-up analysis, one has
\begin{equation*}
\lim_{p\rightarrow +\infty} Z_{p,j}=Z\in C^4_{loc}(\R^4),
\end{equation*}
where  \begin{equation}\label{Z(x)}
    Z(x)=-4\ln \left(1+\frac{|x|^2}{8 \sqrt{6}}\right)
\end{equation} is the solution of

\begin{equation} \label{Liu}
\begin{cases}
\Delta^2 Z(x) = e^{Z(x)} & \text{for } x \in \mathbb{R}^4, \\[2ex]
\displaystyle\int_{\mathbb{R}^4} e^Z dx = 64\pi^2.
\end{cases}
\end{equation}

\vskip 0.2cm

 We recall that the Robin function is defined as
\begin{equation}\label{Robinf}
R(x):=H(x,x).
\end{equation}
For $\boldsymbol{a}=(a_1,\cdots,a_k)$, with $a_j\in \Omega$, $j=1,\cdots,k$, we define the Kirchhoff--Routh function  $\Psi_{k}: \Omega^{k} \rightarrow \R$ as
\begin{equation}\label{stts}
\Psi_{k}(\boldsymbol{a}):= \sum^k_{j=1} \Psi_{k,j}(\boldsymbol{a}),\quad~\mbox{ with }~\Psi_{k,j}(\boldsymbol{a}):=  R\big(a_j\big)+ \sum^{k}_{m=1,m\neq j} G\big(a_j,a_m\big).
\end{equation}
By the symmetry of \(G\), \eqref{x_j} implies that \( \mathbf{x_{\infty}} = ( x_{\infty,1}, \dots, x_{\infty,k})\) is a critical point of \(\Psi_k\).

Inspired by the quantitative analysis of Grossi--Ianni--Luo--Yan \cite{Grossi-Ianni-Luo-Yan}, our first result provides more precise asymptotic estimates for the peak heights, concentration scales and local profiles, which play a key role in computing the Morse index.

\begin{Thm}\label{th1}
Let $u_p$ be the solution of \eqref{1.1} given by Theorem A, and we retain the notation introduced in Theorem A. Set
\[
\mathbf{x}_p:=\big(x_{p,1},\cdots,x_{p,k}\big),\qquad
{\mathbf{x}}_{\infty}:=\big(x_{\infty,1},\cdots,x_{\infty,k}\big).
\]
Then, for a sufficiently small fixed constant $\delta\in(0,1)$ and $j=1,\cdots,k$, it holds
\begin{equation}\label{5-7-52}
\begin{split}
u_{p}(x_{p,j})=\sqrt{e}\bigg(1-\frac{\ln p}{p-1}+
\frac{1}{p}\Big(-32\pi^2 \Psi_{k,j}(\mathbf{x}_p)+2\ln(8\sqrt6)+\frac83\Big)
\bigg)+O\Big(\frac{1}{p^{2-\delta}}\Big).
\end{split}
\end{equation}
Moreover,
\begin{equation*}
\e_{p,j}= e^{-\frac{p}8}\Bigr(
 e^{\big( 8\pi^2 \Psi_{k,j}(\mathbf{x}_p)-\frac12\ln(8\sqrt6)
-\frac{13}{24}\big) }+O\big(\frac{1}{p^{1-\delta}}\big)\Bigr),
\end{equation*}
and hence
\begin{equation}\label{3-29-03}
\frac{\e_{p,j}}{\e_{p,s}}=e^{8\pi^2 \big(\Psi_{k,j}(\mathbf{x}_p)-\Psi_{k,s}(\mathbf{x}_p)\big)}
 +O\big(\frac{1}{p^{1-\delta}}\big),\qquad 1\leq j,s\leq k.
\end{equation}
Furthermore, for any compact subset of $\R^4$,
\begin{equation}\label{lst}
Z_{p,j}=Z+\frac{\eta_0}{p} +O\left(\frac{1}{p^2}\right),
\end{equation}
where $\eta_0$ solves
\begin{equation*}
\Delta^2\eta_0-e^Z\eta_0
=
-\frac{Z^2}{2}e^Z
\quad\text{in }\mathbb R^4.
\end{equation*}
If, in addition, ${\mathbf{x}}_{\infty}$ is a nondegenerate critical point of $\Psi_k$, then
\begin{equation}\label{xp-rate-intro}
|\mathbf{x}_p-{\mathbf{x}}_{\infty}|=O\Big(\frac1p\Big),
\end{equation}
and consequently $\mathbf{x}_p$ can be replaced by ${\mathbf{x}}_{\infty}$ in \eqref{5-7-52}--\eqref{3-29-03} with the same accuracy.
\end{Thm}

\begin{Rem}

In the proof of Theorem \ref{th1}, we require to establish uniform sublinear growth estimates for the first and second correction terms of the rescaled profiles on expanding domains, see Propositions \ref{key-1} and \ref{key-2}. Unfortunately, we can not use the arguments of \cite{Grossi-Ianni-Luo-Yan} in the second-order setting. Since the sign of \(\Delta^2 f\) cannot be determined at a local maximum of \(f\). To overcome this difficulty, we adopt a Green representation approach inspired by \cite{santra-wei} to conclude relevant estimates. More precisely, it separates the contribution of the selected peak from those of the other peaks and the region away from the concentration points.

A further difficulty is computing the coefficient in the refined mass expansion, which requires a separate analysis of an auxiliary fourth-order radial ODE in Appendix \ref{76}. Together, these estimates give the refined asymptotic estimates in Theorem \ref{th1} and the accuracy needed for the subsequent  analysis of the generalized linearized eigenvalue problem in Section \ref{section4}.

\end{Rem}

We next turn to the linearized operator
\[
\mathcal L_p\xi:=\Delta^2\xi-p(u_p^+)^{p-1}\xi,
\]
and its quadratic form
\[
\mathcal Q_p(\xi):=\int_\Omega |\Delta\xi|^2 \,dx -p\int_\Omega (u_p^+)^{p-1}\xi^2 \,dx, \qquad \xi\in H_0^2(\Omega).
\]
Denote by $m(u_p)$ the Morse index, namely the maximal dimension of a subspace on which \( \mathcal Q_p \) is negative definite, and set
\[
m_0 (u_p) = m (u_p) + \dim Ker \mathcal{L}_p.
\]
Equivalently, these indices count, respectively, the generalized eigenvalues below and at or below
1 in
\begin{equation}\label{eq123456}
    \Delta^2v = \lambda p (u_p^+)^{p-1} v \quad \text{in } \Omega, \quad v=\frac{\partial v}{\partial \nu} = 0 \quad \text{on } \partial \Omega.
\end{equation}
For a real symmetric matrix $M$, let $n_+(M)$, $n_-(M)$ and $n_0(M)$ denote the numbers of its positive, negative and zero eigenvalues, counted with multiplicity.

The Morse index analysis is motivated by the corresponding results for the planar Lane--Emden problem in \cite{DGIP2019,Ianni-Luo-Yan}. We analyze the generalized eigenvalues of \eqref{eq123456} and obtain sharp bounds for the Morse and augmented Morse indices, which is relevant to the Hessian of the Kirchhoff--Routh function \( \Psi_k \). When the limiting concentration configuration is a nondegenerate critical point of \( \Psi_k \), we obtain an exact Morse index formula and prove nondegeneracy of the concentrating solutions. In the analysis, the possible sign changes of solutions and the vanishing of the linearized weight create additional difficulties.

\begin{Thm}\label{thMorse}
Let $u_p$ be a family of solutions of \eqref{1.1} satisfying \eqref{1.2}, and let ${\mathbf{x}}_\infty=(x_{\infty,1},\ldots,x_{\infty,k})$ be the limiting concentration configuration given by Theorem~A. Then, along the selected concentration subsequence, for all sufficiently large $p$,
\begin{equation}\label{morse-index-ineq}
k+n_+\big(D^2\Psi_k({\mathbf{x}}_\infty)\big)
\le m(u_p)\le m_0(u_p)
\le k+n_+\big(D^2\Psi_k({\mathbf{x}}_\infty)\big)+n_0\big(D^2\Psi_k({\mathbf{x}}_\infty)\big).
\end{equation}
In particular, if ${\mathbf{x}}_\infty$ is a nondegenerate critical point of $\Psi_k$, then
\begin{equation*}
m(u_p)=m_0(u_p)=k+n_+\big(D^2\Psi_k({\mathbf{x}}_\infty)\big).
\end{equation*}
Equivalently, if
\[
m({\mathbf{x}}_\infty,\Psi_k):=n_-\big(D^2\Psi_k({\mathbf{x}}_\infty)\big)
\]
denotes the finite-dimensional Morse index of $\Psi_k$ at ${\mathbf{x}}_\infty$, then
\begin{equation*}
m(u_p)=5k-m({\mathbf{x}}_\infty,\Psi_k).
\end{equation*}
\end{Thm}

\begin{Cor}\label{th1.1}
Let \(u_p\) be as in Theorem~\ref{thMorse}. If ${\mathbf{x}}_\infty$ is a nondegenerate critical point of the Kirchhoff--Routh function $\Psi_k$ defined in \eqref{stts}, then there exists $p^*>1$ such that
\[
\ker\mathcal L_p=\{0\}
\]
for every \(p\ge p^*\) in the selected concentration subsequence. In particular, \(u_p\) is nondegenerate for all sufficiently large \(p\) along this subsequence.
\end{Cor}

\begin{Rem}
The possible sign changes of \(u_p\) require care when analyzing the generalized linearized eigenvalue problem and proving that the rescaled eigenfunction limits are not all zero. The weight \(p(u_p^+)^{p-1}\) is nonnegative but may vanish, so the associated weighted bilinear form need not be positive definite. As explained in Subsection \ref{sec1}, the compact self-adjoint formulation separates the kernel of this form from the finite generalized eigenvalues. Since \(\mathcal{Q}_p\) is strictly positive on that kernel, these directions contribute neither to the Morse index nor to the nullity of the linearized operator.

Moreover, local compactness alone would allow an \(L^\infty\)-normalized eigenfunction converge to zero on every bubble core, while its maximum remains away from all concentration points, possibly where \(u_p \le 0\). Proposition \ref{MI:prop:tube-capture} excludes this possibility for bounded eigenvalue sequences by combining Green’s representation, the logarithmic oscillation estimate in Lemma \ref{MI:lem:green-osc}, and evaluation at the peak centers. This argument connects the global eigenvalue problem with the local bubble spectrum without assuming positivity of either \(G\) or \(u_p\). Notably, for the planar Lane-Emden problem \eqref{lane}, this possibility is excluded by a comparison argument based on the maximum principle. Such an argument does not directly extend to the Dirichlet biharmonic problem, so we instead use a global Green representation argument.

\end{Rem}

Our results describe the Morse index of the concentrating solutions in Theorem A and show that these solutions are nondegenerate when the limiting configuration \( \mathbf{x}_{\infty}\) is a nondegenerate critical point of the Kirchhoff--Routh function \(\Psi_k\). This naturally raises questions about the existence and nondegeneracy of critical points of \(\Psi_k\), and their dependence on the geometric and topological properties of the domain. On a ball, positivity of the Dirichlet biharmonic Green function \cite{Boggio,GS} and the radial symmetry and strict radial monotonicity of positive solutions \cite{BGW} imply that the concentrating solution families in Theorem A can concentrate only at the center. In this case, \(\Psi_k \) reduces to the Robin function \( R \), and Boggio’s explicit formula shows that the center is its unique nondegenerate critical point.

However, the existence of critical points of \(\Psi_k\) on general domains and their related properties remain unclear. For a given general domain, related questions have been studied for the Dirichlet Laplacian. For example, nondegeneracy of Robin critical points has been established under suitable symmetry and convexity assumptions \cite{GR111}, while generic nondegeneracy under domain perturbations has been proved for both Robin and Kirchhoff–Routh functions\cite{MP,BMP}. The number, location and nondegeneracy of Robin and Kirchhoff–Routh critical points have also been studied in domains with a small hole \cite{GGLY1,GGLY2}. For the polyharmonic operator, research on this problem remains limited due to the relatively complicated structure of Dirichlet boundary conditions, while only a few studies are available under Navier boundary conditions (such as \cite{Takahashi2013}). See also the references therein for further related results.

Furthermore, it would be interesting to investigate whether analogous results on the Morse index and nondegeneracy of concentrating solutions hold under Navier boundary conditions. And another natural question is whether two multi-spike solution families of \eqref{1.1} concentrating at the same points must coincide for all sufficiently large \(p\), namely local uniqueness.

\vskip 0.2cm

The paper is organized as follows. In Section~\ref{s}, we collect basic facts about the biharmonic Green function, the linearized Liouville equation and the bilinear Pohozaev forms used throughout the paper. Section~\ref{section3} is devoted to the refined asymptotic analysis of multi-spike solutions and contains the proof of Theorem~\ref{th1}. In Section~\ref{section4}, we study the generalized linearized eigenvalue problem, identify the low and near-kernel spectrum, and prove the Morse index formula in Theorem~\ref{thMorse}; Corollary~\ref{th1.1} follows as a consequence. Some lengthy computations involving the Green function and the correction profile are deferred to Appendices~\ref{s6} and~\ref{76}, while the elementary bubble integrals used in the Morse index calculation are collected in Appendix~\ref{MI:appendix-integrals}.

\vskip 0.2cm


\section{Some results on biharmonic operator} \label{s}
In this section, we collect some useful results for the biharmonic operator which will be used in the refined asymptotic analysis and in the subsequent spectral arguments.

\subsection{Properties of Green's function for biharmonic operator}$\,$  \vskip 0.2cm
 For the Green's function of $\Delta^2$ under Dirichlet boundary conditions, there are some estimates on itself and its derivatives, and we can get more details from \cite{Green1} and \cite{Green2}.
\begin{Lem}\label{Greenfun}
Suppose that $G$ satisfies \eqref{Greenfunction}, there exists $C > 0$ such that for all $x, y \in \Omega$, $x \neq y$,
\begin{equation*}
|G(x,y)| \leq C \ln \left(1 + \frac{1}{|x-y|}\right)
\end{equation*}
and
\begin{equation*}
|\nabla^i G(x,y)| \leq \frac{C}{|x-y|^i}, \quad i = 1, 2, 3.
\end{equation*}
Moreover, there exists $C > 0$, depending on $\Omega$,
 such that
\begin{equation*}
G(x,y) \geq -C.
\end{equation*}
\end{Lem}
 \subsection{Linearization of the Liouville equation}$\,$  \vskip 0.2cm

Let
\[
a:=8\sqrt6,
\qquad
Z(x):=-4\log\left(1+\frac{|x|^2}{a}\right),
\]
so that
\[
\Delta^2 Z=e^Z
\qquad\text{in }\mathbb R^4.
\]
Consider the linearized operator of the Liouville-type equation \eqref{Liu} at the solution $Z$
\[
\mathcal L v:=\Delta^2v-e^Zv.
\]
Define
\[
\Psi_0(x):=4+x\cdot\nabla Z(x),
\qquad
\Psi_i(x):=\partial_{x_i}Z(x),
\quad i=1,\ldots,4.
\]
Explicitly,
\[
\Psi_0(x)
=
4\,\frac{a-|x|^2}{a+|x|^2},
\qquad
\Psi_i(x)
=
-\frac{8x_i}{a+|x|^2}.
\]

The bounded kernel of \( \mathcal{L} v \) has been well studied in \cite[Lemma 3.1]{Baraket}. However, in the proof of Propositions \ref{key-1} and \ref{key-2}, knowledge of the boundedness of its kernel alone is insufficient to derive the corresponding result. We therefore generalize this result to the case where the kernel satisfies subquadratic growth, as stated in the lemma below. The proof follows a similar idea to \cite[Lemma 3.1]{Baraket}.

\begin{Lem}\label{llm}
Let $v\in C^4(\mathbb R^4)$ solve
\begin{equation}\label{eq:linearized-kernel-equation}
	\Delta^2v=e^Zv
	\qquad\text{in }\mathbb R^4.
\end{equation}
Assume that for some $\tau\in[0,2)$,
\begin{equation}\label{eq:kernel-growth-assumption}
	|v(x)|
	\le
	C(1+|x|)^\tau
	\qquad\text{for all }x\in\mathbb R^4.
\end{equation}
Then there exist constants $c_0,\ldots,c_4\in\mathbb R$ such that
\begin{equation*}
	v
	=
	c_0\Psi_0
	+
	\sum_{i=1}^4c_i\Psi_i.
\end{equation*}
Equivalently,
\begin{equation*}
	\begin{aligned}
		&
		\left\{
		v\in C^4(\mathbb R^4):
		\mathcal Lv=0,\quad
		|v(x)|=O\bigl((1+|x|)^\tau\bigr)
		\text{ for some }\tau<2
		\right\}
		\\
		&\qquad\qquad
		=
		\operatorname{span}
		\left\{
		4+x\cdot\nabla Z,\,
		\partial_{x_1}Z,\ldots,\partial_{x_4}Z
		\right\}.
	\end{aligned}
\end{equation*}
In particular, the kernel of $\mathcal L$ in the subquadratic growth
class is five-dimensional.
\end{Lem}

\begin{proof}
Set
\[
    x=\sqrt a\,y,
    \qquad
    \phi(y):=v(\sqrt a\,y).
\]
Since $a^2=384$, equation \eqref{eq:linearized-kernel-equation} becomes
\begin{equation*}
    \Delta^2\phi
    =
    \frac{384}{(1+|y|^2)^4}\phi
    \qquad\text{in }\mathbb R^4.
\end{equation*}

Let $\mathcal S:\mathbb R^4\to\mathbb S^4\setminus\{N\}$ be the inverse
stereographic projection and define
\[
    \widehat\phi(\mathcal S(y)):=\phi(y).
\]
Since
\[
    \mathcal S^*g_{\mathbb S^4}
    =
    \left(\frac{2}{1+|y|^2}\right)^2dy^2,
\]
the conformal covariance of the Paneitz operator gives
\[
    (P_{\mathbb S^4}-24)\widehat\phi=0
    \qquad\text{on }\mathbb S^4\setminus\{N\},
\]
where
\[
    P_{\mathbb S^4}
    =
    \Delta_{\mathbb S^4}^2-2\Delta_{\mathbb S^4}.
\]

It remains to remove the possible singularity at $N$. Since
$|y|\asymp d_{\mathbb S^4}(\cdot,N)^{-1}$ near $N$,
\eqref{eq:kernel-growth-assumption} gives
\[
    |\widehat\phi(\zeta)|
    \le
    C\,d_{\mathbb S^4}(\zeta,N)^{-\tau}.
\]
Since $\tau<2<4$, the function $\widehat\phi$ is locally integrable
at $N$ and therefore defines a distribution on $\mathbb S^4$. Set
\[
    T:=(P_{\mathbb S^4}-24)\widehat\phi,
\]
so that $\operatorname{supp}T\subset\{N\}$. We show explicitly that $T$
depends only on the value and first derivatives of a test function at $N$.
Choose normal coordinates $z$ centered at $N$ and a smooth cutoff
$\chi_\varepsilon$ which equals one on $B_\varepsilon(N)$, is supported
in $B_{2\varepsilon}(N)$, and satisfies
$|\nabla^m\chi_\varepsilon|\le C_m\varepsilon^{-m}$ for $0\le m\le4$.
If $\psi\in C^\infty(\mathbb S^4)$ satisfies
$\psi(N)=0$ and $d\psi(N)=0$, then Taylor's formula gives
$\psi=O(|z|^2)$ and $\nabla\psi=O(|z|)$ near $N$. Since
$P_{\mathbb S^4}-24$ is a fourth-order operator with smooth coefficients,
\[
\left|(P_{\mathbb S^4}-24)(\chi_\varepsilon\psi)\right|
\le C_\psi\varepsilon^{-2}
\qquad\text{on }B_{2\varepsilon}(N).
\]
Consequently, by the support property of $T$ and self-adjointness,
\[
\begin{aligned}
|\langle T,\psi\rangle|
&=|\langle T,\chi_\varepsilon\psi\rangle|\\
&=\left|\int_{\mathbb S^4}\widehat\phi\,
(P_{\mathbb S^4}-24)(\chi_\varepsilon\psi)\,dV\right|\\
&\le C_\psi\varepsilon^{-2}
\int_{B_{2\varepsilon}(N)}
d_{\mathbb S^4}(\zeta,N)^{-\tau}\,dV
\le C_\psi\varepsilon^{2-\tau}\longrightarrow0.
\end{aligned}
\]
Thus $T$ vanishes on every test function whose value and first
derivatives vanish at $N$, and hence
\[
    T
    =
    A\delta_N
    +
    \sum_{\alpha=1}^4B_\alpha
    \partial_{z_\alpha}\delta_N.
\]
The coordinate functions $\zeta_1,\ldots,\zeta_5$ on $\mathbb S^4$
satisfy
\[
    -\Delta_{\mathbb S^4}\zeta_j=4\zeta_j,
    \qquad
    (P_{\mathbb S^4}-24)\zeta_j=0.
\]
By self-adjointness,
\[
    \langle T,\zeta_j\rangle=0,
    \qquad j=1,\ldots,5.
\]
Since $\zeta_5(N)=1$, $d\zeta_5(N)=0$, while
$d\zeta_1(N),\ldots,d\zeta_4(N)$ form a basis of
$T_N^*\mathbb S^4$, we obtain
\[
    A=B_1=\cdots=B_4=0.
\]
Thus $T=0$ and, by elliptic regularity,
\[
    \widehat\phi\in C^\infty(\mathbb S^4).
\]

We may therefore use the spherical harmonic decomposition. If
$Y_\ell$ is a spherical harmonic of degree $\ell$, then
\[
    -\Delta_{\mathbb S^4}Y_\ell
    =
    \lambda_\ell Y_\ell,
    \qquad
    \lambda_\ell=\ell(\ell+3),
\]
and hence
\[
    (P_{\mathbb S^4}-24)Y_\ell
    =
    (\lambda_\ell^2+2\lambda_\ell-24)Y_\ell.
\]
The kernel condition gives
\[
    (\lambda_\ell-4)(\lambda_\ell+6)=0.
\]
Since $\lambda_\ell\ge0$, necessarily $\lambda_\ell=4$, i.e.
$\ell=1$. Therefore
\[
    \widehat\phi
    \in
    \operatorname{span}\{\zeta_1,\ldots,\zeta_5\}.
\]
Pulling back by stereographic projection yields
\[
    v
    =
    c_0(4+x\cdot\nabla Z)
    +
    \sum_{i=1}^4c_i\partial_{x_i}Z,
\]
which proves the result.
\end{proof}

\begin{Rem}
The kernel naturally splits as
\[
    \ker\mathcal L
    =
    \underbrace{
    \operatorname{span}\{4+x\cdot\nabla Z\}
    }_{\text{scaling mode}}
    \oplus
    \underbrace{
    \operatorname{span}
    \{\partial_{x_1}Z,\ldots,\partial_{x_4}Z\}
    }_{\text{translation modes}}.
\]
Thus the subquadratic kernel is exactly five-dimensional.
\end{Rem}
\vskip 0.4cm

\subsection{Quadratic forms} $\,$  \vskip 0.2cm

For any $x_{p,j},$ $j=1,\ldots, k$ given by \eqref{def:xpj}, we introduce the following two quadratic forms \begin{equation}
\begin{split} \label{P_j}
P_j(u,v)
:=
\displaystyle\int_{\partial B_\theta(x_{p,j})}
\biggl[
&-\theta\,\Delta u\,\Delta v -\theta
\left(
\frac{\partial\Delta u}{\partial\nu}\frac{\partial v}{\partial\nu}
+
\frac{\partial\Delta v}{\partial\nu}\frac{\partial u}{\partial\nu}
\right)\\
&+
\Delta u\,
\frac{\partial}{\partial\nu}
\Bigr\langle x-x_{p,j},\nabla v\Bigr\rangle
+
\Delta v\,
\frac{\partial}{\partial\nu}
\Bigr\langle x-x_{p,j},\nabla u\Bigr\rangle
\biggl]\,d\sigma ,
\end{split}
\end{equation}
and for fixed \(i\in\{1,\dots,4\}\),
\begin{equation}\label{Q_j}
    Q_j(u,v)
:=
\displaystyle\int_{\partial B_\theta(x_{p,j})}
\biggl[
\Delta u\,\Delta v\,\nu_i +
\frac{\partial\Delta u}{\partial\nu}
\frac{\partial v}{\partial x_i}
+
\frac{\partial\Delta v}{\partial\nu}
\frac{\partial u}{\partial x_i}
-
\Delta u
\frac{\partial}{\partial\nu}
\left(\frac{\partial v}{\partial x_i}\right)
-
\Delta v
\frac{\partial}{\partial\nu}
\left(\frac{\partial u}{\partial x_i}\right)
\biggl]\,d\sigma ,
\end{equation}
where $u,v\in C^{4}(\overline{\Omega})$, $\theta>0$ is such that $B_{2\theta}(x_{p,j})\subset\Omega$ and $\nu=\big(\nu_{1},\nu_2,\nu_{3},\nu_4\big)$ stands for the outward unit normal to $\partial B_{\theta}(x_{p,j})$.

\begin{Lem}
Assume that $u$ and $v$ are biharmonic in $ B_d(x_{p,j})\backslash \{x_{p,j}\}$, then $P_{j}(u,v)$ and $Q_{j}(u,v)$ are  independent of the choice of $\theta\in (0,d]$.
\end{Lem}
\begin{proof}

For any  $\theta_{1}<\theta_{2}<d$, we denote $$\Omega'=B_{\theta_{2}}(x_{p,j})\backslash B_{\theta_{1}}(x_{p,j}).$$
Since \eqref{div}, applying integration by parts, we have
\begin{equation}\label{2.5-1}
\begin{split}
\displaystyle\int_{\Omega'}
\Delta^2u
\Bigr\langle x-x_{p,j},\nabla v\Bigr\rangle\,dx
=&
\displaystyle\int_{\partial \Omega'}
\frac{\partial\Delta u}{\partial\nu}
\Bigr\langle x-x_{p,j},\nabla v\Bigr\rangle\,d\sigma
-
\displaystyle\int_{\partial \Omega'}
\Delta u
\frac{\partial}{\partial\nu}
\Bigr\langle x-x_{p,j},\nabla v\Bigr\rangle\,d\sigma \\
&+
\displaystyle\int_{\Omega'}
\Delta u
\Bigr\langle x-x_{p,j},\nabla\Delta v\Bigr\rangle\,dx
+
2\displaystyle\int_{\Omega'}\Delta u\,\Delta v\,dx .
\end{split}
\end{equation}
Exchanging the roles of
$u$ and $v$, we similarly obtain
\begin{equation}
    \label{2.5-2}
    \begin{split}
\displaystyle\int_{\Omega'}
\Delta^2v
\Bigr\langle x-x_{p,j},\nabla u\Bigr\rangle\,dx
=&
\displaystyle\int_{\partial {\Omega'}}
\frac{\partial\Delta v}{\partial\nu}
\Bigr\langle x-x_{p,j},\nabla u\Bigr\rangle\,d\sigma
-
\displaystyle\int_{\partial {\Omega'}}
\Delta v
\frac{\partial}{\partial\nu}
\Bigr\langle x-x_{p,j},\nabla u\Bigr\rangle\,d\sigma \\
&+
\displaystyle\int_{\Omega'}
\Delta v
\Bigr\langle x-x_{p,j},\nabla\Delta u\Bigr\rangle\,dx
+
2\displaystyle\int_{\Omega'}\Delta u\,\Delta v\,dx .
\end{split}
\end{equation}
Since $u$ and $v$ are biharmonic in $ B_d(x_{p,j})\backslash \{x_{p,j}\}$,  adding \eqref{2.5-1} and \eqref{2.5-2}, one has
\begin{equation}\label{2.5-3}
\begin{aligned}
0=&
-\displaystyle\int_{\Omega'}
\biggl[ \Delta^2u
\Bigr\langle x-x_{p,j},\nabla v\Bigr\rangle
+
\Delta^2v
\Bigr\langle x-x_{p,j},\nabla u\Bigr\rangle
\biggl]\,dx \\
=&
\displaystyle\int_{\partial \Omega'}
\biggl[
-
\frac{\partial\Delta u}{\partial\nu}
\Bigr\langle x-x_{p,j},\nabla v\Bigr\rangle
-
\frac{\partial\Delta v}{\partial\nu}
\Bigr\langle x-x_{p,j},\nabla u\Bigr\rangle \\
&\qquad
+
\Delta u
\frac{\partial}{\partial\nu}
\Bigr\langle x-x_{p,j},\nabla v\Bigr\rangle
+
\Delta v
\frac{\partial}{\partial\nu}
\Bigr\langle x-x_{p,j},\nabla u\Bigr\rangle
\biggl]\,d\sigma . \\
&-\displaystyle\int_{\Omega'}
\biggl[
\Delta u
\Bigr\langle x-x_{p,j},\nabla\Delta v\Bigr\rangle
+
\Delta v
\Bigr\langle x-x_{p,j},\nabla\Delta u\Bigr\rangle
+
4\Delta u\,\Delta v
\biggl]\,dx.
\end{aligned}
\end{equation}
Since
\[
\Delta u
\Bigr\langle x-x_{p,j},\nabla\Delta v \Bigr\rangle
+
\Delta v
\Bigr\langle x-x_{p,j},\nabla\Delta u\Bigr\rangle
=
\Bigr\langle x-x_{p,j},
\nabla(\Delta u\,\Delta v)
\Bigr\rangle ,
\]
and
\[
\begin{aligned}
\Bigr\langle x-x_{p,j},
\nabla(\Delta u\,\Delta v)
\Bigr\rangle
+
4\Delta u\,\Delta v
=
\operatorname{div}
\Bigr(
(x-x_{p,j})\Delta u\,\Delta v
\Bigr),
\end{aligned}
\]
the last term involved volume integral on the right side of \eqref{2.5-3} will be reduced to the boundary term
\begin{equation}\label{2.5-33}
\displaystyle\int_{\partial \Omega'}
\Bigr\langle x-x_{p,j},\nu\Bigr\rangle
\Delta u\,\Delta v\,d\sigma .
\end{equation}
On \(\partial B_\theta(x_{p,j})\), we have $\left\langle x-x_{p,j},\nu\right\rangle=\theta$,
then $\left\langle x-x_{p,j},\nabla v\right\rangle
=
\theta\frac{\partial v}{\partial\nu}$ and $\left\langle x-x_{p,j},\nabla u\right\rangle
=
\theta\frac{\partial u}{\partial\nu}$.
Thus, combining \eqref{2.5-33}, the  identity \eqref{2.5-3} gives
\[
P_j(u,v;\theta_2)=P_j(u,v;\theta_1).
\]
Hence \(P_j(u,v)\) is independent of \(\theta\).

\vskip 0.1cm

Now we prove $Q_j$ is independent of the choice of $\theta \in \left(0,d \right]$. By integration of parts twice, we have
\[
\begin{aligned}
\displaystyle\int_{\Omega'} \Delta^2u\frac{\partial v}{\partial x_i}\,dx
=
\displaystyle\int_{\partial {\Omega'}}
\frac{\partial\Delta u}{\partial\nu}
\frac{\partial v}{\partial x_i}\,d\sigma -
\displaystyle\int_{\partial {\Omega'}}
\Delta u
\frac{\partial}{\partial\nu}
\left(
\frac{\partial v}{\partial x_i}
\right)\,d\sigma+
\displaystyle\int_{\Omega'}
\Delta u\,
\frac{\partial\Delta v}{\partial x_i}\,dx ,
\end{aligned}
\]
and
\[
\begin{aligned}
\displaystyle\int_{\Omega'} \Delta^2v\frac{\partial u}{\partial x_i}\,dx
=
\displaystyle\int_{\partial {\Omega'}}
\frac{\partial\Delta v}{\partial\nu}
\frac{\partial u}{\partial x_i}\,d\sigma -
\displaystyle\int_{\partial {\Omega'}}
\Delta v
\frac{\partial}{\partial\nu}
\left(
\frac{\partial u}{\partial x_i}
\right)\,d\sigma +
\displaystyle\int_{\Omega'}
\Delta v\,
\frac{\partial\Delta u}{\partial x_i}\,dx .
\end{aligned}
\]
Adding the two identities, the remaining volume term is
\[
\displaystyle\int_{\Omega'}
\left(
\Delta u\frac{\partial\Delta v}{\partial x_i}
+
\Delta v\frac{\partial\Delta u}{\partial x_i}
\right)\,dx.
\]
Since
\[
\Delta u\frac{\partial\Delta v}{\partial x_i}
+
\Delta v\frac{\partial\Delta u}{\partial x_i}
=
\frac{\partial}{\partial x_i}
\left(\Delta u\,\Delta v\right),
\]
we obtain
\[
\displaystyle\int_{\Omega'}
\left(
\Delta u\frac{\partial\Delta v}{\partial x_i}
+
\Delta v\frac{\partial\Delta u}{\partial x_i}
\right)\,dx
=
\displaystyle\int_{\partial {\Omega'}}\Delta u\,\Delta v\,\nu_i\,d\sigma .
\]
Therefore,
\[
\begin{aligned}
0
=
\displaystyle\int_{\partial {\Omega'}}
\Bigg[
\Delta u\,\Delta v\,\nu_i
+
\frac{\partial\Delta u}{\partial\nu}
\frac{\partial v}{\partial x_i}
+
\frac{\partial\Delta v}{\partial\nu}
\frac{\partial u}{\partial x_i}
-
\Delta u
\frac{\partial}{\partial\nu}
\left(
\frac{\partial v}{\partial x_i}
\right)
-
\Delta v
\frac{\partial}{\partial\nu}
\left(
\frac{\partial u}{\partial x_i}
\right)
\Bigg]\,d\sigma .
\end{aligned}
\]
We also have
\[
Q_j(u,v;\theta_2)=Q_j(u,v;\theta_1).
\]
Thus, \(Q_j(u,v)\) is also independent of \(\theta\).
\end{proof}

Next we give some important results on the quadratic forms $P_{j}$ and $Q_{j}$ about Green's function, which will be useful in the spectral analysis of multi-spike solutions to \eqref{1.1}
\begin{Prop}\label{lem2-1} $P_j$ and $Q_j$ are defined by \eqref{P_j} and \eqref{Q_j} respectively. $\left\{ x_{p,j} \right\}_{j=1}^k$ are given by \eqref{def:xpj}. Then
it holds
\begin{equation}\label{1-1}
P_{j}\Big(G(x_{p,s},x), G(x_{p,m},x)\Big)=
\begin{cases}
\frac{1}{8\pi^{2}} ~&\mbox{for}~s=m=j,\\[1mm]
0~&\mbox{for}~s\neq j ~{or}~m \neq j,
\end{cases}
\end{equation}
and for $h=1,2,3,4$,
 \begin{equation}\label{abb1-1}
P_{j}\Big(G(x_{p,s},x),\partial_hG(x_{p,m},x)\Big)=
\begin{cases}
-\frac{1}{2}\frac{\partial R(x_{p,j})}{\partial x_h} ~&\mbox{for}~s=m=j,\\[1mm]
 -D_{x_h}G \big(x_{p,s},x_{p,j}\big) ~&\mbox{for} ~m=j,~s\neq j,\\[1mm]
0 ~&\mbox{for}~m\neq j.
\end{cases}
\end{equation}
Moreover, we have
\begin{equation}\label{aluo1}
Q_{j}\Big(G(x_{p,m},x),G(x_{p,s},x)\Big)=
\begin{cases}
\frac{\partial R(x_{p,j})}{\partial{x_i}} ~&\mbox{for}~s=m=j,\\[1mm]
D_{x_i}G(x_{p,m},x_{p,j})
 ~&\mbox{for}~m\neq j,~s=j,\\[1mm]
D_{x_i}G(x_{p,s},x_{p,j})
~&\mbox{for}~m=j,~s\neq j,\\[1mm]
0 ~&\mbox{for}~s,m\neq j,
\end{cases}
\end{equation}
and for $h=1,2,3,4$,
\begin{equation}\label{aluo41}
Q_{j}\Big(G(x_{p,m},x),\partial_h G(x_{p,s},x)\Big)=
\begin{cases}
\frac{1}{2}\frac{\partial^2 R(x_{p,j})}{\partial{x_i}\partial {x_h}} ~&\mbox{for}~s=m=j,\\[1mm]
D_{x_i}\partial_h G(x_{p,s},x_{p,j})
 ~&\mbox{for}~m= j,~s\neq j,\\[1mm]
D^2_{x_ix_h} G(x_{p,m},x_{p,j})
~&\mbox{for}~m\neq j,~s=j,\\[1mm]
0 ~&\mbox{for}~s,m\neq j.
\end{cases}
\end{equation}
where $G(x,y)$ and $R(x)$ are the Green and Robin function, respectively (see \eqref{Greenfunction}, \eqref{Robinf}, \eqref{stts}), $\partial_iG(y,x):=\frac{\partial G(y,x)}{\partial y_i}$ and
$D_{x_i}G(y,x):=\frac{\partial G(y,x)}{\partial x_i}$.
\end{Prop}
\begin{proof}
 By \eqref{Hfuction}, we have $G(x,y)=H(x,y)+S(x,y)$ and
$
S({x_{p,j}},x)=-\frac1{8\pi^2}\ln|x-{x_{p,j}}|
$. Now by direct computation,
 we can evaluate $S$ and its derivatives  on $\partial B_\theta(x_{p,j})$. Taking
\[
x=x_{p,j}+\theta\xi \qquad \text{on } \partial B_\theta(x_{p,j}),
\]
where $|\xi|=1$,  then we have
\[
x_h-x_{p,j,h}=\theta\xi_h,\qquad \nu=\xi,\qquad d\sigma=\theta^3\,d\omega.
\]
And\
\[
D_{x_i} S (x_{p,j},x) = - \frac{x_i-x_{p,j,i}}{8 \pi^2 | x- x_{p,j} |^2},
\]
\[
\partial_hS(x_{p,j},x)
=
\frac{\partial S(y,x)}{\partial y_h}\Big|_{y=x_{p,j}}
=
\frac{x_h-x_{p,j,h}}{8\pi^2|x-x_{p,j}|^2}
=
\frac{\xi_h}{8\pi^2\theta}.
\]
Hence, on $\partial B_\theta(x_{p,j})$, it holds
\begin{equation} \label{jisuan}
\begin{gathered}
\Delta_x S(x_{p,j},x) = -\frac{1}{4\pi^2\theta^2},
\quad
\frac{\partial}{\partial\nu}\Delta_x S(x_{p,j},x) = \frac{1}{2\pi^2\theta^3},
\quad
\frac{\partial}{\partial\nu}D_{x_i}S(x_{p,j},x) = \frac{\xi_i}{8\pi^2\theta^2},
\\[4pt]
\Delta_x\bigr(\partial_h S(x_{p,j},x)\bigr) = -\frac{\xi_h}{2\pi^2\theta^3},
\qquad
\frac{\partial}{\partial\nu}\bigr(\partial_h S(x_{p,j},x)\bigr) = -\frac{\xi_h}{8\pi^2\theta^2},
\\[4pt]
\frac{\partial}{\partial\nu}\Delta_x\bigr(\partial_h S(x_{p,j},x)\bigr) = \frac{3\xi_h}{2\pi^2\theta^4},
\qquad
\bigr\langle x-x_{p,j},\nabla_x S(x_{p,j},x)\bigr\rangle = -\frac{1}{8\pi^2},
\\[4pt]
\bigr\langle x-x_{p,j},\nabla_x\bigr(\partial_h S(x_{p,j},x)\bigr)\bigr\rangle = -\frac{\xi_h}{8\pi^2\theta},
\qquad
\frac{\partial}{\partial\nu}\bigr\langle x-x_{p,j},\nabla_x\bigr(\partial_h S(x_{p,j},x)\bigr)\bigr\rangle = \frac{\xi_h}{8\pi^2\theta^2}.
\end{gathered}
\end{equation}
These basic computations will be used throughout the proof of Proposition \ref{lem2-1}.

\vskip 0.1cm

We first prove  \eqref{1-1}, and the rest complete proofs are put  in Appendix \ref{s6}.
By the bilinearity of $P_{j}(u,v)$ and \eqref{stts}, we have
\begin{equation*}
\begin{split}
P_{j}\Big(G(x_{p,j},x),G(x_{p,j},x)\Big)=&
 P_{j}\Big(S(x_{p,j},x),S(x_{p,j},x)\Big)
 +2P_{j}\Big(S(x_{p,j},x),H(x_{p,j},x)\Big)\\&
 +P_{j}\Big(H(x_{p,j},x),H(x_{p,j},x)\Big).
 \end{split}
\end{equation*}
Based on the above computations \eqref{jisuan}, it is easy to find that
\begin{equation}\label{eq:PSS-final}
\begin{aligned}
P_j\Bigr(S({x_{p,j}},x),S({x_{p,j}},x)\Bigr)
&=\displaystyle\int_{\partial B_\theta({x_{p,j}})}
\biggl[
-\theta (\Delta_x S(x_{p,j},x))^2
-2\theta \frac{\partial}{\partial\nu}\Delta_x S(x_{p,j},x)\,\frac{\partial}{\partial\nu} S(x_{p,j},x)
\biggl]\,d\sigma \\
&=\displaystyle\int_{\partial B_\theta({x_{p,j}})}
\biggl[
-\theta\frac{1}{16\pi^4\theta^4}
+\theta\frac{1}{8\pi^4\theta^4}
\biggl]\,d\sigma \\
&=\frac{1}{16\pi^4\theta^3}|\partial B_\theta({x_{p,j}})|
=\frac{1}{16\pi^4\theta^3}\cdot 2\pi^2\theta^3
=\frac1{8\pi^2}.
\end{aligned}
\end{equation}
On the other hand, Since \(H\) is smooth, the quantities \(\Delta_x H(x_{p,j},x)\), \(\frac{\partial}{\partial\nu} H(x_{p,j},x)\), \(\frac{\partial}{\partial\nu}\Delta_x H(x_{p,j},x)\) and \(\frac{\partial}{\partial\nu}\bigr\langle x-x_{p,j},\nabla_x H(x_{p,j},x)\bigr\rangle\) are uniformly bounded on \(\partial B_\theta(x_{p,j})\) for small \(\theta\). It is known by previous computations \eqref{jisuan} that \(\Big|\partial B_\theta(x_{p,j})\Big| = O(\theta^3)\), \(\Big|\Delta_x S(x_{p,j},x)\Big| = O(\theta^{-2})\), \(\Big|\frac{\partial}{\partial\nu}\Delta_x S(x_{p,j},x)\Big| = O(\theta^{-3})\) and \(\Big|\frac{\partial}{\partial\nu} S(x_{p,j},x)\Big| = O(\theta^{-1})\). Then we have
\begin{equation}\label{SH1}
\begin{aligned}
P_j \Bigr(S(x_{p,j},x),H(x_{p,j},x)\Bigr)
=&\int_{\partial B_\theta(x_{p,j})} \Big[
-\theta\Delta_x S(x_{p,j},x)\,\Delta_x H(x_{p,j},x) \\
&-\theta\,\frac{\partial}{\partial\nu}\Delta_x S(x_{p,j},x)\,\frac{\partial}{\partial\nu} H(x_{p,j},x)
-\theta\,\frac{\partial}{\partial\nu}\Delta_x H(x_{p,j},x)\,\frac{\partial}{\partial\nu} S(x_{p,j},x) \\
&+ \Delta_x S(x_{p,j},x)\,\frac{\partial}{\partial\nu}\bigr\langle x-x_{p,j},\nabla_x H(x_{p,j},x)\bigr\rangle
\Big] d\sigma  = O(\theta),
\end{aligned}
\end{equation}
and
\begin{equation}\label{SH2}
\begin{aligned}
P_j\Bigr(H(x_{p,j},x),H(x_{p,j},x)\Bigr) =& \int_{\partial B_\theta(x_{p,j})} \Big[
-\theta(\Delta_x H(x_{p,j},x))^2 - 2\theta\,\frac{\partial}{\partial\nu}\Delta_x H(x_{p,j},x)\,\frac{\partial}{\partial\nu} H(x_{p,j},x) \\
&+ 2\Delta_x H(x_{p,j},x)\,\frac{\partial}{\partial\nu}\bigr\langle x-x_{p,j},\nabla_x H(x_{p,j},x)\bigr\rangle
\Big] d\sigma = O(\theta^3).
\end{aligned}
\end{equation}
Letting $\theta\to0$ in \eqref{SH1} and \eqref{SH2}, combining \eqref{eq:PSS-final} we obtain
\[
P_j\Bigr(G(x_{p,j},x),G(x_{p,j},x)\Bigr)=\frac1{8\pi^2}.
\]
\vskip 0.1cm

Let $m\neq j$,$s=j$. By the bilinearity of $P_{j}(u,v)$, we have
\[
P_j\Bigr(G(x_{p,j},x),G(x_{p,m},x)\Bigr)=P_j\Bigr(S(x_{p,j},x),G(x_{p,m},x)\Bigr)+P_j\Bigr(H(x_{p,j},x),G(x_{p,m},x)\Bigr).
\]
Since $G(x_{p,m},x)$ and its derivatives are bounded in $B_d(x_{p,j})$, it is easy to find that $$P_j\Bigr(H(x_{p,j},x),G(x_{p,m},x)\Bigr)=O(\theta^3).$$
We now consider $P_j\Big(S(x_{p,j},x),G(x_{p,m},x)\Big)$. Since Lemma \ref{Greenfun}, we know that $G(x_{p,m},x)$ and its derivatives up to order three are bounded. Therefore by \eqref{jisuan}, we obtain
\begin{gather*}
\int_{\partial B_{\theta}(x_{p,j})}\theta\,\Delta_x S(x_{p,j},x)\,\Delta_x G(x_{p,m},x)\,d\sigma =O(\theta^2), \\
\int_{\partial B_\theta(x_{p,j})}\theta\,\frac{\partial}{\partial\nu}\Delta_x S(x_{p,j},x)\,\frac{\partial}{\partial\nu} G(x_{p,m},x)\,d\sigma =O(\theta),\\
\int_{\partial B_\theta(x_{p,j})}\theta\,\frac{\partial}{\partial\nu}\Delta_x G(x_{p,m},x)\,\frac{\partial}{\partial\nu} S(x_{p,j},x)\,d\sigma =O(\theta^3) , \\
\int_{\partial B_\theta(x_{p,j})}\Delta_x G(x_{p,m},x)\,\frac{\partial}{\partial\nu}\bigr\langle x-x_{p,j},\nabla_x S(x_{p,j},x)\bigr\rangle\,d\sigma = 0,\\
\int_{\partial B_\theta(x_{p,j})}\Delta_x S(x_{p,j},x)\,\frac{\partial}{\partial\nu}\bigr\langle x-x_{p,j},\nabla_x G(x_{p,m},x) \bigr\rangle\,d\sigma =O(\theta).
\end{gather*}
Hence $P_j\Big(S(x_{p,j},x),G(x_{p,m},x)\Big)=O(\theta)\to 0$, and consequently $P_j\Big(G(x_{p,j},x),G(x_{p,m},x)\Big)\to0$.
\vskip 0.1cm
By symmetry of $P_j$, the same argument gives $P_j\Big(G(x_{p,s},x),G(x_{p,j},x)\Big)\to0$ as $\theta \rightarrow 0$. And the case of $m\neq j$, $s \neq j$ is much easier since $G(x_{p,m},x)$,  $G(x_{p,s},x)$ and their derivatives up to order three are bounded. Therefore, we deduce that $$P_{j}\Big(G(x_{p,s},x), G(x_{p,m},x)\Big)=0\,\,~\mbox{for}~s\neq j~\mbox{or}~m\neq j.$$
\end{proof}

\subsection{Pohozaev-type identities} $\,$  \vskip 0.2cm

We next give some  local Pohozaev identities in balls near the concentrated points. Actually, these identities still hold in general bounded local domain.

\begin{Lem}
Let $u\in C^4(\Omega)$ be a solution of \eqref{1.1},  $x_{p,j}$, $j=1,\ldots, k$ be the points defined in \eqref{def:xpj} and let $\theta>0$ be such that $B_{2\theta}(x_{p,j})\subset\Omega$. Then
\begin{equation}\label{aclp-1}
Q_{j}(u,u)= \frac{2}{p+1}\displaystyle\int_{\partial B_{\theta}(x_{p,j})}  (u^+)^{p+1} \nu_id\sigma
\end{equation}
and
\begin{equation} \label{aclp-10}
P_j(u,u)
=
\frac{8}{p+1}
\displaystyle\int_{B_\theta(x_{p,j})}
(u^+)^{p+1}\,dx
-
\frac{2\theta}{p+1}
\displaystyle\int_{\partial B_\theta(x_{p,j})}
(u^+)^{p+1}\,d\sigma.
\end{equation}
where $\nu=\big(\nu_{1},\nu_2,\nu_{3},\nu_4\big)$ is the outward unit normal of $\partial B_{\theta}(x_{p,j})$.
\end{Lem}
\begin{proof}
Firstly, we prove \eqref{aclp-1}. Multiplying \eqref{1.1} by $\frac{\partial u}{\partial x_i}$ and integrating on $B_{\theta}(x_{p,j})$, we obtain
\begin{equation*}
  \displaystyle\int_{B_\theta(x_{p,j})}
\Delta^2u\frac{\partial u}{\partial x_i}\,dx
=
\displaystyle\int_{B_\theta(x_{p,j})}
(u^+)^p\frac{\partial u}{\partial x_i}\,dx
=
\frac{1}{p+1}
\displaystyle\int_{B_\theta(x_{p,j})}
\frac{\partial}{\partial x_i}
\Bigr((u^+)^{p+1}\Bigr)\,dx .
\end{equation*}
Using integration by parts and the divergence theorem, it follows that
\begin{equation}\label{sandu3}
\displaystyle\int_{B_\theta(x_{p,j})}
\frac{\partial}{\partial x_i}
\Bigr((u^+)^{p+1}\Bigr)\,dx
=
\displaystyle\int_{\partial B_\theta(x_{p,j})}
(u^+)^{p+1}\nu_i\,d\sigma ,
\end{equation}
and
\begin{equation}\label{sandu1}
    \displaystyle\int_{B_\theta(x_{p,j})}
\Delta^2u\frac{\partial u}{\partial x_i}\,dx
=
\displaystyle\int_{\partial B_\theta(x_{p,j})}
\frac{\partial\Delta u}{\partial\nu}
\frac{\partial u}{\partial x_i}\,d\sigma -
\displaystyle\int_{\partial B_\theta(x_{p,j})}
\Delta u
\frac{\partial}{\partial\nu}
\left(\frac{\partial u}{\partial x_i}\right)\,d\sigma  +
\displaystyle\int_{B_\theta(x_{p,j})}
\Delta u
\frac{\partial\Delta u}{\partial x_i}\,dx ,
\end{equation}
with
\begin{equation}\label{sandu2}
    \displaystyle\int_{B_\theta(x_{p,j})}
\Delta u
\frac{\partial\Delta u}{\partial x_i}\,dx
=
\frac12
\displaystyle\int_{\partial B_\theta(x_{p,j})}
(\Delta u)^2\nu_i\,d\sigma .
\end{equation}
Hence by \eqref{sandu3},\eqref{sandu1} and \eqref{sandu2}, we complete the proof of \eqref{aclp-1}.

\par Similarly, we prove the second identity \eqref{aclp-10}. Multiplying \eqref{1.1} by $\left\langle x-x_{p,j},\nabla u\right\rangle$ and integrating on $B_{\theta}(x_{p,j})$,  we have
\begin{equation*}
    \displaystyle\displaystyle\int_{B_{\theta}(x_{p,j})} \Delta^2 u \Bigr\langle x-x_{p,j},\nabla u\Bigr\rangle d x
    = \displaystyle\displaystyle\int_{B_{\theta}(x_{p,j})} (u^+)^p \Bigr\langle x-x_{p,j},\nabla u\Bigr\rangle d x.
\end{equation*}
 It follows from integration by parts that
\begin{equation}\label{psandu1}
    \begin{split}
        \displaystyle\int_{B_\theta(x_{p,j})}
\Delta^2u
\Bigr\langle x-x_{p,j},\nabla u\Bigr\rangle\,dx
&=
\displaystyle\int_{\partial B_\theta(x_{p,j})}
\frac{\partial\Delta u}{\partial\nu}
\Bigr\langle x-x_{p,j},\nabla u\Bigr\rangle\,d\sigma
-
\displaystyle\int_{\partial B_\theta(x_{p,j})}
\Delta u
\frac{\partial}{\partial\nu}
\Bigr\langle x-x_{p,j},\nabla u\Bigr\rangle\,d\sigma  \\
&\quad
+
\displaystyle\int_{B_\theta(x_{p,j})}
\Delta u\,
\Delta
\Bigr\langle x-x_{p,j},\nabla u\Bigr\rangle\,dx .
    \end{split}
\end{equation}
For the last term on the right-hand side of \eqref{psandu1}, since
\begin{equation}\label{div}
\Delta
\Bigr\langle x-x_{p,j},\nabla u\Bigr\rangle
=
\Bigr\langle x-x_{p,j},\nabla\Delta u\Bigr\rangle
+
2\Delta u ,
\end{equation}
we get
\begin{equation}\label{psandu2}
        \displaystyle\int_{B_\theta(x_{p,j})}
\Delta u\,
\Delta
\Bigr\langle x-x_{p,j},\nabla u\Bigr\rangle\,dx
=
\displaystyle\int_{B_\theta(x_{p,j})}
\Delta u
\Bigr\langle x-x_{p,j},\nabla\Delta u\Bigr\rangle\,dx
+
2\displaystyle\int_{B_\theta(x_{p,j})}
(\Delta u)^2\,dx .
\end{equation}
Furthermore,
\begin{equation}\label{psandu3}
    \begin{split}
    \displaystyle\int_{B_\theta(x_{p,j})}
\Delta u
\Bigr\langle x-x_{p,j},\nabla\Delta u\Bigr\rangle\,dx
&=
\frac12
\displaystyle\int_{B_\theta(x_{p,j})}
\Bigr\langle x-x_{p,j},
\nabla\big((\Delta u)^2\big)
\Bigr\rangle\,dx  \\
&=
\frac12
\displaystyle\int_{\partial B_\theta(x_{p,j})}
(\Delta u)^2
\Bigr\langle x-x_{p,j},\nu\Bigr\rangle\,d\sigma
-
2\displaystyle\int_{B_\theta(x_{p,j})}
(\Delta u)^2\,dx .
 \end{split}
\end{equation}
From \eqref{psandu2} and \eqref{psandu3}, we have
\begin{equation}\label{xiang1}
    \displaystyle\int_{B_\theta(x_{p,j})}
\Delta u\,
\Delta
\Bigr\langle x-x_{p,j},\nabla u\Bigr\rangle\,dx  =
\frac12
\displaystyle\int_{\partial B_\theta(x_{p,j})}
(\Delta u)^2
\Bigr\langle x-x_{p,j},\nu\Bigr\rangle\,d\sigma .
\end{equation}
Hence from \eqref{psandu1} and \eqref{xiang1}, it follows that\begin{equation}\label{xiang2}
    \begin{split}
        \displaystyle\int_{B_\theta(x_{p,j})}
\Delta^2u
\Bigr\langle x-x_{p,j},\nabla u\Bigr\rangle\,dx
&=
\displaystyle\int_{\partial B_\theta(x_{p,j})}
\frac{\partial\Delta u}{\partial\nu}
\Bigr\langle x-x_{p,j},\nabla u\Bigr\rangle\,d\sigma
-
\displaystyle\int_{\partial B_\theta(x_{p,j})}
\Delta u
\frac{\partial}{\partial\nu}
\Bigr\langle x-x_{p,j},\nabla u\Bigr\rangle\,d\sigma  \\
&\quad
+
\frac12
\displaystyle\int_{\partial B_\theta(x_{p,j})}
(\Delta u)^2
\Bigr\langle x-x_{p,j},\nu\Bigr\rangle\,d\sigma .
    \end{split}
\end{equation}
On the other hand, on \(\partial B_\theta(x_{p,j})\), one has $\left\langle x-x_{p,j},\nu\right\rangle=\theta$ and $\left\langle x-x_{p,j},\nabla u\right\rangle
=
\theta\frac{\partial u}{\partial\nu}$. Thus, it holds
\begin{equation}\label{sandu4}
\begin{aligned}
    \displaystyle\int_{B_\theta(x_{p,j})}
(u^+)^p
\Bigr\langle x-x_{p,j},\nabla u\Bigr\rangle\,dx
&=\frac{1}{p+1}
\displaystyle\int_{B_\theta(x_{p,j})}
\Bigr\langle x-x_{p,j},
\nabla\left((u^+)^{p+1}\right)
\Bigr\rangle\,dx  \\
&=\frac{\theta}{p+1}
\displaystyle\int_{\partial B_\theta(x_{p,j})}
(u^+)^{p+1}\,d\sigma
- \frac{4}{p+1} \displaystyle\int_{B_\theta(x_{p,j})}
(u^+)^{p+1}\,dx .
\end{aligned}
\end{equation}
Consequently, by \eqref{xiang2} and \eqref{sandu4}, \eqref{aclp-10} is  proved.
\end{proof}

\begin{Lem}
Let $u\in C^4(\Omega)$ be a solution of \eqref{1.1},  $\xi\in C^{4}(\Omega)$ be a solution of ${\Delta}^2 \xi=\lambda p(u^+)^{p-1}\xi$ in $\Omega$ and  $x_{p,j}$, $j=1,\ldots, k$ be the points defined in \eqref{def:xpj} and let $\theta>0$ be such that $B_{2\theta}(x_{p,j})\subset\Omega$. Then
\begin{equation}\label{MI:eq:Q-general}
    Q_{j}(\xi,u)=\int_{\partial B_{\theta}(x_{p,j})}(u^+)^p\xi\nu_i \,d\sigma + (\lambda-1)p\int_{B_{\theta}(x_{p,j})}(u^+)^{p-1}\xi\,\frac{\partial u }{\partial {x_i}} \,dx,
\end{equation}
\begin{equation}\label{MI:eq:P-general}
    P_j(\xi,u)
=
4\int_{B_{\theta}(x_{p,j})}(u^+)^p\xi \,dx
-
{\theta}\int_{\partial B_{\theta}(x_{p,j})}(u^+)^p\xi \,d\sigma
-
(\lambda-1)p\int_{B_{\theta}(x_{p,j})}
(u^+)^{p-1}\xi\,(x-x_{p,j})\cdot\nabla u \,dx.
\end{equation}

\end{Lem}

\begin{proof}
By the definition \eqref{Q_j} of $Q_j$, it follows that \begin{equation}\label{QJ1}
        Q_j(\xi,u)
=
\displaystyle\int_{\partial B_\theta(x_{p,j})}
\Bigg[
\Delta\xi\,\Delta u\,\nu_i
+
\frac{\partial\Delta\xi}{\partial\nu}
\frac{\partial u}{\partial x_i}
+
\frac{\partial\Delta u}{\partial\nu}
\frac{\partial \xi}{\partial x_i}
-
\Delta\xi
\frac{\partial}{\partial\nu}
\left(\frac{\partial u}{\partial x_i}\right)
-
\Delta u
\frac{\partial}{\partial\nu}
\left(\frac{\partial \xi}{\partial x_i}\right)
\Bigg]\,d\sigma .
\end{equation}
And  we can compute directly that
\[
\begin{aligned}
\displaystyle\int_{B_\theta(x_{p,j})}\Delta^2\xi
\frac{\partial u}{\partial x_i}\,dx
=
\displaystyle\int_{\partial B_\theta(x_{p,j})}
\frac{\partial\Delta\xi}{\partial\nu}
\frac{\partial u}{\partial x_i}\,d\sigma -
\displaystyle\int_{\partial B_\theta(x_{p,j})}
\Delta\xi
\frac{\partial}{\partial\nu}
\left(\frac{\partial u}{\partial x_i}\right)d\sigma +
\displaystyle\int_{B_\theta(x_{p,j})}
\Delta\xi\,
\frac{\partial\Delta u}{\partial x_i}\,dx ,
\end{aligned}
\]
\[
\begin{aligned}
\displaystyle\int_{B_\theta(x_{p,j})}\Delta^2u
\frac{\partial \xi}{\partial x_i}\,dx
=
\displaystyle\int_{\partial B_\theta(x_{p,j})}
\frac{\partial\Delta u}{\partial\nu}
\frac{\partial \xi}{\partial x_i}\,d\sigma -
\displaystyle\int_{\partial B_\theta(x_{p,j})}
\Delta u
\frac{\partial}{\partial\nu}
\left(\frac{\partial \xi}{\partial x_i}\right)d\sigma+
\displaystyle\int_{B_\theta(x_{p,j})}
\Delta u\,
\frac{\partial\Delta \xi}{\partial x_i}\,dx .
\end{aligned}
\]
Adding the two identities, and using
\[
\Delta\xi\frac{\partial\Delta u}{\partial x_i}
+
\Delta u\frac{\partial\Delta \xi}{\partial x_i}
=
\frac{\partial}{\partial x_i}(\Delta u\,\Delta\xi),
\]
we obtain \begin{equation}\label{zhankai1}
    \begin{split}
        \displaystyle\int_{B_\theta(x_{p,j})}
\left(
\Delta^2\xi
\frac{\partial u}{\partial x_i}
+
\Delta^2u
\frac{\partial \xi}{\partial x_i}
\right)dx
=&
\displaystyle\int_{\partial B_\theta(x_{p,j})}
\Bigg[
\Delta\xi\,\Delta u\,\nu_i
+
\frac{\partial\Delta\xi}{\partial\nu}
\frac{\partial u}{\partial x_i}
+
\frac{\partial\Delta u}{\partial\nu}
\frac{\partial \xi}{\partial x_i} \\
&\qquad
-
\Delta\xi
\frac{\partial}{\partial\nu}
\left(\frac{\partial u}{\partial x_i}\right)
-
\Delta u
\frac{\partial}{\partial\nu}
\left(\frac{\partial \xi}{\partial x_i}\right)
\Bigg]\,d\sigma .
    \end{split}
\end{equation}
Hence from \eqref{QJ1} and \eqref{zhankai1}, we get
\begin{equation}\label{QJ2}
    \begin{split}
        Q_j(\xi,u)
=
\displaystyle\int_{B_\theta(x_{p,j})}
\left(
\Delta^2\xi
\frac{\partial u}{\partial x_i}
+
\Delta^2u
\frac{\partial \xi}{\partial x_i}
\right)dx .
    \end{split}
\end{equation}
Substituting $\Delta^2u=(u^+)^p$ and $\Delta^2\xi=\lambda p (u^+)^{p-1}\xi$ into \eqref{QJ2}, we have
\begin{align*}
    Q_j(\xi,u)
=&
\displaystyle\int_{B_\theta(x_{p,j})}
\lambda p (u^+)^{p-1}\xi
\frac{\partial u}{\partial x_i}
+
(u^+)^p
\frac{\partial \xi}{\partial x_i} dx \\
=&
\displaystyle\int_{B_\theta(x_{p,j})}
\frac{\partial}{\partial x_i}((u^+)^p\xi)\,dx + \left( \lambda - 1 \right) p \int_{B_{\theta}(x_{p,j}) } (u^+)^{p-1} \xi \frac{\partial u }{\partial x_i} \,dx \\
=&
\displaystyle\int_{\partial B_\theta(x_{p,j})}
(u^+)^p\xi\nu_i\,d\sigma + \left( \lambda - 1 \right) p \int_{B_{\theta}(x_{p,j}) } (u^+)^{p-1} \xi \frac{\partial u }{\partial x_i} \,dx.
\end{align*}
Next, we prove the identity \eqref{MI:eq:P-general}. It follows from the definition \eqref{P_j} for \(P_j(\xi,u)\) that
\begin{equation*}
    \begin{split}
        P_j(\xi,u)
=
\displaystyle\int_{\partial B_\theta(x_{p,j})}
\Bigg[
&-\theta\Delta\xi\,\Delta u -\theta
\left(
\frac{\partial\Delta\xi}{\partial\nu}
\frac{\partial u}{\partial\nu}
+
\frac{\partial\Delta u}{\partial\nu}
\frac{\partial \xi}{\partial\nu}
\right) \\
&+
\Delta\xi
\frac{\partial}{\partial\nu}
\Bigr\langle x-x_{p,j},\nabla u\Bigr\rangle+
\Delta u
\frac{\partial}{\partial\nu}
\Bigr\langle x-x_{p,j},\nabla \xi \Bigr\rangle
\Bigg]\,d\sigma .
    \end{split}
\end{equation*}
Just like \eqref{2.5-1} and \eqref{2.5-2} showed, after integrating by parts twice,
we obtain
\[
\begin{aligned}
\displaystyle\int_{B_\theta(x_{p,j})}
\Delta^2\xi\,
\Bigr\langle x-x_{p,j},\nabla u\Bigr\rangle\,dx
=&
\displaystyle\int_{\partial B_\theta(x_{p,j})}
\frac{\partial\Delta\xi}{\partial\nu}
\Bigr\langle x-x_{p,j},\nabla u\Bigr\rangle\,d\sigma
-
\displaystyle\int_{\partial B_\theta(x_{p,j})}
\Delta\xi
\frac{\partial}{\partial\nu}
\Bigr\langle x-x_{p,j},\nabla u\Bigr\rangle\,d\sigma \\
&+
\displaystyle\int_{B_\theta(x_{p,j})}
\Delta\xi
\Bigr\langle x-x_{p,j},\nabla\Delta u\Bigr\rangle\,dx
+
2\displaystyle\int_{B_\theta(x_{p,j})}
\Delta\xi\,\Delta u\,dx ,
\end{aligned}
\]
and
\[
\begin{aligned}
\displaystyle\int_{B_\theta(x_{p,j})}
\Delta^2u\,
\Bigr\langle x-x_{p,j},\nabla \xi\Bigr\rangle\,dx
=&
\displaystyle\int_{\partial B_\theta(x_{p,j})}
\frac{\partial\Delta u}{\partial\nu}
\Bigr\langle x-x_{p,j},\nabla \xi\Bigr\rangle\,d\sigma
-
\displaystyle\int_{\partial B_\theta(x_{p,j})}
\Delta u
\frac{\partial}{\partial\nu}
\Bigr\langle x-x_{p,j},\nabla \xi\Bigr\rangle\,d\sigma \\
&+
\displaystyle\int_{B_\theta(x_{p,j})}
\Delta u
\Bigr\langle x-x_{p,j},\nabla\Delta \xi\Bigr\rangle\,dx
+
2\displaystyle\int_{B_\theta(x_{p,j})}
\Delta u\,\Delta\xi\,dx .
\end{aligned}
\]
We add the above two identities. And similar to dealing with \eqref{2.5-3}, we will have the following result  by divergence theorem that
\begin{equation*}
\begin{split}
&\int_{B_\theta(x_{p,j})}
\left[
\Delta^2\xi
\Bigr\langle x-x_{p,j},\nabla u\Bigr\rangle
+
\Delta^2u
\Bigr\langle x-x_{p,j},\nabla \xi\Bigr\rangle
\right]dx \\
&=
\int_{\partial B_\theta(x_{p,j})}
\Bigr\langle x-x_{p,j},\nu\Bigr\rangle
\Delta u\,\Delta\xi\,d\sigma +
\int_{\partial B_\theta(x_{p,j})}
\left[
\frac{\partial\Delta\xi}{\partial\nu}
\Bigr\langle x-x_{p,j},\nabla u\Bigr\rangle
+
\frac{\partial\Delta u}{\partial\nu}
\Bigr\langle x-x_{p,j},\nabla \xi\Bigr\rangle
\right.  
\\  
&\left.  
\qquad \qquad-
\Delta\xi
\frac{\partial}{\partial\nu}
\Bigr\langle x-x_{p,j},\nabla u\Bigr\rangle
-
\Delta u
\frac{\partial}{\partial\nu}
\Bigr\langle x-x_{p,j},\nabla \xi\Bigr\rangle
\right] d\sigma . 
  \end{split}
\end{equation*}
It follows that
\begin{equation}\label{RHS}
    \begin{split}
       &-
\displaystyle\int_{B_\theta(x_{p,j})}
\left[
\Delta^2\xi
\Bigr\langle x-x_{p,j},\nabla u\Bigr\rangle
+
\Delta^2u
\Bigr\langle x-x_{p,j},\nabla \xi\Bigr\rangle
\right]dx \\
=&
\displaystyle\int_{\partial B_\theta(x_{p,j})}
\Bigg[
-\theta\Delta\xi\,\Delta u
-\theta
\left(
\frac{\partial\Delta\xi}{\partial\nu}
\frac{\partial u}{\partial\nu}
+
\frac{\partial\Delta u}{\partial\nu}
\frac{\partial \xi}{\partial\nu}
\right)  +
\Delta\xi
\frac{\partial}{\partial\nu}
\Bigr\langle x-x_{p,j},\nabla u\Bigr\rangle
+
\Delta u
\frac{\partial}{\partial\nu}
\Bigr\langle x-x_{p,j},\nabla \xi\Bigr\rangle
\Bigg]d\sigma .
    \end{split}
\end{equation}
The right-hand side of \eqref{RHS} is exactly \(P_j(\xi,u)\). Hence, \[
P_j(\xi,u)
=
-
\displaystyle\int_{B_\theta(x_{p,j})}
\left[
\Delta^2\xi
\Bigr\langle x-x_{p,j},\nabla u\Bigr\rangle
+
\Delta^2u
\Bigr\langle x-x_{p,j},\nabla \xi\Bigr\rangle
\right]dx .
\]
By $\Delta^2u=(u^+)^p$ and $\Delta^2\xi=\lambda p (u^+)^{p-1}\xi$,
we obtain
\begin{equation}\label{PJJJJ}
    \begin{split}
        P_j(\xi,u)
=&
-
\displaystyle\int_{B_\theta(x_{p,j})}
\left[
\lambda p (u^+)^{p-1}\xi
\Bigr\langle x-x_{p,j},\nabla u\Bigr\rangle
+
(u^+)^p
\Bigr\langle x-x_{p,j},\nabla \xi\Bigr\rangle
\right] \,dx \\
=&
-
\displaystyle\int_{B_\theta(x_{p,j})}
\langle x-x_{p,j},\nabla\left((u^+)^p\xi\right)\rangle \,dx - (\lambda - 1) p \int_{B_{\theta} (x_{p,j})} (u^+)^{p-1} \xi (x-x_{p,j}) \cdot \nabla u \,dx.
    \end{split}
\end{equation}
Since \begin{equation}\label{DIV}
    \begin{split}
        \displaystyle\int_{B_\theta(x_{p,j})}
\langle x-x_{p,j},\nabla((u^+)^p\xi)\rangle dx
=&
\displaystyle\int_{\partial B_\theta(x_{p,j})}
\Bigr\langle x-x_{p,j},\nu\Bigr\rangle (u^+)^p\xi\,d\sigma
-
4\displaystyle\int_{B_\theta(x_{p,j})}(u^+)^p\xi\,dx \\
=&
\theta\displaystyle\int_{\partial B_\theta(x_{p,j})}(u^+)^p\xi\,d\sigma
-
4\displaystyle\int_{B_\theta(x_{p,j})}(u^+)^p\xi\,dx .
    \end{split}
\end{equation}
Hence from \eqref{PJJJJ} and \eqref{DIV}, we have \[
P_j(\xi,u)
=
4\displaystyle\int_{B_\theta(x_{p,j})}(u^+)^p\xi\,dx
-
\theta\displaystyle\int_{\partial B_\theta(x_{p,j})}(u^+)^p\xi\,d\sigma - (\lambda - 1) p \int_{B_{\theta} (x_{p,j})} (u^+)^{p-1} \xi (x-x_{p,j}) \cdot \nabla u \,dx.
\]
The proof is complete.
\end{proof}

\begin{Lem}
    Let $u\in C^4(\Omega)$ be a solution of \eqref{1.1},  $\xi\in C^{4}(\Omega)$ be a solution of ${\Delta}^2 \xi=\lambda p(u^+)^{p-1}\xi$ in $\Omega$ and  $x_{p,j}$, $j=1,\ldots, k$ be the points defined in \eqref{def:xpj}. For each \(j=1,\ldots,k\), define the fourth-order Green concomitant
\[
\mathfrak B_j(f,g)
:=
\int_{\partial B_d(x_{p,j})}
\left[
f\,\partial_\nu\Delta g
-g\,\partial_\nu\Delta f
-\partial_\nu f\,\Delta g
+\partial_\nu g\,\Delta f
\right]\,d\sigma.
\]
Then
\begin{equation}\label{MI:eq:local-I}
    (\lambda p-1)
\int_{B_d(x_{p,j})}
(u^+)^p \xi \,dx
=
\mathfrak B_j(u,\xi).
\end{equation}
\end{Lem}
\begin{proof}
    By Green's formula for the biharmonic operator, for sufficiently regular functions \(f\) and \(g\),
\begin{equation}\label{eq:00000001}
    \int_{B_d(x_{p,j})}
\left(
f\Delta^2g-g\Delta^2f
\right)\,dx
=
\mathfrak B_j(f,g).
\end{equation}
We now take
\[
f=u,\qquad g=\xi.
\]
Since
\[
\Delta^2u=(u^+)^p, \qquad \Delta^2 \xi
=
\lambda p (u^+)^{p-1}\xi.
\]
Therefore, by \eqref{eq:00000001}, it follows that
\begin{equation*}
\mathfrak B_j(u,\xi)
=
\int_{B_d(x_{p,j})}
\left(
u\Delta^2\xi
-
\xi \Delta^2u
\right)\,dx
=
\int_{B_d(x_{p,j})}
\left[
\lambda p\,
u(u^+)^{p-1}\xi
-
(u^+)^p \xi
\right]dx.
\end{equation*}
Note that
\[
u(u^+)^{p-1}
=
(u^+)^p
\]
pointwise in \(\Omega\), so we obtain \eqref{MI:eq:local-I}.
\end{proof}

\vskip 0.5cm

\section{The asymptotic behavior of the multi-spike solutions}  \label{section3}

In this section, we mainly refine the estimates stated in Theorem $A$.
We first derive an explicit second-order expansion of $Z_{p,j}$ by blow-up analysis and Green's representation, and we can see Proposition~\ref{prop3-2} and Proposition~\ref{key-2} for more details. Then we  give a complete proof of Theorem~\ref{th1}. Finally, we improve \eqref{11-14-03N} and collect some refined estimates, which will be used in the proof of Corollary~\ref{th1.1}.

We define
\begin{equation}\label{dsa}
\eta_{p,j}:=p\Bigr(Z_{p,j}-Z\Bigr),
\end{equation}
where $Z_{p,j}$ is the rescaled function given by \eqref{defwpj}, and $Z$ denotes its limit function introduced in \eqref{Z(x)}.
To obtain a sharper estimate for $Z_{p,j}$, we first analyze $\eta_{p,j}$. Referring to Lemma~5.2 of \cite{santra-wei}, we have the following lemma.
\begin{Lem}\label{llma}
For any small fixed $\sigma,d>0$, there exist $R_\sigma>1$ and $p_\sigma> 1$ such that
\begin{equation*}
Z_{p,j}(y) \leq \big(8-\sigma\big)\ln \frac{1}{|y|}+C_\sigma,~~\mbox{for}~j=1,\cdots,k,
\end{equation*}
for some $C_\sigma>0$, provided $R_\sigma\leq |y|\leq \frac{d}{\e_{p,j}}$ and $p\geq p_\sigma$.
And we also have
\begin{equation}\label{abs}
\lim_{p\to +\infty}\displaystyle\int_{B_{d/\varepsilon_{p,j}}(0)}\Big[\Big(1+\frac{ Z_{p,j}(z)}{p}
  \Big)^+\Big]^p\,dz
   = 64\pi^2.
\end{equation}
\end{Lem}


We shall prove the following estimate:
\begin{Prop}\label{key-1}
Let $\eta _{p,j}$ be given by \eqref{dsa}. Then, for any fixed sufficiently small $d_0>0$ and fixed $\tau\in (0,1)$, there exists $C>0$ such that
\begin{equation}\label{lpy1}
| \eta _{p,j}(x)|\leq C(1+|x| )^\tau~\mbox{in}~ B_{\frac{d_0}{\e_{p,j}}}(0).
\end{equation}
\end{Prop}

\begin{proof}
By Green's representation,
\begin{equation}\label{z}
Z(y)=\int_{\mathbb R^4}K(y,z)e^{Z(z)}\,dz,
\end{equation}
and
\begin{equation}\label{u_p}
u_p(x)=\int_\Omega G(x,\xi)(u_p^+)^p(\xi)\,d\xi.
\end{equation}
Taking $x=x_{p,j}+\varepsilon_{p,j}y$ and
$\xi=x_{p,j}+\varepsilon_{p,j}z$ in \eqref{u_p}, we obtain
\begin{equation}\label{zp}
Z_{p,j}(y)
=
\int_{\Omega_{p,j}}K(y,z)F_{p,j}(z)\,dz+R_p(y),
\end{equation}
where
\begin{equation}\label{Fpj}
F_{p,j}(z):=
\left[\left(1+\frac{Z_{p,j}(z)}p\right)^+\right]^p,
\qquad
K(y,z):=-\frac1{8\pi^2}\ln\frac{|y-z|}{|z|},
\end{equation}
and
\begin{equation}\label{RRP}
R_p(y):=
\int_{\Omega_{p,j}}
\left[
H(x_{p,j}+\varepsilon_{p,j}y,x_{p,j}+\varepsilon_{p,j}z)
-H(x_{p,j},x_{p,j}+\varepsilon_{p,j}z)
\right]F_{p,j}(z)\,dz.
\end{equation}

We first record two estimates. Since $u_p(x_{p,j})\to\sqrt e$ and
\[
\varepsilon_{p,j}
=\left(pu_p(x_{p,j})^{p-1}\right)^{-1/4},
\]
we have
\begin{equation}\label{rough-epsilon-rate}
\frac1p\ln\frac1{\varepsilon_{p,j}}\longrightarrow\frac18.
\end{equation}
In particular, for every fixed $N>0$ and $\gamma>0$,
\begin{equation}\label{epsilon-super-polynomial}
p^N\varepsilon_{p,j}^{\gamma}\longrightarrow0.
\end{equation}
Moreover, by H\"older's inequality and \eqref{1.2},
\[
p\int_\Omega (u_p^+)^p\,dx\le C.
\]
After scaling, this gives the global mass bound
\begin{equation}\label{global-rescaled-mass}
\int_{\Omega_{p,j}}F_{p,j}(z)\,dz
=
\frac{p}{u_p(x_{p,j})}\int_\Omega (u_p^+)^p\,dx
\le C.
\end{equation}
We shall also use the elementary estimate: for every $q>4$ and $R\ge2$,
\begin{equation}\label{log-kernel-tail}
\int_{|z|\ge R}
\left|\ln\frac{|y-z|}{|z|}\right|(1+|z|)^{-q}\,dz
\le
CR^{4-q}\bigl(1+\ln(2+|y|)+\ln R\bigr).
\end{equation}
It follows by splitting the integral into $\{|y-z|\le1\}$ and
$\{|y-z|>1\}$.

Fix $d>0$ sufficiently small so that the balls
$B_{4d}(x_{\infty,s})$, $s=1,\ldots,k$, are mutually disjoint and
compactly contained in $\Omega$, and choose $d_0>0$ smaller if necessary.
Since $x_{p,s}\to x_{\infty,s}$, the same separation holds for $x_{p,s}$
when $p$ is large. By \eqref{global-rescaled-mass} and the smoothness of
$H$ in the first variable,
\[
|R_p(y)|\le C\varepsilon_{p,j}|y|,
\qquad
|y|\le\frac{d_0}{\varepsilon_{p,j}},
\]
and hence, by \eqref{epsilon-super-polynomial},
\begin{equation*}
pR_p(y)=o\bigl((1+|y|)^\tau\bigr)
\end{equation*}
uniformly in $B_{d_0/\varepsilon_{p,j}}(0)$.

Combining \eqref{zp} and \eqref{z}, we have
\begin{equation*}
\begin{aligned}
Z_{p,j}(y)-Z(y)
={}&
\int_{\Omega_{p,j}}K(y,z)\bigl(F_{p,j}(z)-e^{Z(z)}\bigr)\,dz\\
&-\int_{\mathbb R^4\setminus\Omega_{p,j}}K(y,z)e^{Z(z)}\,dz
+R_p(y).
\end{aligned}
\end{equation*}
For $z\notin\Omega_{p,j}$ and $|y|\le d_0/\varepsilon_{p,j}$,
we have $|z|\ge 3d_0/\varepsilon_{p,j}$ and
\begin{equation}\label{K_y}
|K(y,z)|\le C\frac{|y|}{|z|}.
\end{equation}
Since $e^Z\le C(1+|z|)^{-8}$, it follows that
\begin{equation*}
p\int_{\mathbb R^4\setminus\Omega_{p,j}}
|K(y,z)|e^{Z(z)}\,dz
=o\bigl((1+|y|)^\tau\bigr).
\end{equation*}
Thus
\begin{equation}\label{eta-pre-tail}
\eta_{p,j}(y)
=
p\int_{\Omega_{p,j}}K(y,z)\bigl(F_{p,j}(z)-e^{Z(z)}\bigr)\,dz
+o\bigl((1+|y|)^\tau\bigr).
\end{equation}

We now estimate the far field taking into account all the bubbles. Choose
$0<\beta<1/16$. By \eqref{rough-epsilon-rate},
$e^{\beta p}\varepsilon_{p,j}\to0$; in particular,
$B_{e^{\beta p}}(0)\subset B_{d/\varepsilon_{p,j}}(0)$ and this ball
contains no rescaled concentration point other than the origin for large $p$.
For $s\ne j$, set
\[
\mathcal B_{p,j}^{(s)}
:=\frac{B_d(x_{p,s})-x_{p,j}}{\varepsilon_{p,j}},
\qquad
\mathcal E_{p,j}
:=\frac{\Omega\setminus\bigcup_{s=1}^kB_d(x_{p,s})-x_{p,j}}
{\varepsilon_{p,j}}.
\]
Then
\[
\Omega_{p,j}\setminus B_{e^{\beta p}}(0)
=
\left(B_{d/\varepsilon_{p,j}}(0)\setminus B_{e^{\beta p}}(0)\right)
\cup\bigcup_{s\ne j}\mathcal B_{p,j}^{(s)}
\cup\mathcal E_{p,j}.
\]
On the first region, Lemma~\ref{llma}, \eqref{log-kernel-tail}, and
$e^Z\le C(1+|z|)^{-8}$ give
\begin{equation*}
p\int_{B_{d/\varepsilon_{p,j}}\setminus B_{e^{\beta p}}}
|K(y,z)|\bigl|F_{p,j}(z)-e^{Z(z)}\bigr|\,dz=o(1).
\end{equation*}
On $\mathcal B_{p,j}^{(s)}$, $s\ne j$, the separation of the peaks yields
$|z|\ge c/\varepsilon_{p,j}$ and therefore
\[
|K(y,z)|\le C\varepsilon_{p,j}|y|.
\]
Together with \eqref{global-rescaled-mass},
\begin{equation*}
\frac{p\displaystyle\sum_{s\ne j}
\int_{\mathcal B_{p,j}^{(s)}}|K(y,z)|F_{p,j}(z)\,dz}
{(1+|y|)^\tau}
\le Cp\varepsilon_{p,j}^{\tau}\longrightarrow0.
\end{equation*}
The corresponding $e^Z$ contribution is smaller. Finally, by
\eqref{11-14-03N},
\[
|u_p(x)|\le\frac{C_d}{p}
\quad\text{in }\Omega\setminus\bigcup_{s=1}^kB_d(x_{p,s}),
\]
so that
\[
\int_{\mathcal E_{p,j}}F_{p,j}(z)\,dz
\le Cp\left(\frac{C_d}{p}\right)^p.
\]
Using again $|K(y,z)|\le C\varepsilon_{p,j}|y|$ on $\mathcal E_{p,j}$,
we obtain
\begin{equation}\label{exterior-tail-key1}
p\int_{\mathcal E_{p,j}}|K(y,z)|F_{p,j}(z)\,dz
=o\bigl((1+|y|)^\tau\bigr).
\end{equation}
Since $|z|\ge d/\varepsilon_{p,j}$ on $\mathcal E_{p,j}$, the same
estimate for the $e^Z$ term follows directly from $e^Z\le C(1+|z|)^{-8}$.
Consequently,
\begin{equation}\label{full-tail-key1}
p\int_{\Omega_{p,j}\setminus B_{e^{\beta p}}(0)}
|K(y,z)|\bigl|F_{p,j}(z)-e^{Z(z)}\bigr|\,dz
=o\bigl((1+|y|)^\tau\bigr).
\end{equation}

Let
\[
f_p(s):=\left[\left(1+\frac{s}{p}\right)^+\right]^p.
\]
On $B_{e^{\beta p}}(0)$, the mean value theorem gives
\[
p\bigl[f_p(Z_{p,j})-f_p(Z)\bigr]
=c_{p,j}(z)\eta_{p,j}(z),
\]
where
\begin{equation}\label{c_pj}
c_{p,j}(z):=
\int_0^1
\left[
\left(1+\frac{Z(z)+t(Z_{p,j}(z)-Z(z))}{p}\right)^+
\right]^{p-1}\,dt.
\end{equation}
By Lemma~\ref{llma} and the local convergence $Z_{p,j}\to Z$, for every
small $\sigma>0$,
\[
0\le c_{p,j}(z)\le\frac{C_\sigma}{(1+|z|)^{8-\sigma}}
\quad\text{in }B_{e^{\beta p}}(0),
\]
and
\begin{equation}\label{c_pjto}
c_{p,j}\to e^Z
\quad\text{locally uniformly in }\mathbb R^4.
\end{equation}
Set
\begin{equation}\label{bp}
b_p(z):=p\bigl[f_p(Z(z))-e^{Z(z)}\bigr].
\end{equation}
From \eqref{eta-pre-tail} and \eqref{full-tail-key1},
\begin{equation}\label{nitaBp}
\eta_{p,j}(y)
=
\int_{B_{e^{\beta p}}(0)}K(y,z)c_{p,j}(z)\eta_{p,j}(z)\,dz
+B_p(y)+o\bigl((1+|y|)^\tau\bigr),
\end{equation}
where
\[
B_p(y):=\int_{B_{e^{\beta p}}(0)}K(y,z)b_p(z)\,dz.
\]
A Taylor expansion gives
\begin{equation*}
b_p\to-\frac12Z^2e^Z
\quad\text{locally uniformly in }\mathbb R^4,
\end{equation*}
and
\[
|b_p(z)|\le
C\frac{\ln^2(2+|z|)}{(1+|z|)^8}.
\]
Hence, using the logarithmic kernel estimate as above,
\begin{equation}\label{BP<}
|B_p(y)|\le C\ln(2+|y|).
\end{equation}

We finally prove \eqref{lpy1} by contradiction. Let
\[
N_p:=\max_{|y|\le d_0/\varepsilon_{p,j}}
\frac{|\eta_{p,j}(y)|}{(1+|y|)^\tau}.
\]
Assume that $N_p\to\infty$ and set
\[
\psi_{p,j}:=\frac{\eta_{p,j}}{N_p}.
\]
Then $|\psi_{p,j}(y)|\le(1+|y|)^\tau$ and there exists $y_p$ such that
\begin{equation}\label{faipc}
\frac{|\psi_{p,j}(y_p)|}{(1+|y_p|)^\tau}=1.
\end{equation}
By \eqref{nitaBp} and \eqref{BP<},
\[
\psi_{p,j}(y)
=
\int_{B_{e^{\beta p}}(0)}K(y,z)c_{p,j}(z)\psi_{p,j}(z)\,dz
+o\bigl((1+|y|)^\tau\bigr).
\]
Choosing $\sigma>0$ small, we obtain
\[
\left|
\int_{B_{e^{\beta p}}(0)}K(y,z)c_{p,j}(z)\psi_{p,j}(z)\,dz
\right|
\le C\ln(2+|y|).
\]
Thus $|y_p|$ is bounded; otherwise \eqref{faipc} would be contradicted.
By standard elliptic estimates, up to a subsequence,
\[
\psi_{p,j}\to\psi\not\equiv0
\quad\text{in }C^4_{\mathrm{loc}}(\mathbb R^4),
\qquad
\Delta^2\psi=e^Z\psi,
\]
with $|\psi(y)|\le C(1+|y|)^\tau$. By Lemma~\ref{llm},
\[
\psi
=a_0\left(4+x\cdot\nabla Z\right)
+\sum_{i=1}^4a_i\partial_iZ.
\]
Since
$Z_{p,j}(0)=Z(0)=0$ and
$\nabla Z_{p,j}(0)=\nabla Z(0)=0$, we have
$\psi(0)=0$ and $\nabla\psi(0)=0$. Therefore
$a_0=a_1=\cdots=a_4=0$, a contradiction. This proves \eqref{lpy1}.
\end{proof}
\begin{Prop}\label{prop3-2}
Let $\eta_{p,j}$ be given by \eqref{dsa}. Then
\[
\eta_{p,j}\longrightarrow\eta_0
\quad\mbox{in }C^4_{\rm loc}(\R^4),
\]
where $\eta_0$ is the unique solution in the class \eqref{boundw0} satisfying
\begin{equation}\label{6-4-2}
\Delta^2\eta_0-e^Z\eta_0=-\frac12Z^2e^Z
\quad\mbox{in }\R^4,
\end{equation}
with
\begin{equation}\label{eta0-normalization}
\eta_0(0)=0,\qquad \nabla\eta_0(0)=0.
\end{equation}
Moreover, for every $\tau\in(0,1)$,
\begin{equation}\label{boundw0}
|\eta_0(x)|\le C_\tau(1+|x|)^\tau,
\end{equation}
and
\begin{equation}\label{eta0-green}
\eta_0(y)=\int_{\R^4}K(y,z)e^{Z(z)}
\left(\eta_0(z)-\frac12Z^2(z)\right)\,dz.
\end{equation}
\end{Prop}
\begin{proof}
By \eqref{c_pj}, \eqref{c_pjto}, \eqref{bp} and Proposition~\ref{key-1},
\[
\Delta^2\eta_{p,j}=c_{p,j}\eta_{p,j}+b_p,
\qquad
|\eta_{p,j}(y)|\le C(1+|y|)^\tau.
\]
Hence every subsequence admits a further subsequence converging in
$C^4_{\rm loc}(\R^4)$ to a solution $\eta_*$ of \eqref{6-4-2} satisfying
\eqref{boundw0}. Since $Z_{p,j}(0)=Z(0)=0$ and
$\nabla Z_{p,j}(0)=\nabla Z(0)=0$, every such limit satisfies
\eqref{eta0-normalization}.

If $\eta_*$ and $\widetilde\eta_*$ are two limits, then
$w=\eta_*-\widetilde\eta_*$ satisfies
$\Delta^2w=e^Zw$ and has sublinear growth. Lemma~\ref{llm} gives
\[
w=a_0(4+x\cdot\nabla Z)+\sum_{i=1}^4a_i\partial_iZ.
\]
The conditions $w(0)=0$ and $\nabla w(0)=0$ imply
$a_0=a_1=\cdots=a_4=0$. Thus the limit is unique and the whole sequence
converges to $\eta_0$.

Finally, let the right-hand side of \eqref{eta0-green} be denoted by
$\widetilde\eta_0$. The decay of $e^Z$ and \eqref{boundw0} imply that the
integral is well defined and $|\widetilde\eta_0(y)|\le C\log(2+|y|)$.
Moreover $\Delta^2\widetilde\eta_0=\Delta^2\eta_0$. Hence
$h=\eta_0-\widetilde\eta_0$ is an entire biharmonic function of sublinear
growth and therefore is constant. Since $K(0,z)=0$ and \eqref{eta0-normalization}
holds, $h(0)=0$, so $h\equiv0$. This proves \eqref{eta0-green} and, in particular, the sharper bound
\begin{equation}\label{eta0-log-growth-refined}
|\eta_0(y)|\le C\log(2+|y|),\qquad y\in\mathbb R^4.
\end{equation}
\end{proof}

Next, we will study the second order expansion for $Z_{p,j}$, and define
\begin{equation}\label{sat}
\theta_{p,j}(y)
:=
p\big[\eta_{p,j}(y)-\eta_0(y)\big]
=
p^2\left[
Z_{p,j}(y)-Z(y)-\frac{\eta_0(y)}p
\right],
\end{equation}
where $\eta_{p,j}$ is given by \eqref{dsa} and $\eta_{0}$ is the limit function obtained in Proposition \ref{prop3-2}. Similar to the proof of Proposition \ref{key-1} and Proposition \ref{prop3-2}, we prove the following estimate for  $\theta_{p,j}$.
\begin{Prop}\label{key-2}
Let $\theta_{p,j}$ be given by \eqref{sat}. Then for any fixed sufficiently small $d_0>0$ and $\tau_1\in(0,1)$, there exists $C>0$ such that
\begin{equation}\label{blpy1}
\big|\theta_{p,j}(x)\big|\leq C(1+|x| )^{\tau_1}~\mbox{ in}~ B_{\frac{d_0}{\e_{p,j}}}(0).
\end{equation}
 It follows that
\[Z_{p,j}=Z+\frac{\eta_{0}}{p}+O\left(\frac{1}{p^{2}}\right)
\,\,\mbox{ in }C^{4}_{loc}(\mathbb R^{4}).\]
\end{Prop}
\begin{proof}
Fix $\tau_1\in(0,1)$ and choose $0<\tau<\tau_1/2$. By
Proposition~\ref{prop3-2},
\begin{equation}\label{eta0-growth-key2}
|\eta_0(z)|\le C(1+|z|)^\tau.
\end{equation}
Recall that
\[
Z_{p,j}(y)
=\int_{\Omega_{p,j}}K(y,z)F_{p,j}(z)\,dz+R_p(y),
\]
where $F_{p,j}$, $K$, and $R_p$ are given by
\eqref{Fpj} and \eqref{RRP}. Moreover,
\[
Z(y)=\int_{\mathbb R^4}K(y,z)e^{Z(z)}\,dz,
\]
and Proposition~\ref{prop3-2} gives the representation \eqref{eta0-green}.
Therefore
\begin{equation}\label{theta-pre-tail}
\begin{aligned}
\theta_{p,j}(y)
={}&p^2\int_{\Omega_{p,j}}K(y,z)
\left[
F_{p,j}(z)-e^{Z(z)}
-\frac1p e^{Z(z)}
\left(\eta_0(z)-\frac12Z^2(z)\right)
\right]dz\\
&-p^2\int_{\mathbb R^4\setminus\Omega_{p,j}}K(y,z)e^{Z(z)}\,dz\\
&-p\int_{\mathbb R^4\setminus\Omega_{p,j}}K(y,z)e^{Z(z)}
\left(\eta_0(z)-\frac12Z^2(z)\right)dz
+p^2R_p(y).
\end{aligned}
\end{equation}
By \eqref{rough-epsilon-rate},
\eqref{epsilon-super-polynomial}, \eqref{K_y} and \eqref{eta0-growth-key2}, the last
three terms in \eqref{theta-pre-tail} are
$o((1+|y|)^{\tau_1})$ uniformly for
$|y|\le d_0/\varepsilon_{p,j}$. Hence
\begin{equation}\label{eqq3.18}
\theta_{p,j}(y)
=
p^2\int_{\Omega_{p,j}}K(y,z)
\left[
F_{p,j}-e^Z
-\frac1p e^Z\left(\eta_0-\frac12Z^2\right)
\right]dz
+o\bigl((1+|y|)^{\tau_1}\bigr).
\end{equation}

We next remove the far field. With the notation introduced in the proof
of Proposition~\ref{key-1},
\[
\Omega_{p,j}\setminus B_p(0)
=
\left(B_{d/\varepsilon_{p,j}}(0)\setminus B_p(0)\right)
\cup\bigcup_{s\ne j}\mathcal B_{p,j}^{(s)}
\cup\mathcal E_{p,j}.
\]
On the first region, Lemma~\ref{llma} and \eqref{log-kernel-tail} imply,
for $0<\sigma<1$,
\[
\begin{aligned}
&p^2\int_{B_{d/\varepsilon_{p,j}}\setminus B_p}
|K(y,z)|\bigl|F_{p,j}(z)-e^{Z(z)}\bigr|\,dz\\
&\qquad\le
Cp^{-2+\sigma}\bigl(1+\ln(2+|y|)+\ln p\bigr)=o(1),
\end{aligned}
\]
because \eqref{rough-epsilon-rate} gives
$\ln(2+|y|)=O(p)$ on $B_{d_0/\varepsilon_{p,j}}(0)$.
On the other bubble regions, using \eqref{global-rescaled-mass},
\[
\frac{p^2\displaystyle\sum_{s\ne j}
\int_{\mathcal B_{p,j}^{(s)}}|K(y,z)|F_{p,j}(z)\,dz}
{(1+|y|)^{\tau_1}}
\le Cp^2\varepsilon_{p,j}^{\tau_1}\longrightarrow0.
\]
The contribution of $\mathcal E_{p,j}$ is super-polynomially small by
\eqref{11-14-03N}, exactly as in \eqref{exterior-tail-key1}. Therefore
\begin{equation*}
p^2\int_{\Omega_{p,j}\setminus B_p(0)}
|K(y,z)|\,|F_{p,j}(z)-e^{Z(z)}|\,dz
=o\bigl((1+|y|)^{\tau_1}\bigr).
\end{equation*}
Furthermore, by \eqref{eta0-growth-key2},
\[
e^Z\bigl(|\eta_0|+Z^2\bigr)
\le
C(1+|z|)^{-8+\tau}
+C(1+|z|)^{-8}\ln^2(2+|z|),
\]
and \eqref{log-kernel-tail} yields
\begin{equation}\label{eqq3.20}
p\int_{\Omega_{p,j}\setminus B_p(0)}
|K(y,z)|e^{Z(z)}\bigl(|\eta_0(z)|+Z^2(z)\bigr)\,dz
=o\bigl((1+|y|)^{\tau_1}\bigr).
\end{equation}
Combining \eqref{eqq3.18}--\eqref{eqq3.20}, we obtain
\begin{equation}\label{xitap1}
\theta_{p,j}(y)
=
p^2\int_{B_p(0)}K(y,z)
\left[
F_{p,j}-e^Z
-\frac1p e^Z\left(\eta_0-\frac12Z^2\right)
\right]dz
+o\bigl((1+|y|)^{\tau_1}\bigr).
\end{equation}

Set
\begin{equation*}
\widehat Z_p:=Z+\frac{\eta_0}{p},
\end{equation*}
and
\begin{equation*}
\widehat F_p:=\left[ \left(1+\frac{\widehat Z_p}{p}\right)^+\right]^p.
\end{equation*}
By \eqref{Z(x)} and \eqref{eta0-log-growth-refined},
$|Z|+|\eta_0|\le C\ln p$ on $B_p(0)$, so
$|\widehat Z_p|/p=o(1)$ uniformly there. We write
\begin{equation*}
\begin{aligned}
F_{p,j}-e^Z-\frac1p e^Z\left(\eta_0-\frac12Z^2\right)
={}&(F_{p,j}-\widehat F_p)+G_p.
\end{aligned}
\end{equation*}
By the mean value theorem and \eqref{sat},
\[
p^2(F_{p,j}-\widehat F_p)
=a_{p,j}(z)\theta_{p,j}(z),
\]
where
\[
a_{p,j}(z):=
\int_0^1
\left[
\left(
1+\frac{\widehat Z_p(z)+t(Z_{p,j}(z)-\widehat Z_p(z))}{p}
\right)^+
\right]^{p-1}\,dt.
\]
For every sufficiently small $\sigma>0$,
\begin{equation}\label{estimateaipj}
a_{p,j}\to e^Z
\quad\text{locally uniformly in }\mathbb R^4,
\qquad
|a_{p,j}(z)|\le C(1+|z|)^{-8+\sigma}
\quad\text{in }B_p(0).
\end{equation}
On the other hand,
\begin{equation*}
p\ln\left(1+\frac{\widehat Z_p}{p}\right)
=
Z+\frac1p\left(\eta_0-\frac12Z^2\right)
+\frac1{p^2}\left(-Z\eta_0+\frac13Z^3\right)
+O\left(\frac{T(z)}{p^3}\right),
\end{equation*}
where
\[
T(z):=1+\ln^4(2+|z|)+(1+|z|)^{2\tau}.
\]
By \eqref{eta0-log-growth-refined},
$\eta_0-\frac12Z^2=O(\ln^2(2+|z|))$. Since
$\ln^6(2+|z|)\le C_\tau(1+|z|)^{2\tau}$, the same function $T$
controls the cubic remainder in the exponential expansion below.
Consequently,
\begin{equation*}
\widehat F_p
=
e^Z\left[
1+\frac1p\left(\eta_0-\frac12Z^2\right)
+\frac{D(z)}{p^2}
\right]
+O\left(\frac{e^ZT(z)}{p^3}\right),
\end{equation*}
where
\[
D(z):=-Z\eta_0+\frac13Z^3
+\frac12\left(\eta_0-\frac12Z^2\right)^2.
\]
Thus
\[
|p^2G_p(z)|
\le
Ce^{Z(z)}\left[
1+\ln^4(2+|z|)+(1+|z|)^{2\tau}
\right].
\]
It follows from \eqref{xitap1} that
\begin{equation}\label{theta}
\theta_{p,j}(y)
=
\int_{B_p(0)}K(y,z)a_{p,j}(z)\theta_{p,j}(z)\,dz
+B'_p(y)
+o\bigl((1+|y|)^{\tau_1}\bigr),
\end{equation}
where
\[
B'_p(y):=p^2\int_{B_p(0)}K(y,z)G_p(z)\,dz.
\]
As in the proof of \eqref{BP<},
\begin{equation}\label{B_p}
|B'_p(y)|\le C\ln(2+|y|).
\end{equation}

We conclude by contradiction. Let
\[
M_p:=
\sup_{|y|\le d_0/\varepsilon_{p,j}}
\frac{|\theta_{p,j}(y)|}{(1+|y|)^{\tau_1}}.
\]
Assume that $M_p\to\infty$ and set
\[
\psi_p:=\frac{\theta_{p,j}}{M_p}.
\]
Then $|\psi_p(y)|\le(1+|y|)^{\tau_1}$, and by \eqref{theta} and \eqref{B_p},
\begin{equation}\label{faiip}
\psi_p(y)
=
\int_{B_p(0)}K(y,z)a_{p,j}(z)\psi_p(z)\,dz
+o\bigl((1+|y|)^{\tau_1}\bigr).
\end{equation}
Choosing $\sigma>0$ small enough, \eqref{estimateaipj} gives
\begin{equation}\label{yichangtiaojifen}
\left|
\int_{B_p(0)}K(y,z)a_{p,j}(z)\psi_p(z)\,dz
\right|
\le C\ln(2+|y|).
\end{equation}
If $y_p$ is a point where the normalized supremum is attained, then
\eqref{faiip}--\eqref{yichangtiaojifen} imply that $|y_p|$ is bounded.
Hence, by standard elliptic estimates, up to a subsequence,
\[
\psi_p\to\psi\not\equiv0
\quad\text{in }C^4_{\mathrm{loc}}(\mathbb R^4),
\qquad
\Delta^2\psi=e^Z\psi,
\]
with $|\psi(y)|\le C(1+|y|)^{\tau_1}$. Moreover,
$\theta_{p,j}(0)=0$ and $\nabla\theta_{p,j}(0)=0$, so
$\psi(0)=0$ and $\nabla\psi(0)=0$. Arguing as in the proof of
Proposition~\ref{key-1} and using Lemma~\ref{llm}, we obtain
$\psi\equiv0$, a contradiction. Therefore
\[
|\theta_{p,j}(y)|\le C(1+|y|)^{\tau_1}
\qquad
\text{for }|y|\le\frac{d_0}{\varepsilon_{p,j}},
\]
which proves \eqref{blpy1}. The local $C^4$ expansion follows immediately.
\end{proof}

Now, based on previous discussion, we can get the second-order Taylor's expansion of $Z_p$, and then we improve \eqref{abs} to be more explicit.

\begin{Prop}\label{Prop:5-12-52}
Let $d>0$ be fixed and sufficiently small. Then there exists a sufficiently small constant $\delta\in (0,1)$ such that
\begin{equation}\label{5-12-52}
\displaystyle\int_{B_{d/\varepsilon_{p,j}}(0)}\Big[\Big(1+\frac{ Z_{p,j}(z)}{p}
  \Big)^+\Big]^p\,dz
   = 64{\pi}^{2}\left(1-\frac{13}{3p}\right)+O\Big(\frac{1}{p^{2-\delta}}\Big).
\end{equation}
\end{Prop}
\begin{proof}
From \eqref{Fpj} and Lemma \ref{llma}, for any small $\delta>0$, we have $0\le F_{p,j}(y)\le C_\delta\big(1+|y|\big)^{-8+\delta}$ and $|y|\le \frac d{\varepsilon_{p,j}}$.
Thus, it follows that
\[
\begin{aligned}
\displaystyle\int_{B_{d/\varepsilon_{p,j}}(0)\setminus B_p(0)}F_{p,j}(y)\,dy
&\le
C\displaystyle\int_{|y|>p}\big(1+|y|\big)^{-8+\delta}\,dy
&\le
C\displaystyle\int_p^\infty r^{3}r^{-8+\delta}\,dr
&=
C p^{-4+\delta},
\end{aligned}
\]
then we have \[
\displaystyle\int_{B_{d/\varepsilon_{p,j}}(0)\setminus B_p(0)}F_{p,j}(y)\,dy
=
O\big(p^{-2+\delta}\big).
\]
By Proposition \ref{key-1}, Proposition \ref{prop3-2} and Proposition \ref{key-2}, we have
\begin{equation} \label{Z}
Z_{p,j}
=
Z+\frac{\eta_0}{p}+\frac{\theta_{p,j}}{p^2},
\end{equation}
and we can get from  \eqref{Z(x)}, \eqref{boundw0}, \eqref{blpy1} that in \( B_p (0) \)
\begin{equation} \label{guji}
\big|Z(y)\big|
\le C\ln\big(2+|y|\big)
\le C\ln p, ~~\text{and}~~
\big|\eta_0(y)\big|,\big|\theta_{p,j}(y)\big| \le C\big(1+|y|\big)^\tau\le Cp^\tau.
\end{equation}
Hence, we obtain
\begin{equation} \label{Z_pj}
|Z_{p,j}(y)|
\le
C\ln p
+
Cp^{\tau-1}
+
Cp^{\tau-2}
\le
C\ln p.
\end{equation}
Since $p$ is large enough,  it follows that
\begin{equation}\label{taile1}
    p\ln\left[\left(1+\frac{Z_{p,j}}{p}\right)^+\right]=
    p\ln\left(1+\frac{Z_{p,j}}{p}\right)
=
Z_{p,j}-\frac{Z_{p,j}^2}{2p}
+
\frac{Z_{p,j}^3}{3p^2}
+
O\left(\frac{|Z_{p,j}|^4}{p^3}\right).
\end{equation}
By \eqref{Z}, we have
\begin{equation}\label{Zp2}
    \frac{Z_{p,j}^2}{2p}
=
\frac{Z^2}{2p}
+
\frac{Z\eta_0}{p^2}
+
O\left(
\frac{\eta_0^2+|Z||\theta_{p,j}|}{p^3}
\right)
\end{equation}and
\begin{equation}\label{Zp3}
    \frac{Z_{p,j}^3}{3p^2}
=
\frac{Z^3}{3p^2}
+
O\left(
\frac{|Z|^2|\eta_0|+|\eta_0|^2+|Z|^2|\theta_{p,j}|}
{p^3}
\right).
\end{equation}
 Substituting \eqref{Z}, \eqref{Zp2} and \eqref{Zp3} into \eqref{taile1}, we obtain
\begin{equation}\label{taile2}
    p\ln\left(1+\frac{Z_{p,j}}{p}\right)
=
Z
+
\frac1p\left(\eta_0-\frac12Z^2\right)
+
E_{p,j}(y),
\end{equation}where
\[
\big|E_{p,j}(y)\big|
\le
\frac C{p^2}
\left(
1+\big|\theta_{p,j}\big|+\big|Z\big|\big|\eta_0\big|+\big|Z\big|^3+\eta_0^2
\right)
+
O\left(\frac{|Z_{p,j}|^4}{p^3}\right).
\]
Owing to \eqref{Z_pj}, we obtain
\begin{equation*}
    |E_{p,j}(y)|
\le
\frac C{p^2}H_{p,j}(y),
\end{equation*}
with
\[
H_{p,j}(y):=
1+\big|\theta_{p,j}(y)\big|+\big|Z(y)\big|\big|\eta_0(y)\big|+\big|Z(y)\big|^3+\eta_0^2(y)+\ln^4\big(2+|y|\big).
\]
By \eqref{Z(x)} and \eqref{guji}, it is easy to find that \[
H_{p,j}(y)
\le
C\left[
1+\big(1+|y|\big)^\tau
+\big(1+|y|\big)^{2\tau}
+\ln^4\big(2+|y|\big)
\right],
\] and
\begin{equation}\label{jihen1}
    \displaystyle\int_{B_p(0)}e^Z H_{p,j}(y)\,dy\le C.
\end{equation}
Now, we denote
\[
g(y):=\eta_0(y)-\frac12Z^2(y).
\]
 Then in  $B_p(0)$, by \eqref{Fpj}, \eqref{taile2} and Taylor's expansion, it follows that
\begin{equation}\label{FpO}
\begin{aligned}
F_{p,j}
=
e^Z\exp\left(\frac{g}{p}+E_{p,j}\right)
=e^Z\left(1+\frac g p\right)
+
O\left[
e^Z\left(\frac{g^2}{p^2}+|E_{p,j}|\right)
\right].
\end{aligned}
\end{equation}
From \eqref{guji}, we have
$g^2
\le
C\left(\big(1+|y|\big)^{2\tau}+\ln^4\big(2+|y|\big)\right)$,
 and by \eqref{jihen1}, we have
\begin{equation}\label{jifen3}
    \displaystyle\int_{B_p(0)}e^Z\left(\frac{g^2}{p^2}+|E_{p,j}|\right)\,dy
\le
\frac C{p^2}.
\end{equation}
Together with \eqref{FpO} and \eqref{jifen3}, it follows that
\begin{equation*}
    \displaystyle\int_{B_p(0)}F_{p,j}(y)\,dy
=
\displaystyle\int_{B_p(0)}e^Z\,dy
+
\frac1p\displaystyle\int_{B_p(0)}e^Zg\,dy
+
O(p^{-2}).
\end{equation*}
By some basic computations, it has
\[
\displaystyle\int_{\mathbb R^4\setminus B_p(0)}e^Z\,dy
\le
C\displaystyle\int_p^\infty r^{-8}r^3\,dr
=
Cp^{-4},\]
and  \[
\begin{aligned}
\displaystyle\int_{\mathbb R^4\setminus B_p(0)}e^Z|g|\,dy
&\le
C\displaystyle\int_p^\infty
r^{-8}
\left(r^\tau+\ln^2r\right)
r^3\,dr
\le
C p^{-4+\tau}
+
Cp^{-4}\ln^2p.
\end{aligned}
\]
Thus, it holds
\begin{equation*}
    \displaystyle\int_{B_p(0)}F_{p,j}(y)\,dy
=
\displaystyle\int_{\mathbb R^4}e^Z\,dy
+
\frac1p
\displaystyle\int_{\mathbb R^4}e^Zg\,dy
+
O(p^{-2+\delta}).
\end{equation*}
Now, we define \begin{equation*}
    A:=
\displaystyle\int_{\mathbb R^4}e^Z
\left(\eta_0-\frac12Z^2\right)\,dy.
\end{equation*}
We claim that
\begin{equation} \label{AAA}
 A
=
-\frac{416}{3}|\mathbb S^3|,
\end{equation}
and the proof has been put into the  Appendix \ref{76}.
Hence
\[
\displaystyle\int_{B_{d/\varepsilon_{p,j}}(0)}\Big[\Big(1+\frac{ Z_{p,j}(z)}{p}
  \Big)^+\Big]^p\,dz
   = 64{\pi}^{2}\left(1-\frac{13}{3p}\right)+O\Big(\frac{1}{p^{2-\delta}}\Big).
\]
\end{proof}

For convenience, in the rest of this section (or paper) we write
\begin{equation*}
m_{p,j}:=u_p(x_{p,j}),\qquad
\mathbf{x}_p:=\big(x_{p,1},\cdots,x_{p,k}\big),\qquad
q_p:=1-\frac{13}{3p}.
\end{equation*}
We first refine the logarithmic moment which appears when the Green representation is evaluated at $x_{p,j}$.

\begin{Lem}\label{lem:log-moment}
For the fixed small constant $\delta$ above, it holds
\begin{equation*}
\int_{B_{d/\varepsilon_{p,j}}(0)}
\ln|z|\Big[\Big(1+\frac{Z_{p,j}(z)}p\Big)^+\Big]^p dz
=
32\pi^2\ln(8\sqrt6)+O\Big(\frac1p\Big).
\end{equation*}
\end{Lem}
\begin{proof}
By Proposition \ref{key-2},
\[
Z_{p,j}=Z+\frac{\eta_0}{p}+\frac{\theta_{p,j}}{p^2},
\qquad
|\eta_0(z)|+|\theta_{p,j}(z)|\le C(1+|z|)^\tau
\]
for some fixed $\tau\in(0,1)$. Arguing exactly as in the proof of Proposition \ref{Prop:5-12-52}, on $B_p(0)$ we have
\[
\Big[\Big(1+\frac{Z_{p,j}}p\Big)^+\Big]^p
=e^Z\left[1+\frac1p\left(\eta_0-\frac{Z^2}{2}\right)\right]+\mathcal R_{p,j},
\]
where
\[
\int_{B_p(0)}(1+|\ln|z||)|\mathcal R_{p,j}(z)|\,dz=O(p^{-2}).
\]
Lemma \ref{llma} also gives
\[
\int_{B_{d/\varepsilon_{p,j}}(0)\setminus B_p(0)}
|\ln|z||\Big[\Big(1+\frac{Z_{p,j}}p\Big)^+\Big]^p dz
=O(p^{-4+\delta}\ln p).
\]
Consequently,
\[
\int_{B_{d/\varepsilon_{p,j}}(0)}
\ln|z|\Big[\Big(1+\frac{Z_{p,j}}p\Big)^+\Big]^p dz
=
\int_{\R^4}\ln|z|e^Zdz+O(p^{-1}).
\]
Since $e^Z=(1+|z|^2/(8\sqrt6))^{-4}$, a direct radial computation yields
\[
\int_{\R^4}\ln|z|e^Zdz=32\pi^2\ln(8\sqrt6),
\]
and the conclusion follows.
\end{proof}

\begin{Prop}\label{prop:peak-expansion-xp}
For the same fixed small constant $\delta$ and $j=1,\cdots,k$, it holds
\begin{equation}\label{peak-height-xp}
\begin{split}
m_{p,j}=\sqrt e\bigg(1-\frac{\ln p}{p-1}
+\frac1p\Big[-32\pi^2\Psi_{k,j}(\mathbf{x}_p)
+2\ln(8\sqrt6)+\frac83\Big]\bigg)
+O(p^{-2+\delta}).
\end{split}
\end{equation}
In particular,
\begin{equation}\label{peak-ratio-rough}
\frac{m_{p,j}}{m_{p,s}}=1+O\Big(\frac1p\Big),\qquad 1\le j,s\le k.
\end{equation}
Moreover,
\begin{equation}\label{epsilon-xp}
\varepsilon_{p,j}
=e^{-p/8}\Bigr(
 e^{8\pi^2\Psi_{k,j}(\mathbf{x}_p)-\frac12\ln(8\sqrt6)-\frac{13}{24}}
+O(p^{-1+\delta})\Bigr),
\end{equation}
and
\begin{equation}\label{epsilon-ratio-xp}
\frac{\varepsilon_{p,j}}{\varepsilon_{p,s}}
=e^{8\pi^2(\Psi_{k,j}(\mathbf{x}_p)-\Psi_{k,s}(\mathbf{x}_p))}
+O(p^{-1+\delta}).
\end{equation}
\end{Prop}
\begin{proof}
Evaluating the Green representation at $x_{p,j}$, using Proposition \ref{Prop:5-12-52} and Lemma \ref{lem:log-moment}, we obtain
\begin{equation}\label{coupled-height-short}
\begin{split}
p={}&-8q_p\ln\varepsilon_{p,j}
+64\pi^2q_p\left[R(x_{p,j})+
\sum_{\ell\ne j}\frac{m_{p,\ell}}{m_{p,j}}G(x_{p,\ell},x_{p,j})\right]
\\&-4\ln(8\sqrt6)+O(p^{-1+\delta}).
\end{split}
\end{equation}
Indeed, the self-interaction term follows from
\[
G(y,x_{p,j})=H(y,x_{p,j})-\frac1{8\pi^2}\ln|y-x_{p,j}|,
\]
while, for $\ell\ne j$, the smoothness of $G$ gives
\[
\int_{B_d(x_{p,\ell})}G(y,x_{p,j})(u_p^+)^pdy
=\frac{m_{p,\ell}}p\left[64\pi^2q_pG(x_{p,\ell},x_{p,j})+O(p^{-2+\delta})\right].
\]

By the definition of $\varepsilon_{p,j}$,
\begin{equation}\label{log-epsilon-mpj}
\ln\varepsilon_{p,j}
=-\frac14\big[\ln p+(p-1)\ln m_{p,j}\big].
\end{equation}
Subtracting \eqref{coupled-height-short} for two indices $j$ and $s$, and using $m_{p,j}\to\sqrt e$, gives \eqref{peak-ratio-rough}.
Hence
\[
R(x_{p,j})+
\sum_{\ell\ne j}\frac{m_{p,\ell}}{m_{p,j}}G(x_{p,\ell},x_{p,j})
=\Psi_{k,j}(\mathbf{x}_p)+O(p^{-1}).
\]
Substituting this relation and \eqref{log-epsilon-mpj} into \eqref{coupled-height-short}, we find
\[
\ln m_{p,j}
=\frac12-\frac{\ln p}{p-1}
+\frac1p\left[-32\pi^2\Psi_{k,j}(\mathbf{x}_p)
+2\ln(8\sqrt6)+\frac83\right]+O(p^{-2+\delta}).
\]
Exponentiating proves \eqref{peak-height-xp}. Finally, inserting this expansion into \eqref{log-epsilon-mpj} gives \eqref{epsilon-xp}, and \eqref{epsilon-ratio-xp} follows immediately.
\end{proof}

\vskip 0.1cm
In the last part of this section, we give two refined estimates that involve $u_p(x)$.
 For convenience, we define \begin{equation}\label{def_Cpj}
    C_{p,j}:=
\displaystyle\int_{B_d(x_{p,j})}(u^+_{p})^p(x) dx,
\end{equation}
then by  scaling and \eqref{defwpj}, combining the result \eqref{5-12-52} in Proposition \ref{Prop:5-12-52}, we have
\begin{equation}\label{3.553.55}
    \begin{split}
        C_{p,j}
&=
u_p^p(x_{p,j})\varepsilon _{p,j}^4
\displaystyle\int_{B_{d/\varepsilon _{p,j}}(0)}
\Big[\Big(1+\frac{ Z_{p,j}(z)}{p}
  \Big)^+\Big]^p\,dz\\
&=
\frac{u_p(x_{p,j})}p
\displaystyle\int_{B_{d/\varepsilon _{p,j}}(0)}
\Big[\Big(1+\frac{ Z_{p,j}(z)}{p}
  \Big)^+\Big]^p\,d z\\
&=
\frac{m_{p,j}}p
64\pi^2
\left(
1-\frac{13}{3p}
\right)
+
O(p^{-3+\delta}).
    \end{split}
\end{equation}

\begin{Lem}\label{prop:expansionupwithdelta}
For any fixed small $d>0$, it holds
\begin{equation}\label{dluo-1}
u_p(x)= \sum^k_{j=1}C_{p,j}G(x_{p,j},x)+
O\Big(\sum^k_{j=1}\frac{\varepsilon_{p,j}}{p^{2-\delta}}\Big)\,\,
~\mbox{in}~C^3\Big(\Omega\backslash \displaystyle\bigcup^k_{j=1} B_{2d}(x_{p,j})\Big),
\end{equation}
where $\delta$ is a small fixed positive  constant and $C_{p,j}$ is as in \eqref{def_Cpj}.
\end{Lem}
\begin{proof}
For $x\in \Omega\backslash  \bigcup^k_{j=1} B_{2d}(x_{p,j})$, we obtain
\begin{equation}\label{5-8-11}
\begin{split}
    u_p(x)= \displaystyle\int_{\Omega}G(y,x)
(u^+_{p})^{p}(y) dy=&\sum^k_{j=1}
 \displaystyle\int_{B_d(x_{p,j})}G(y,x)
(u^+_{p})^{p}(y)dy \\
&\hspace{2em}+\displaystyle\int_{\Omega\backslash \bigcup^k_{j=1} B_{d}(x_{p,j})}G(y,x)
(u^+_{p})^{p}(y)dy.
\end{split}
\end{equation}
By \eqref{11-14-03N} we have
\begin{equation}\label{5-8-12}
\begin{split}
 \displaystyle\int_{\Omega\backslash  \bigcup^k_{j=1} B_{d}(x_{p,j})}G(y,x)
(u^+_{p})^{p}(y)dy
=O\Big(\frac{C^p}{p^p}\Big)~~\mbox{uniformly for}~ x\in \Omega\backslash  \bigcup^k_{j=1} B_{2d}(x_{p,j}).
\end{split}\end{equation}
Moreover, it holds
\begin{equation}\label{5-8-13}
\begin{split}
\int_{B_d(x_{p,j})} G(y,x) \left(u_p^+\right)^p(y) dy
&= G(x_{p,j},x) \int_{B_d(x_{p,j})} \left(u_p^+\right)^p(y) dy \\
&\quad + \int_{B_d(x_{p,j})} \Big[ G(y,x) - G(x_{p,j},x) \Big] \left(u_p^+\right)^p(y) dy .
\end{split}
\end{equation}
As for the second term on the right hand of \eqref{5-8-13}, by Taylor's expansion and a scaling, it becomes
\begin{equation}
\begin{aligned}[b]
&\int_{B_d(x_{p,j})} \bigl( G(y,x) - G(x_{p,j},x) \bigr) (u^+_p)^{p}(y) \, dy \\
=& \frac{\varepsilon_{p,j} u_{p}(x_{p,j})}{p}
   \int_{B_{d/\varepsilon_{p,j}}(0)}
   \bigl\langle \nabla G(x_{p,j},x), z \bigr\rangle
   \Bigl[ \Bigl( 1 + \frac{Z_{p,j}(z)}{p} \Bigr)^{\!+} \Bigr]^{p} \, dz \\
&\quad + O\biggl( \frac{\varepsilon_{p,j}^{2}}{p}
   \int_{B_{d/\varepsilon_{p,j}}(0)}
   |z|^{2} \Bigl[ \Bigl( 1 + \frac{Z_{p,j}(z)}{p} \Bigr)^{\!+} \Bigr]^{p} \, dz \biggr).\label{5-8-14}
\end{aligned}
\end{equation}
It is easy to find that
\begin{equation}\label{f5-8-14}
\lim_{p\rightarrow \infty}
\displaystyle\int_{B_{\frac{d}{\varepsilon_{p,j}}}(0)}
\big| z\big|^{2}\cdot \Big[\Big(1+\frac{ Z_{p,j}(z)}{p}
  \Big)^+\Big]^pdz
=\displaystyle\int_{\R^4} |z|^{2}  e^{Z(z)}dz < \infty.
\end{equation}
We next estimate the first term on the right-hand side of \eqref{5-8-14}.
The function $\nabla G(x_{p,j},x)$ is uniformly bounded for
$x\in\Omega\setminus\bigcup_{s=1}^kB_{2d}(x_{p,s})$. Hence, by Lemma~\ref{llma},
for some $\delta\in(0,1)$,
\begin{equation}\label{dluo-a1}
\begin{aligned}
&\left|\int_{B_{d/\varepsilon_{p,j}}(0)\setminus B_p(0)}
\big\langle\nabla G(x_{p,j},x),z\big\rangle
\Big[\Big(1+\frac{Z_{p,j}(z)}p\Big)^+\Big]^p\,dz\right| \\
&\qquad\le C
\int_{B_{d/\varepsilon_{p,j}}(0)\setminus B_p(0)}|z|e^{Z_{p,j}(z)}\,dz \\
&\qquad\le C
\int_{B_{d/\varepsilon_{p,j}}(0)\setminus B_p(0)}
\frac{|z|}{1+|z|^{8-\delta}}\,dz
=O\bigl(p^{-3+\delta}\bigr),
\end{aligned}
\end{equation}
uniformly for such $x$.
Next, note that
\begin{equation*}
\frac{|Z_{p,j}|}{p},\ \frac{|Z_{p,j}|^2}{p}\to 0
\quad\text{uniformly in } B_p(0) \text{ as } p\to\infty.
\end{equation*}
Using the relation $Z_{p,j}=Z+\frac{\eta_{p,j}}{p}$ from \eqref{dsa} and Taylor's expansion, when $p$ large enough, we have
\begin{equation*}
\begin{split}
\int_{ B_{{p}}(0)}
& \Bigr\langle\nabla G(x_{p,j},x),z\Bigr\rangle \Big[\Big(1+\frac{ Z_{p,j}(z)}{p}
\Big)^+\Big]^pdz\\
&= \int_{ B_{{p}}(0)} \Bigr\langle\nabla G(x_{p,j},x),z\Bigr\rangle
e^{ p\ln \big(1+\frac{ Z_{p,j}(z)}{p}\big) }dz\\
&= \int_{ B_{{p}}(0)} \Bigr\langle\nabla G(x_{p,j},x),z\Bigr\rangle
e^{Z_{p,j}(z)}\Big(1-\frac{Z^2_{p,j}(z)}{2p}+o\big(\frac{Z^2_{p,j}(z)}{p}\big) \Big)\,dz \\
&= \int_{ B_{{p}}(0)}\Bigr\langle\nabla G(x_{p,j},x),z\Bigr\rangle
e^{Z(z)}\Big(1+\frac{\eta_{p,j}(z)}{p}+o\big(\frac{|\eta_{p,j}(z)|}{p}\big) \Big)
\Big(1-\frac{Z^2_{p,j}(z)}{2p}+o\big(\frac{Z^2_{p,j}(z)}{p}\big) \Big)\,dz .
\end{split}
\end{equation*}
From Lemma \ref{llma}, Proposition \ref{key-1}, Proposition \ref{prop3-2} and the dominated convergence theorem, we then conclude that
\begin{equation}\label{0.50}
\int_{ B_{{p}}(0)}
\Bigr\langle\nabla G(x_{p,j},x),z\Bigr\rangle
\Big[\Big(1+\frac{ Z_{p,j}(z)}{p}\Big)^+\Big]^pdz
= O\Big( \frac{1}{p} \Big).
\end{equation}
Therefore, combining \eqref{dluo-a1} and \eqref{0.50}, we obtain
\begin{equation} \label{3.73}
\int_{B_{d/\varepsilon_{p,j}}(0)}
\Bigr\langle\nabla G(x_{p,j},x),z\Bigr\rangle
\Big[\Big(1+\frac{ Z_{p,j}(z)}{p}\Big)^+\Big]^pdz
= O\Big(\frac{1}{p}\Big)~\text{uniformly for } x\in \Omega\setminus \bigcup_{j=1}^k B_{2d}(x_{p,j}).
\end{equation}
Thus we can prove from \eqref{5-8-14}, \eqref{f5-8-14} and \eqref{3.73} that
\begin{equation}\label{g5-8-14}
\begin{split}
\displaystyle\int_{B_d(x_{p,j})}&\Big(G(y,x)-G(x_{p,j},x)\Big) (u^+_{p})^{p}(y)dy
=O\Big(\frac{\varepsilon_{p,j}}{p^2}+\frac{\varepsilon_{p,j}^2}{p}\Big)
=O\Big(\frac{\varepsilon_{p,j}}{p^{2-\delta}}\Big)
\end{split}
\end{equation}
 Thus, by \eqref{5-8-13} and \eqref{g5-8-14},
\begin{equation}\label{5-8-15}
\begin{split}
 \displaystyle\int_{B_d(x_{p,j})}G(y,x)
(u^+_{p})^{p}(y)dy= C_{p,j} G(x_{p,j},x)+
O\Big(\frac{\varepsilon_{p,j}}{p^{2-\delta}}\Big) ~\mbox{uniformly in}~ \Omega\backslash  \bigcup^k_{j=1} B_{2d}(x_{p,j}).
\end{split}
\end{equation}
Therefore, combining \eqref{5-8-11}, \eqref{5-8-12} and \eqref{5-8-15}, we obtain
\begin{equation*}
u_p(x)= \sum^k_{j=1}C_{p,j}G(x_{p,j},x)+
O\Big(\sum^k_{j=1}\frac{\varepsilon_{p,j}}{p^{2-\delta}}\Big)\,\,
~\mbox{uniformly in}~ \Omega\backslash \displaystyle\bigcup^k_{j=1} B_{2d}(x_{p,j}).
\end{equation*}

For $|\alpha|\le3$, differentiating the Green representation gives
\[
D_x^\alpha u_p(x)=\int_\Omega D_x^\alpha G(y,x)(u_p^+)^p(y)\,dy.
\]
The exterior part is $O(C^p/p^p)$ uniformly on
$\Omega\setminus\cup_{j=1}^kB_{2d}(x_{p,j})$. On \( B_{d}(x_{p,j})\), by Taylor's formula and a scaling  $y=x_{p,j}+\e_{p,j}z$, we have
\[
D_x^\alpha G(y,x) - D_x^\alpha G(x_{p,j},x) =
\e_{p,j}\nabla D_x^\alpha G(x_{p,j},x)\cdot z
+O(\e_{p,j}^2|z|^2).
\]
With the argument leading to \eqref{3.73}, it follows that
\[
\int_{B_{d/\e_{p,j}}(0)} \Big\langle \nabla D_x^\alpha G(x_{p,j},x) , z \Big\rangle \Big[\Big(1+\frac{Z_{p,j}(z)}p\Big)^+\Big]^p dz=O\left(\frac{1}{p}\right) ~~\text{uniformly for } x\in \Omega\setminus \bigcup_{j=1}^k B_{2d}(x_{p,j}).
\]
While \eqref{f5-8-14} yields a uniformly bounded second moment. Therefore,
\[
\int_{B_d(x_{p,j})}\!D_x^\alpha G(y,x)(u_p^+)^p(y)dy
=C_{p,j}D_x^\alpha G(x_{p,j},x)
+O\Big(\frac{\varepsilon_{p,j}}{p^{2-\delta}}\Big) ~\mbox{uniformly in}~ \Omega\backslash  \bigcup^k_{j=1} B_{2d}(x_{p,j}).
\]
Summing over $j$ proves \eqref{dluo-1} for every $|\alpha|\le3$, and hence in $C^3$.

\end{proof}

\begin{Prop}
 There exists a sufficiently small constant $\delta\in (0,1)$, such that, for $j=1,\cdots,k$, it holds
\begin{equation}\label{daluo-gil}
C_{p,j} \frac{\partial R(x_{p,j})}{\partial{x_i}}+2\displaystyle\sum^k_{m=1,m\neq j}
C_{p,m}
D_{x_i} G(x_{p,m},x_{p,j})
=O\Big(\frac{\varepsilon_{p}}{p^{2-\delta}}\Big),
\end{equation}
where $D_{x_i}G(y,x):=\frac{\partial G(y,x)}{\partial x_i}$, $C_{p,j}$ is defined in \eqref{def_Cpj} and ${\varepsilon}_{p}:=\max\big\{\varepsilon_{p,1},
\cdots,\varepsilon_{p,k}\big\}$.
\end{Prop}
\begin{proof}
By Pohozaev identity \eqref{aclp-1},  we obtain
\begin{equation*}
Q_{j}(u_p,u_p)=\frac{2}{p+1}\displaystyle\int_{\partial B_{\theta}(x_{p,j})}(u_{p}^+)^{p+1}\nu_i \,d\sigma.
\end{equation*}
By the asymptotic expansion of $u_p$ in Lemma \ref{prop:expansionupwithdelta},
\begin{equation*}
u_p(x)= \sum^k_{j=1}C_{p,j}G(x_{p,j},x)+
O\Big(\sum^k_{j=1}\frac{\varepsilon_{p,j}}{p^{2-\delta}}\Big)\,\,
~\mbox{in}~C^3\Big(\Omega\backslash \displaystyle\bigcup^k_{j=1} B_{2d}(x_{p,j})\Big),
\end{equation*}
combining the estimate $C_{p,j}=O\big(\frac{1}{p}\big)$ given by \eqref{3.553.55}, it follows that
\begin{equation*}
Q_j(u_p,u_p)=\sum^k_{s=1}\sum^k_{m=1}
C_{p,s} C_{p,m} Q_{j}\Big( G(x_{p,s},x), G(x_{p,m},x)\Big)+O\Big(\frac{\varepsilon_{p}}{p^{3-\delta}}\Big).
\end{equation*}
On the other hand, by \eqref{11-14-03N}, it holds
\begin{equation*}
\frac{2}{p+1}\displaystyle\int_{\partial B_{\theta}(x_{p,j})}(u_{p}^+)^{p+1}\nu_i \,d\sigma
=O\Big( \frac{C^p}{p^{p+2}}  \Big)=O\Big(\frac{\varepsilon_{p}}{p^{3-\delta}}\Big).
\end{equation*}
Hence we proved \eqref{daluo-gil} by above analysis.
\end{proof}

\begin{Prop}\label{prop:location-rate}
Assume that ${\mathbf{x}}_{\infty}$ is a nondegenerate critical point of $\Psi_k$. Then
\begin{equation}\label{xp-rate}
|\mathbf{x}_p-{\mathbf{x}}_{\infty}|=O\Big(\frac1p\Big).
\end{equation}
Consequently,
\begin{equation}\label{Psi-xp-xinf}
\Psi_{k,j}(\mathbf{x}_p)=\Psi_{k,j}({\mathbf{x}}_{\infty})+O\Big(\frac1p\Big),
\qquad j=1,\cdots,k.
\end{equation}
\end{Prop}
\begin{proof}
From \eqref{3.553.55} and \eqref{peak-ratio-rough},
\[
C_{p,j}\asymp \frac1p,
\qquad
\frac{C_{p,m}}{C_{p,j}}=1+O\Big(\frac1p\Big).
\]
Moreover, $m_{p,j}\to\sqrt e$ and the definition of $\varepsilon_{p,j}$ imply
$\varepsilon_p=o(p^{-N})$ for every fixed $N>0$. Dividing \eqref{daluo-gil} by $C_{p,j}$, we obtain
\[
\nabla R(x_{p,j})
+2\sum_{m\ne j}\nabla_2G(x_{p,m},x_{p,j})
=O\Big(\frac1p\Big).
\]
By the symmetry of $G$ and the definition \eqref{stts}, this is equivalent to
\begin{equation}\label{grad-Psi-xp}
|\nabla\Psi_k(\mathbf{x}_p)|=O\Big(\frac1p\Big).
\end{equation}
Since $\nabla\Psi_k({\mathbf{x}}_{\infty})=0$ and $D^2\Psi_k({\mathbf{x}}_{\infty})$ is invertible, Taylor's formula gives
\[
\nabla\Psi_k(\mathbf{x}_p)
=D^2\Psi_k({\mathbf{x}}_{\infty})(\mathbf{x}_p-{\mathbf{x}}_{\infty})
+O(|\mathbf{x}_p-{\mathbf{x}}_{\infty}|^2).
\]
As $\mathbf{x}_p\to {\mathbf{x}}_{\infty}$, the quadratic term can be absorbed, and \eqref{xp-rate} follows from \eqref{grad-Psi-xp}. Formula \eqref{Psi-xp-xinf} is then immediate from the smoothness of $\Psi_{k,j}$.
\end{proof}

\begin{proof}[\textbf{Proof of Theorem \ref{th1}}]
The estimates \eqref{5-7-52}--\eqref{3-29-03} follow from Proposition \ref{prop:peak-expansion-xp}, while \eqref{lst} follows from Propositions \ref{prop3-2} and \ref{key-2}. If ${\mathbf{x}}_{\infty}$ is a nondegenerate critical point of $\Psi_k$, Proposition \ref{prop:location-rate} yields \eqref{xp-rate-intro}, and \eqref{Psi-xp-xinf} allows us to replace $\mathbf{x}_p$ by ${\mathbf{x}}_{\infty}$ in \eqref{5-7-52}--\eqref{3-29-03} without changing the remainder estimates.
\end{proof}

\vskip 0.8cm

\section{Morse index and nondegeneracy of the multi-spike solutions}\label{section4}

In this section, employing the refined estimates obtained in Section~\ref{section3}, we complete the proof of Theorem~\ref{thMorse} and Corollary~\ref{th1.1} by investigating the asymptotic behavior, as $ p \to +\infty $, of the eigenpairs~$ \left( v_{p, \ell} , { \lambda }_{p, \ell}\right)\left(\ell \in \mathbb{N}\right)$, of the following linearized problem
\begin{equation}\label{MI:eq:eig}
\begin{cases}
\Delta^2 v
=
\lambda \, p(u_p^+)^{p-1}v
&\text{in }\Omega,\\
v=\partial_\nu v=0
&\text{on }\partial\Omega.
\end{cases}
\end{equation}

To do this, for any \( j = 1 , \ldots , k \), we set
\begin{equation}\label{eq:rese}
    V_{p,\ell,j} (z) := v_{p,\ell} \left( x_{p,j} + {\varepsilon }_{p,j} z \right), \hspace{3em} z \in { \Omega }_{p,j} := \left\{ z : x_{p,j} + { \varepsilon }_{p,j} z \in \Omega \right\}
\end{equation}
It is easy to check that the eigenpairs $ \left( V_{p,\ell,j} , {\lambda }_{p,\ell} \right)\left( \ell \in \mathbb{N} \right) $ satisfy the following eigenvalue problem
\begin{equation*}
\begin{cases}
\Delta^2 v
=
\lambda W_{p,j} v
&\text{in }\Omega_{p,j},\\
v=\partial_\nu v=0
&\text{on }\partial\Omega_{p,j},
\end{cases}
\end{equation*}
where
\begin{equation*}
W_{p,j}(z):=\left[ \left( 1 + \frac{ Z_{p,j} (z) }{ p } \right)^+ \right]^{p-1}
\end{equation*}
is the rescaled linearized weight and $ Z_{p,j} (z) $ is the rescaled function defined in \eqref{defwpj}.
By Lemma~\ref{llma}, there exists \(\sigma\in(0,1)\) such that, for every sufficiently small fixed \(d>0\) and all sufficiently large \(p\),
\begin{equation}\label{MI:weight-tail}
0\le W_{p,j}(z)\le \frac{C}{(1+|z|)^{8-\sigma}}, \quad \text{for } | z| \le \frac{d}{\varepsilon_{p,j}},
\quad
\text{and }
W_{p,j}\longrightarrow e^Z \quad\text{locally uniformly in }\mathbb R^4.
\end{equation}

A point which has to be kept explicit in the fourth-order problem is that
$u_p$ is not assumed to be positive in $\Omega$.  This is different from
the planar Lane--Emden problem and affects the global spectral argument.
Nevertheless, the concentration cores are eventually positive.  Indeed,
for every fixed $R>0$, the local convergence $Z_{p,j}\to Z$ implies
\[
1+\frac{Z_{p,j}(z)}p>0
\qquad\text{for }|z|\le R,
\]
for all sufficiently large $p$, and hence
\[
u_p(x_{p,j}+\e_{p,j}z)>0
\qquad\text{for }|z|\le R.
\]
Thus the positive-part truncation is inactive on every fixed bubble core,
while it must be retained globally.  In particular, the local linearized
weight still converges to $e^Z$, whereas one must still exclude the possibility that the eigenfunctions vanish on every fixed bubble core while remaining nontrivial away from the concentration points, possibly in the region where $u_p\le 0$. This is achieved by a separate Green representation argument in Proposition~\ref{MI:prop:tube-capture} below.

We keep the notation introduced above and set
\[
\e_p:=\max_{1\le j\le k}\e_{p,j}.
\]

\subsection{Spectral formulation of the Morse index}\label{sec1}$\,$  \vskip 0.2cm

Since $p>1$, the map $t\mapsto (t^+)^{p+1}$ is of class $C^2$ and the
energy functional associated with \eqref{1.1} is
\begin{equation*}
J_p(u)
:=
\frac12\int_\Omega |\Delta u|^2
-
\frac1{p+1}\int_\Omega (u^+)^{p+1},
\qquad u\in H_0^2(\Omega).
\end{equation*}
At a solution $u_p$, its second variation is the quadratic form
\begin{equation*}
\mathcal Q_p(v)
=
J_p''(u_p)[v,v]
=
\int_\Omega|\Delta v|^2
-p\int_\Omega (u_p^+)^{p-1}v^2.
\end{equation*}
Notice that the weight $(u_p^+)^{p-1}$ is merely nonnegative.  In
particular, in contrast with the positive Lane--Emden setting, the
weighted form below is in general only semidefinite.  The following
lemma records the precise spectral formulation needed for the Morse
index and isolates the harmless zero-weight directions.

Define
\[
a(v,w):=\int_\Omega\Delta v\,\Delta w,
\qquad
b_p(v,w)
:=
p\int_\Omega(u_p^+)^{p-1}vw.
\]
The form $a$ is an inner product on $H_0^2(\Omega)$, while $b_p$ is
symmetric and nonnegative but may have a nontrivial kernel.
Let $K_p:H_0^2(\Omega)\to H_0^2(\Omega)$ be defined by
\begin{equation*}
a(K_pv,\phi)=b_p(v,\phi)
\qquad\text{for every }\phi\in H_0^2(\Omega).
\end{equation*}

\begin{Lem}\label{MI:lem:semidefinite-weight}
The operator $K_p$ is compact, nonnegative and self-adjoint with respect
to $a$.  Moreover,
\begin{equation*}
\ker K_p=\ker b_p,
\qquad
\mathcal Q_p(v)=a(v,v)>0
\quad\text{for every }0\ne v\in\ker K_p.
\end{equation*}
If
\[
\mu_{p,1}\ge\mu_{p,2}\ge\cdots>0
\]
are the positive eigenvalues of $K_p$, repeated according to multiplicity,
and
\begin{equation*}
\lambda_{p,\ell}:=\mu_{p,\ell}^{-1},
\end{equation*}
then the $\lambda_{p,\ell}$ are precisely the finite eigenvalues of
\eqref{MI:eq:eig}, ordered as
\[
0<\lambda_{p,1}\le\lambda_{p,2}\le\cdots,
\]
and
\begin{equation}\label{MI:eq:Morse-vs-lambda}
m(u_p)
=
\#\{\ell:\lambda_{p,\ell}<1\},
\qquad
m_0(u_p)
=
\#\{\ell:\lambda_{p,\ell}\le1\}.
\end{equation}
In particular, regions where the linearized weight vanishes do not make any additional contribution to the Morse index.
\end{Lem}

\begin{proof}
Since $u_p$ is fixed and bounded for each $p$, the weight
$p(u_p^+)^{p-1}$ belongs to $L^\infty(\Omega)$.  The compact embedding
$H_0^2(\Omega)\hookrightarrow L^2(\Omega)$ therefore implies that the map
$v\mapsto b_p(v,\cdot)$ is compact from $H_0^2(\Omega)$ into
$(H_0^2(\Omega))^*$.  Composing with the inverse of the Riesz map induced
by $a$ shows that $K_p$ is compact.  By definition,
\[
a(K_pv,w)=b_p(v,w)=b_p(w,v)=a(v,K_pw),
\]
so $K_p$ is self-adjoint with respect to $a$, and
\[
a(K_pv,v)=b_p(v,v)\ge0.
\]
If $K_pv=0$, then $b_p(v,\phi)=0$ for every $\phi$. Conversely, if
$b_p(v,v)=0$, then
$(u_p^+)^{(p-1)/2}v=0$ almost everywhere and hence
$b_p(v,\phi)=0$ for every $\phi$, so $K_pv=0$.  This proves
$\ker K_p=\ker b_p$.  For $0\ne v\in\ker K_p$, we therefore have
\[
\mathcal Q_p(v)=a(v,v)-b_p(v,v)=a(v,v)>0,
\]
which shows that the entire kernel of the weighted form is harmless for
the Morse index.

On $(\ker K_p)^{\perp_a}$ the compact self-adjoint spectral theorem gives
an $a$-orthonormal basis of eigenvectors $e_{p,\ell}$ with positive
eigenvalues $\mu_{p,\ell}$.  Since
\[
b_p(e_{p,\ell},\phi)
=a(K_pe_{p,\ell},\phi)
=\mu_{p,\ell}a(e_{p,\ell},\phi),
\]
$e_{p,\ell}$ solves \eqref{MI:eq:eig} with
$\lambda_{p,\ell}=1/\mu_{p,\ell}$.  Furthermore,
\begin{equation*}
\mathcal Q_p(e_{p,\ell})
=
(1-\mu_{p,\ell})a(e_{p,\ell},e_{p,\ell}).
\end{equation*}
Thus a negative direction corresponds exactly to
$\mu_{p,\ell}>1$, equivalently $\lambda_{p,\ell}<1$, while a nonpositive
direction corresponds to $\mu_{p,\ell}\ge1$, equivalently
$\lambda_{p,\ell}\le1$.  Together with the strict positivity on
$\ker K_p$, this proves \eqref{MI:eq:Morse-vs-lambda}.
\end{proof}

\begin{Rem}
The possible sign change of $u_p$ does not itself generate extra negative
directions.  On the set where $u_p\le0$ the linearized potential
$p(u_p^+)^{p-1}$ vanishes, and the quadratic form reduces there to the
positive Dirichlet biharmonic energy.  The real additional difficulty is
global: the positivity and maximum principle arguments available for the
second-order Lane--Emden problem cannot be used to force a low
linearized eigenfunction to remain near the concentration set.  This is
why the spectral capture result in
Subsection~\ref{MI:sec:capture} is indispensable in the present
fourth-order setting.
\end{Rem}

For every finite eigenvalue, $b_p(v,v)>0$.  We choose eigenfunctions
$v_{p,\ell}$ which are mutually $b_p$-orthonormal:
\begin{equation*}
b_p(v_{p,\ell},v_{p,m})=\delta_{\ell m}.
\end{equation*}
When a local Green or blow-up estimate is used below, we may multiply
each individual eigenfunction by a nonzero scalar and impose
$\|v_{p,\ell}\|_{L^\infty(\Omega)}=1$.
This does not change its eigenvalue and preserves the mutual
$b_p$-orthogonality for distinct branches.  We use this harmless
renormalization implicitly in the pointwise outer expansions.

\subsection{The bounded low spectrum of a single bubble}$\,$  \vskip 0.2cm

For the Morse-index argument we only need the limit eigenvalue problem in the
\emph{bounded} class.  This restriction is essential.  If one only assumes
subquadratic growth, then at \(\Lambda=0\) the functions \(1,z_1,\ldots,z_4\)
are biharmonic, so the zero eigenspace is not one-dimensional.

Consider
\begin{equation}\label{MI:eq:limit-eig}
\Delta^2V=\Lambda e^ZV
\qquad\text{in }\R^4,
\qquad \|V\|_{L^\infty(\R^4)}<\infty.
\end{equation}
To analyze the bounded solutions of \eqref{MI:eq:limit-eig}, we shall transfer the problem to \(S^4\) by stereographic projection. Since the point at infinity in \(\mathbb R^4\) corresponds to the north pole \(N\in S^4\), we first record the following removable-singularity result.

\begin{Lem}\label{MI:lem:removable}
Let $c\in\R$ and let $U$ be a bounded distributional solution of
\[
(P_{\mathbb S^4}-c)U=0
\qquad\mbox{in }B_\rho(N)\setminus\{N\}.
\]
Then the singularity at $N$ is removable and $U$ extends to a smooth solution in $B_\rho(N)$.
\end{Lem}
\begin{proof}
Extend $U$ as a bounded distribution across $N$ and set $T:=(P_{\mathbb S^4}-c)U$. Then $\operatorname{supp}T\subset\{N\}$. By the structure theorem for distributions supported at one point, $T$ is a finite linear combination of derivatives of $\delta_N$. A local parametrix for the fourth-order uniformly elliptic operator $P_{\mathbb S^4}-c$ has a logarithmic fundamental singularity in dimension four, while derivatives of a point mass generate stronger singularities. Thus a nonzero point supported defect would contradict the boundedness of $U$ near $N$. Hence $T=0$, and elliptic regularity gives the desired smooth extension.
\end{proof}

The bounded kernel corresponding to \(\Lambda=1\) was characterized in Baraket–Dammak–Ouni–Pacard \cite[Lemma 3.1]{Baraket}. Using the same stereographic projection argument and the standard spectrum of the Paneitz operator on \(S^4\), one obtains the following slightly more complete bounded spectral classification.

\begin{Prop}\label{MI:prop:bubblespec}
Let \(V\not\equiv0\) be a bounded solution of \eqref{MI:eq:limit-eig} with
\(0\le\Lambda<5\).  Then
\[
\Lambda\in\{0,1\}.
\]
More precisely,
\begin{enumerate}[(i)]
\item if \(\Lambda=0\), then \(V\) is constant;
\item if \(\Lambda=1\), then
\begin{equation*}
V(z)=\beta_0Y_0(z)+\sum_{i=1}^4\beta_iY_i(z),
\end{equation*}
where
\begin{equation}\label{eq:TZJXL}
    Y_0(z):=\frac{a-|z|^2}{a+|z|^2},
\qquad
Y_i(z):=\frac{z_i}{a+|z|^2},\quad i=1,\dots,4,
\qquad a=8\sqrt6.
\end{equation}
\end{enumerate}
The next eigenvalue in the bounded class is \(5\).
\end{Prop}

\begin{proof}
Set \(z=\sqrt a\,y\) and \(\Phi(y):=V(\sqrt a\,y)\).  Since
\(a^2=384\), equation \eqref{MI:eq:limit-eig} becomes
\[
\Delta^2\Phi
=
\frac{384\Lambda}{(1+|y|^2)^4}\Phi.
\]
Let \(\mathcal S:\R^4\to\mathbb S^4\setminus\{N\}\) be inverse
stereographic projection and define
\(\widehat\Phi(\mathcal S(y))=\Phi(y)\). By the conformal covariance of the Paneitz operator in dimension four,
\begin{equation}\label{MI:eq:sphere-general}
P_{\mathbb S^4}\widehat\Phi
=24\Lambda\widehat\Phi
\qquad\text{on }\mathbb S^4\setminus\{N\},
\end{equation}
where
\[
P_{\mathbb S^4}=\Delta_{\mathbb S^4}^2-2\Delta_{\mathbb S^4}.
\]
Because $V$ is bounded, $\widehat\Phi$ is bounded near $N$. By Lemma~\ref{MI:lem:removable}, the north-pole singularity is removable; hence \eqref{MI:eq:sphere-general} holds on all of $\mathbb S^4$ and $\widehat\Phi\in C^\infty(\mathbb S^4)$.

If \(\mathcal Y_\ell\) is a spherical harmonic of degree \(\ell\), then
\[
-\Delta_{\mathbb S^4}\mathcal Y_\ell
=\mu_\ell\mathcal Y_\ell,
\qquad
\mu_\ell=\ell(\ell+3),
\]
and therefore
\[
P_{\mathbb S^4}\mathcal Y_\ell
=\mu_\ell(\mu_\ell+2)\mathcal Y_\ell.
\]
Consequently, the bounded eigenvalues are
\begin{equation*}
\Lambda_\ell
=
\frac{\mu_\ell(\mu_\ell+2)}{24},
\qquad \ell=0,1,2,\dots.
\end{equation*}
Since \(\mu_\ell\) is strictly increasing,
\[
\Lambda_0=0,
\qquad
\Lambda_1=1,
\qquad
\Lambda_2=5,
\]
and hence \(0\le\Lambda<5\) implies \(\Lambda\in\{0,1\}\). The eigenspace corresponding to \(\Lambda_0=0\) consists of constants, whereas the eigenspace corresponding to \(\Lambda_1=1\) is five-dimensional and is spanned, after pullback by stereographic projection, by \(Y_0,Y_1,\ldots,Y_4\).
\end{proof}

\subsection{Bubble-scale capture of eigenfunctions with bounded eigenvalues}\label{MI:sec:capture}$\,$  \vskip 0.2cm

This subsection is devoted to prove Proposition \ref{MI:prop:tube-capture}, which is the fourth-order counterpart of the corresponding non-vanishing result for the planar second-order Lane–Emden problem, where the proof relies on a maximum principle argument. Since no such maximum principle is available for the Dirichlet biharmonic operator on a general domain, we instead use a Green representation argument.

We first record a logarithmic oscillation estimate.

\begin{Lem}\label{MI:lem:green-osc}
Fix \(\sigma\in(0,1)\).  There exists a function \(\omega(R)\to 0\)
as \(R\to\infty\) such that, uniformly for all sufficiently large \(p\),
all \(j=1,\dots,k\), and all \(x\in\Omega\) satisfying
\(|x-x_{p,j}|\ge R\e_{p,j}\),
\begin{equation*}
\int_{B_{d/\e_{p,j}}(0)}
\big|
G(x_{p,j}+\e_{p,j}z,x)-G(x_{p,j},x)
\big|
\frac{dz}{(1+|z|)^{8-\sigma}}
\le \omega(R)+C\e_{p,j}.
\end{equation*}
One may take \(\omega(R)=C(R^{-1}+R^{-4+\sigma})\).
\end{Lem}

\begin{proof}
Write
\[
G(y,x)=-\frac1{8\pi^2}\log|y-x|+H(y,x).
\]
Since all spike balls are compactly contained in $\Omega$, standard interior regularity for the regular part of the Green function gives a uniform bound for $\nabla_yH$ on the relevant compact support set, uniformly for the second variable in $\overline\Omega$. Consequently,
\[
|H(x_{p,j}+\e_{p,j}z,x)-H(x_{p,j},x)|
\le C\e_{p,j}|z|.
\]
The weighted first moment is finite, so this part contributes
\(O(\e_{p,j})\).

Put
\[
t:=\frac{|x-x_{p,j}|}{\e_{p,j}}\ge R,
\qquad
\theta:=\frac{x-x_{p,j}}{|x-x_{p,j}|}.
\]
For the singular part we have to estimate
\[
\int
\left|
\log|t\theta-z|-\log t
\right|
(1+|z|)^{-8+\sigma}dz.
\]
On \(|z|\le t/2\),
\[
\left|\log|\theta-z/t|\right|
\le C\frac{|z|}{t},
\]
and hence this contribution is \(O(t^{-1})\).  On \(|z|>t/2\), set
\(z=t\eta\).  Since the logarithmic singularity at \(\eta=\theta\) is
locally integrable in four dimensions and
\(|\eta|^{-8+\sigma}\log(2+|\eta|)\in L^1(\{|\eta|>1/2\})\), we obtain
\[
\int_{|z|>t/2}
\left|\log|\theta-z/t|\right|
(1+|z|)^{-8+\sigma}dz
\le Ct^{-4+\sigma}.
\]
This proves the claim.
\end{proof}

Returning to the eigenvalue problem \eqref{MI:eq:eig}, we may normalize the corresponding eigenfunction so that
\begin{equation}\label{eq:GYH}
    \| v_{p,\ell} \|_{L^{\infty} (\Omega) } = 1 .
\end{equation}
Then, by the preceding lemma, we obtain the following result.

\begin{Prop}\label{MI:prop:tube-capture}
Let \( V_{p,\ell,j} \), \( \ell \in \mathbb{N} \), be the rescaled eigenfunction defined in \eqref{eq:rese}. If \(\lambda_{p,\ell}\) is bounded, after passing to a subsequence,
\[
\lambda_{p,\ell}\to\Lambda\ge0
\]
and, for every \(j=1,\ldots,k\),
\[
V_{p,\ell,j}\to V_{\ell,j}
\qquad\text{in }C^3_{\mathrm{loc}}(\R^4),
\]
where
\begin{equation*}
\Delta^2V_{\ell,j}=\Lambda e^ZV_{\ell,j} \hspace{1em} in \hspace{0.3em} {\mathbb{R}}^4, \hspace{1em} \|V_{\ell,j}\|_{L^{\infty}({\mathbb{R}}^4)}\le1.
\end{equation*}
Moreover,
\begin{equation}\label{MI:eq:nonzero-profile}
(V_{\ell,1},\dots,V_{\ell,k})\ne(0,\dots,0).
\end{equation}
\end{Prop}

\begin{proof}
The local compactness follows immediately from
\[
\Delta^2V_{p,\ell,j}=\lambda_{p,\ell}W_{p,j}V_{p,\ell,j},
\qquad \|V_{p,\ell,j}\|_{L^\infty ({\mathbb{R}}^4) }\le1,
\]
the local convergence \(W_{p,j}\to e^Z\), and standard biharmonic
interior estimates.

It remains to prove \eqref{MI:eq:nonzero-profile}.  Suppose, to the contrary,
that \(V_{\ell,j}\equiv0\) for every \(j\).  Define
\begin{equation*}
\mathcal{A}_{p,\ell,j}:=
\int_{B_{d/\e_{p,j}}(0)}W_{p,j}(z)V_{p,\ell,j}(z)\,dz.
\end{equation*}
By the uniform decay \eqref{MI:weight-tail}, local convergence to zero,
and dominated convergence,
\begin{equation}\label{MI:eq:Aj-to-zero}
\mathcal{A}_{p,\ell,j}\to0,
\qquad
\int_{B_{d/\e_{p,j}}(0)} \log|z|\,W_{p,j}V_{p,\ell,j}\,dz\to0.
\end{equation}
The Green representation reads, uniformly in \(x\in\Omega\),
\begin{equation}\label{MI:eq:exact-green-capture}
v_{p,\ell}(x)
=
\lambda_{p,\ell}\sum_{m=1}^k
\int_{B_{d/\e_{p,m}}(0)}
G(x_{p,m}+\e_{p,m}z,x)
W_{p,m}(z)V_{p,\ell,m}(z)\,dz
+o(1).
\end{equation}
Indeed, outside the spike balls one has \(u_p=O(p^{-1})\), while
\(\sup_x\int_\Omega|G(y,x)|dy<\infty\), so that the omitted contribution
is super-polynomially small.

Evaluate \eqref{MI:eq:exact-green-capture} at \(x=x_{p,j}\).  The terms
coming from \(m\ne j\) tend to zero by dominated convergence.  For the
self-interaction,
\[
G(x_{p,j}+\e_{p,j}z,x_{p,j})
=
-\frac1{8\pi^2}\log\e_{p,j}
-\frac1{8\pi^2}\log|z|
+H(x_{p,j}+\e_{p,j}z,x_{p,j}).
\]
Since \(V_{p,\ell,j}(0)\to0\), \eqref{MI:eq:Aj-to-zero} gives
\begin{equation}\label{MI:eq:lambda-p-A}
\lambda_{p,\ell}|\log\e_{p,j}|\mathcal{A}_{p,\ell,j}\to0.
\end{equation}
Using \(|\log\e_{p,j}|=p/8+O(1)\), we obtain
\begin{equation*}
\lambda_{p,\ell} p \mathcal{A}_{p,\ell,j}\to0
\qquad\text{for every }j.
\end{equation*}

Fix \(R\gg1\) and take
\[
x\in\Omega\setminus\bigcup_{m=1}^k
B_{R\e_{p,m}}(x_{p,m}).
\]
For each spike term in \eqref{MI:eq:exact-green-capture}, we set
\begin{equation*}
I_{p,\ell,m}(x) := \int_{B_{d/\varepsilon_{p,m}}(0)} G(x_{p,m}+\varepsilon_{p,m}z,x) \times W_{p,m}(z)V_{p,\ell,m}(z)\,dz.
\end{equation*}
Adding and subtracting \(G(x_{p,m},x)\mathcal A_{p,\ell,m}\), we obtain
\[ I_{p,\ell,m}(x) = G(x_{p,m},x)\mathcal A_{p,\ell,m} + R_{p,\ell,m}(x), \]
where
\begin{equation*}
R_{p,\ell,m}(x):=\int_{B_{d/\varepsilon_{p,m}}(0)}\Big[G(x_{p,m}+\varepsilon_{p,m}z,x)-G(x_{p,m},x)\Big]\times W_{p,m}(z)V_{p,\ell,m}(z)\,dz.
\end{equation*}
Since
\[
|G(x_{p,m},x)|
\le C\left(1+|\log\e_{p,m}|\right)
\]
in this region, \eqref{MI:eq:lambda-p-A} shows that \(\lambda_{p,\ell} G(x_{p,m},x)\mathcal A_{p,\ell,m} = o(1)\).  Lemma~\ref{MI:lem:green-osc}, \(\|V_{p,\ell,m}\|_{L^{\infty} ({\mathbb{R}}^4)}\le1\) and \eqref{MI:weight-tail} give
\[
\sup_{x\in\Omega\setminus\bigcup_{m=1}^k
B_{R\e_{p,m}}(x_{p,m})}
|v_{p,\ell}(x)|
\le o(1)+C\omega(R).
\]
On the other hand, for fixed \(R\), local convergence gives
\[
\max_{1\le m \le k}\|v_{p,\ell}\|_{L^\infty(B_{R\e_{p,m}}(x_{p,m}))}
=o(1).
\]
First let \(p\to\infty\), then \(R\to\infty\).  We obtain
\(\|v_{p,\ell}\|_{L^\infty (\Omega) }\to0\), contradicting \eqref{eq:GYH}.
\end{proof}

\begin{Cor}\label{MI:cor:captured-class}
For \( \ell \in \mathbb{N} \). If \(\lambda_{p,\ell}\to\Lambda\le1\), then for any \( j =1,\dots,k\),
\begin{enumerate}[(i)]
\item if \(\Lambda<1\), necessarily \(\Lambda=0\) and
\[
V_{\ell,j}\equiv c_{\ell,j},
\qquad \mathbf{c_{\ell}}=(c_{\ell,1},\dots,c_{\ell,k})\ne \mathbf{0};
\]
\item if \(\Lambda=1\), then
\begin{equation*}
V_{\ell,j}
=
\beta_{0,\ell,j}Y_0+
\sum_{i=1}^4\beta_{i,\ell,j}Y_i,
\end{equation*}
with at least one nonzero coefficient among all \(5k\) coefficients, where \( Y_{i} (i=0,1,\dots,4 ) \) defined in \eqref{eq:TZJXL}.
\end{enumerate}
\end{Cor}

\begin{proof}
Combine Proposition~\ref{MI:prop:tube-capture} with
Proposition~\ref{MI:prop:bubblespec}.
\end{proof}

\subsection{Auxiliary estimates} $\,$  \vskip 0.2cm

We first establish a uniform gradient estimate for the rescaled profiles. We then derive an outer expansion for eigenfunctions with bounded eigenvalues. And determine the leading term of bounded eigenvalues corresponding to the eigenfunctions with a nonzero scale component.

\begin{Lem}\label{MI:lem:gradient-tail}
For every sufficiently small fixed $d>0$, there exists $C>0$, independent of $p$ and $j$, such that
\begin{equation}\label{MI:eq:gradient-tail}
 |\nabla Z_{p,j}(z)|\le \frac{C}{1+|z|}+C\e_{p,j},
 \qquad |z|\le \frac{d}{\e_{p,j}},\quad j=1,\ldots,k.
\end{equation}
Consequently,
\begin{equation}\label{MI:eq:radial-gradient-tail}
 |z\cdot\nabla Z_{p,j}(z)|\le C
 \qquad\text{for } |z|\le \frac{d}{\e_{p,j}}.
\end{equation}
\end{Lem}

\begin{proof}
Choose $d>0$ sufficiently small that the balls $B_{2d}(x_{p,j})$ are pairwise disjoint and compactly contained in $\Omega$ for all sufficiently large $p$, and that Lemma~\ref{llma} applies with $2d$ in place of $d$. Fix $j$, set
\[
 x=x_{p,j}+\e_{p,j} z,
 \qquad
 F_{p,j}(y):=\left[\left(1+\frac{Z_{p,j}(y)}p\right)^+\right]^p.
\]
Differentiating the Green representation gives
\[
 \nabla Z_{p,j}(z)
 =\frac{p\e_{p,j}}{m_{p,j}}\int_\Omega
 \nabla_xG(y,x)(u_p^+)^p(y)\,dy.
\]
By H\"older's inequality and \eqref{1.2},
$\int_\Omega(u_p^+)^p=O(p^{-1})$, while $m_{p,j}\to\sqrt e$.
For $|z|\le d/\e_{p,j}$, the point $x$ lies in $B_d(x_{p,j})$ and is at distance at least $d$ from $\Omega\setminus B_{2d}(x_{p,j})$. Thus Lemma~\ref{Greenfun} bounds the contribution of this exterior region by $C\e_{p,j}$. The derivative of the regular part $H(y,x)$ is uniformly bounded for $y\in B_{2d}(x_{p,j})$ and $x\in B_d(x_{p,j})$, so its contribution is also $O(\e_{p,j})$. Using
$G(y,x)=-(8\pi^2)^{-1}\log|x-y|+H(y,x)$ and scaling the remaining integral, we obtain
\[
 \nabla Z_{p,j}(z)
 =-\frac1{8\pi^2}
 \int_{B_{2d/\e_{p,j}}(0)}
 \frac{z-y}{|z-y|^2}F_{p,j}(y)\,dy
 +O(\e_{p,j}),
\]
uniformly for $|z|\le d/\e_{p,j}$. Here we used $m_{p,j}^{p-1}\e_{p,j}^4=p^{-1}$.

Lemma~\ref{llma} and local convergence imply, for a fixed $\sigma\in(0,1)$,
\[
 0\le F_{p,j}(y)\le C(1+|y|)^{-8+\sigma}
 \qquad\text{for } |y|\le 2d/\e_{p,j}.
\]
We claim that
\[
 \int_{\mathbb R^4}
 \frac{(1+|y|)^{-8+\sigma}}{|z-y|}\,dy
 \le \frac{C}{1+|z|}.
\]
For $|z|\le2$, this follows from the local integrability of $|z-y|^{-1}$ and the integrability of the weight. For $r=|z|>2$, split the integral into
\[
 \{|y|\le r/2\},\qquad
 \{|y|>r/2,\ |y-z|\le r/2\},\qquad
 \{|y|>r/2,\ |y-z|>r/2\}.
\]
The first and third regions contribute at most $C/r$. On the second region, the weight is bounded by $Cr^{-8+\sigma}$, and the integral of $|z-y|^{-1}$ over a ball of radius $r/2$ is $O(r^3)$. Its contribution is therefore $O(r^{-5+\sigma})\le C/r$. This proves the claim and hence \eqref{MI:eq:gradient-tail}. Finally, $\e_{p,j}|z|\le d$ on the stated domain, which gives \eqref{MI:eq:radial-gradient-tail}.
\end{proof}

For \( \ell \in \mathbb{N} \), define
\begin{equation}\label{MI:eq:A-eig}
A_{p,j,\ell}
:=
p\int_{B_d(x_{p,j})}(u_p^+)^{p-1}(y)v_{p,\ell}(y) \,dy,
\end{equation}
and
\begin{equation*}
B_{p,i,j,\ell}
:=
\frac{p}{\e_{p,j}}
\int_{B_d(x_{p,j})}
(y_i-x_{p,j,i})(u_p^+)^{p-1}(y)v_{p,\ell}(y) \,dy.
\end{equation*}
For an \(L^\infty\)-normalized eigenfunction, the quantities
\(B_{p,i,j,\ell}\) are uniformly bounded by the first-moment estimate following
from \eqref{MI:weight-tail}.

\begin{Lem}\label{MI:lem:eig-outer}
For \( \ell \in \mathbb{N}\), if \(\lambda_{p,\ell}=O(1)\) and \(\|v_{p,\ell}\|_{L^\infty (\Omega)}=1\), then
\begin{equation}\label{MI:eq:eig-outer}
v_{p,\ell}(x)
=
\lambda_{p,\ell}\sum_{j=1}^kA_{p,j,\ell} G(x_{p,j},x)
+
\lambda_{p,\ell}\sum_{j=1}^k\sum_{i=1}^4
\e_{p,j}B_{p,i,j,\ell}
\partial_iG(x_{p,j},x)
+O(\e_p^2),
\end{equation}
in \(C^3\Big(\Omega\backslash \displaystyle\bigcup^k_{j=1} B_{2d}(x_{p,j})\Big)\).
If
\[
V_{p,\ell,j}\to
\beta_{0,\ell,j}Y_0+
\sum_{h=1}^4\beta_{h,\ell,j}Y_h,
\]
then
\begin{equation}\label{MI:eq:B-limit}
B_{p,i,j,\ell}\to8\pi^2\beta_{i,\ell,j}.
\end{equation}
\end{Lem}

\begin{proof}
Since \(v_{p,\ell}\) satisfies
\[
\Delta^2v_{p,\ell}
=
\lambda_{p,\ell} p (u_p^+)^{p-1}v_{p,\ell}
\qquad\text{in }\Omega,
\]
the Green representation yields that for every multi-index \( | \alpha | \le 3 \),
\begin{align*}
    D_x^\alpha v_{p,\ell}(x)
=&
\lambda_{p,\ell}p
\int_\Omega D_x^\alpha G(y,x)(u_p^+)^{p-1}(y)v_{p,\ell}(y)\,dy \\
=&\lambda_{p,\ell}p
\sum_{j=1}^{k}\int_{B_d (x_{p,j})} D_x^\alpha G(y,x)(u_p^+)^{p-1}(y)v_{p,\ell}(y)\,dy \\
&+ \lambda_{p,\ell}p \int_{\Omega\backslash \displaystyle\bigcup^k_{j=1} B_{d}(x_{p,j})} D_x^\alpha G(y,x)(u_p^+)^{p-1}(y)v_{p,\ell}(y)\,dy
\end{align*}
uniformly for \(x\in\Omega\setminus\bigcup_{j=1}^kB_{2d}(x_{p,j})\). By \eqref{11-14-03N} and the integrability of $D_x^\alpha G(y,x)$ for $|\alpha|\le3$, the contribution outside the spike balls is
\[
 O\!\left(p(C/p)^{p-1}\right)=o(\e_p^2),
\]
where the last equality follows from Theorem~\ref{th1}.
On \(B_d(x_{p,j})\), the variables \(x\) and \(y\) stay uniformly separated, and \(D_x^\alpha G(y,x)\) is smooth with respect to the source variable \(y\). Taylor expansion at \(y=x_{p,j}\) gives
\[
D_x^\alpha G(y,x)
=
D_x^\alpha G(x_{p,j},x)
+
\sum_{i=1}^4
(y_i-x_{p,j,i})
D_x^\alpha\partial_iG(x_{p,j},x)
+O(|y-x_{p,j}|^2).
\]
After scaling, the remainder is bounded by
\[
C\e_{p,j}^2
\int_{B_{d/{{\varepsilon}_{p,j}}}(0)} |z|^2W_{p,j}(z)|V_{p,\ell,j}(z)|\,dz
=O(\e_{p,j}^2),
\]
which proves \eqref{MI:eq:eig-outer}.
Finally,
\[
B_{p,i,j,\ell}
=
\int_{B_{d/{{\varepsilon}_{p,j}}}(0)} z_iW_{p,j}(z)V_{p,\ell,j}(z)\,dz,
\]
and the dominated convergence theorem gives
\[
B_{p,i,j,\ell}
\longrightarrow
\int_{\mathbb R^4}
z_i e^Z
\left(
\beta_{0,\ell,j}Y_0
+
\sum_{h=1}^4\beta_{h,\ell,j}Y_h
\right)\,dz.
\]
Since both \(e^Z\) and \(Y_0\) are radial,
\[
\int_{\mathbb R^4}
z_i e^ZY_0\,dz
=
0
\]
by oddness. Likewise, if \(h\ne i\),
\[
\int_{\mathbb R^4}
z_i e^ZY_h\,dz
=
\int_{\mathbb R^4}
\frac{z_iz_h}{a+|z|^2}e^Z\,dz
=
0,
\]
again by parity. Therefore only the \(h=i\) term survives, and
\[
B_{p,i,j,\ell}
\longrightarrow
\beta_{i,\ell,j}
\int_{\mathbb R^4}
z_i e^ZY_i\,dz
=
\beta_{i,\ell,j}
\int_{\mathbb R^4}
\frac{z_i^2e^Z}{a+|z|^2}\,dz.
\]
Using
\[
\int_{\mathbb R^4}
\frac{z_i^2e^Z}{a+|z|^2}\,dz
=
8\pi^2,
\]
we obtain \eqref{MI:eq:B-limit}.
\end{proof}

\begin{Prop}\label{MI:prop:scale}
Let \(\ell \in \mathbb{N} \). Assume \(\lambda_{p,\ell}\to1\), \(\|v_{p,\ell}\|_{L^\infty (\Omega)}=1\), and for every \(j=1,\dots,k\),
\[
V_{p,\ell,j}\to
\beta_{0,\ell,j}Y_0+
\sum_{i=1}^4\beta_{i,\ell,j}Y_i.
\]
If \(\beta_{0,\ell,j_0}\ne0\) for at least one index
\(j_0\), then it holds
\begin{equation}\label{MI:eq:scale-eigen}
\lambda_{p,\ell}
=
1+\frac{10}{p}\bigl(1+o(1)\bigr) \quad \text{as} \hspace{0.5em} p \to +\infty.
\end{equation}
\end{Prop}

\begin{proof}
Since
\[
\int_{\R^4}e^ZY_0 \,dz =0,
\qquad
\int_{\R^4}e^ZY_i \,dz =0,
\]
by \eqref{MI:eq:A-eig} and the dominated convergence theorem, we have
\begin{equation*}
A_{p,j,\ell}\to0,
\quad j=1,\dots,k.
\end{equation*}
Write
\[
I_{p,j,\ell}:=
\int_{B_d(x_{p,j})}(u_p^+)^p (y)v_{p,\ell}(y) \,dy.
\]
From the Green concomitant identity \eqref{MI:eq:local-I},
\[
(\lambda_{p,\ell}p-1)I_{p,j,\ell}=\mathfrak B_j(u_p,v_{p,\ell}).
\]
On \(\partial B_d(x_{p,j})\), Lemmas~\ref{prop:expansionupwithdelta} and \ref{MI:lem:eig-outer} imply that for every multi-index \( | \alpha | \le 3 \)
\[
|D^\alpha u_p|=O(p^{-1}),
\qquad
|D^\alpha v_{p,\ell}|=O(A_p^*+\e_p),
\]
where \(A_p^*=\max_{1\le j \le k}|A_{p,j,\ell}|=o(1)\). Hence,
\begin{equation}\label{MI:eq:I-scale-small}
I_{p,j,\ell}
=
O\left(\frac{A_p^*+\e_p}{p^2}\right)
=o(p^{-2}).
\end{equation}
A direct calculation shows that
\begin{equation*}
m_{p,j}A_{p,j,\ell}-pI_{p,j,\ell}
=
p\int_{B_d(x_{p,j})}
(u_p^+)^{p-1}(m_{p,j}-u_p)v_{p,\ell} \,dy
=
-\frac{m_{p,j}}p
\int_{B_{d/\e_{p,j}}(0)}
Z_{p,j}W_{p,j}V_{p,\ell,j} \,dz.
\end{equation*}
Hence, by dominated convergence and parity,
\[
m_{p,j}A_{p,j,\ell}-pI_{p,j,\ell}
=
-\frac{64\pi^2\sqrt e}{p}
\bigl(\beta_{0,\ell,j}+o(1)\bigr),
\]
because
\[
\int_{\R^4}Ze^ZY_0 \,dz =64\pi^2.
\]
Therefore
\begin{equation*}
A_{p,j,\ell}
=
-\frac{64\pi^2}{p}
\bigl(\beta_{0,\ell,j}+o(1)\bigr).
\end{equation*}

Insert the estimates \eqref{MI:eq:eig-outer} and \eqref{dluo-1} for \(v_{p,\ell}\) and \(u_p\), respectively, into \(P_j(v_{p,\ell},u_p)\).
Since
\[
P_j(G(x_{p,j},x),G(x_{p,j},x))=\frac1{8\pi^2},
\qquad
P_j(G(x_{p,s},x),G(x_{p,m},x))=0
\quad\text{if }(s,m)\ne(j,j),
\]
using \eqref{3.553.55}, we find
\begin{equation}\label{MI:eq:P-scale-left}
P_j(v_{p,\ell},u_p)
=
-\frac{512\pi^2\sqrt e}{p^2}
\bigl(\beta_{0,\ell,j}+o(1)\bigr).
\end{equation}
On the other hand, the generalized identity \eqref{MI:eq:P-general} gives
\begin{equation}\label{eq:p123}
    P_j(v_{p,\ell},u_p)
=
4I_{p,j,\ell}+o(p^{-2})
-(\lambda_{p,\ell}-1)pJ_{p,j,\ell},
\end{equation}
where
\[
J_{p,j,\ell}
:=
\int_{B_d(x_{p,j})}
(u_p^+)^{p-1}v_{p,\ell}
(x-x_{p,j})\cdot\nabla u_p \,dx.
\]
By \eqref{MI:weight-tail}, \eqref{MI:eq:radial-gradient-tail} and $|V_{p,\ell,j}|\le1$, the rescaled integrand satisfies
\[
 \big|W_{p,j}(z)V_{p,\ell,j}(z)\,z\cdot\nabla Z_{p,j}(z)\big|
 \le C(1+|z|)^{-8+\sigma}
 \qquad\text{for } |z|\le d/\e_{p,j}.
\]
This is an integrable bound independent of $p$. Scaling at the spike and applying dominated convergence, using the local convergence of $Z_{p,j}$ and $V_{p,\ell,j}$, therefore gives
\[
J_{p,j,\ell}
=
\frac{m_{p,j}}{p^2}
\left[
\int_{\R^4}e^Z
\left(\beta_{0,\ell,j}Y_0+
\sum_{i=1}^4\beta_{i,\ell,j}Y_i\right)
z\cdot\nabla Z
+o(1)
\right].
\]
The translation terms vanish by parity and
\begin{equation*}
\int_{\R^4}e^ZY_0\,z\cdot\nabla Z \,dz
=
\frac{256\pi^2}{5}.
\end{equation*}
Thus
\begin{equation}\label{MI:eq:pJscale}
pJ_{p,j,\ell}
=
\frac{\sqrt e}{p}
\frac{256\pi^2}{5}
\bigl(\beta_{0,\ell,j}+o(1)\bigr).
\end{equation}
Choose \(j=j_0\) with \(\beta_{0,\ell,j_0}\ne0\). Equations
\eqref{MI:eq:I-scale-small}, \eqref{MI:eq:P-scale-left}, \eqref{eq:p123} and
\eqref{MI:eq:pJscale} first imply that
\[|\lambda_{p,\ell}-1|=O(p^{-1}),\]
and then
\[
\frac{512\pi^2}{p^2}
=
(\lambda_{p,\ell}-1)
\frac{256\pi^2}{5p}(1+o(1)).
\]
This proves \eqref{MI:eq:scale-eigen}.
\end{proof}

\subsection{The estimate of \(\lambda_{p,\ell}\) in the case of \( 1 \le \ell \le k \)}$\,$  \vskip 0.2cm
\begin{Prop}\label{MI:prop:firstk}
For \(1\le\ell\le k\),
\begin{equation*}
\lambda_{p,\ell}
\le\frac1p+O(p^{-2}).
\end{equation*}
Consequently
\[
\lambda_{p,\ell}\to0,
\qquad 1\le\ell\le k,
\]
and
\begin{equation}\label{MI:eq:kplus1-liminf}
\liminf_{p\to\infty}\lambda_{p,k+1}\ge1.
\end{equation}
\end{Prop}

\begin{proof}
Fix \(d>0 \) sufficiently small and choose cutoffs \(\chi_j\)(\(j=1\dots,k\)) satisfying
\[
\chi_j=1\quad\text{in }B_d(x_{p,j}),
\qquad
\chi_j=0\quad\text{outside }B_{2d}(x_{p,j}),
\]
with the supports of \(\chi_j\) pairwise disjoint. Set \(U_{p,j}=\chi_ju_p\). By Lemma~\ref{prop:expansionupwithdelta}, \(|D^\alpha u_p|=O(p^{-1})\) on the cutoff annuli for \(|\alpha|\le3\). Testing the equation of \(u_p\) by \(\chi_j^2u_p\) and expanding the cutoff terms yields
\[
a(U_{p,j},U_{p,j})
=
\int_{\Omega}\chi_j^2(u_p^+)^{p+1} \,dx+O(p^{-2})
=
\frac{64\pi^2e+o(1)}p,
\]
while
\[
b_p(U_{p,j},U_{p,j})
=
p\int_{\Omega}\chi_j^2(u_p^+)^{p+1} \,dx
=
64\pi^2e+o(1).
\]
Let
\[
E_p:=\operatorname{span}\{U_{p,1},\ldots,U_{p,k}\}.
\]
Since the supports of \(U_{p,1},\ldots,U_{p,k}\) are pairwise disjoint, the functions \(U_{p,1},\ldots,U_{p,k}\) are linearly independent, and therefore
\[
\dim E_p=k.
\]
For any
\[
U=\sum_{j=1}^k \alpha_j U_{p,j}\in E_p,
\]
the disjointness of the supports yields
\[
a(U,U)
=
\sum_{j=1}^k \alpha_j^2 a(U_{p,j},U_{p,j})
\quad
\text{and}
\quad
b_p(U,U)
=
\sum_{j=1}^k \alpha_j^2 b_p(U_{p,j},U_{p,j}).
\]
Hence
\[
\frac{a(U,U)}{b_p(U,U)}
\le
\max_{1\le j\le k}
\frac{a(U_{p,j},U_{p,j})}
     {b_p(U_{p,j},U_{p,j})}.
\]
By the estimates obtained above,
\[
\frac{a(U_{p,j},U_{p,j})}
     {b_p(U_{p,j},U_{p,j})}
=
\frac1p+O(p^{-2}),
\qquad j=1,\ldots,k,
\]
uniformly in \(j\). Therefore,
\[
\sup_{0\neq U\in E_p}
\frac{a(U,U)}{b_p(U,U)}
\le
\frac1p+O(p^{-2}).
\]
The min-max characterization of the generalized eigenvalues then gives
\[
\lambda_{p,k}
\le
\sup_{0\neq U\in E_p}
\frac{a(U,U)}{b_p(U,U)}
\le
\frac1p+O(p^{-2}).
\]
Since
\[
0<\lambda_{p,1}\le \lambda_{p,2}\le\cdots\le\lambda_{p,k},
\]
we conclude that
\[
\lambda_{p,\ell}
\le
\frac1p+O(p^{-2}),
\qquad
1\le \ell\le k.
\]
In particular,
\[
\lambda_{p,\ell}\longrightarrow0,
\qquad
1\le \ell\le k.
\]

Let \(v_{p,1},\ldots,v_{p,k}\) be eigenfunctions associated with \(\lambda_{p,1},\ldots,\lambda_{p,k}\), chosen to be mutually \(b_p\)-orthogonal. Since only finitely many eigenvalue branches are involved, after passing to a common diagonal subsequence, we may assume that the bubble-scale limits of all these eigenfunctions exist simultaneously at every concentration point. Multiplying each eigenfunction by a nonzero constant if necessary, we further normalize
\(
\|v_{p,\ell}\|_{L^\infty(\Omega)}=1, \ell=1,\ldots,k.
\)
This rescaling does not affect either the corresponding eigenvalue or the mutual \(b_p\)-orthogonality. By the preceding upper bound and
Corollary~\ref{MI:cor:captured-class}, the local limits are constants
\[
V_{\ell,j}\equiv c_{\ell,j},
\qquad
\mathbf{c_\ell}=(c_{\ell,1},\dots,c_{\ell,k})\ne\mathbf{0}.
\]
For \(\ell\ne m\), the orthogonality
\(b_p(v_{p,\ell},v_{p,m})=0\) becomes, after scaling around all spikes,
\[
0
=
\sum_{j=1}^k
\int_{\R^4}e^Zc_{\ell,j}c_{m,j}+o(1)
=
64\pi^2\,\mathbf{c_\ell}\cdot \mathbf{ c_m }+o(1).
\]
Thus \(\mathbf{ c_1},\dots,\mathbf{c_k}\) are \(k\) nonzero mutually orthogonal vectors in \(\mathbb R^k\).

It remains to prove \eqref{MI:eq:kplus1-liminf}. Suppose, by contradiction, that this fails. Then along a subsequence \(\lambda_{p,k+1}\longrightarrow\Lambda<1\). By Corollary~\ref{MI:cor:captured-class}, \(\Lambda=0\) and the corresponding
constant vector \(\mathbf{c_{k+1}}\ne0\).  Orthogonality to the first \(k\)
eigenfunctions gives \(\mathbf{c_{k+1}}\perp\R^k\), and hence \( \mathbf{c_{k+1}}=0 \), a contradiction.
\end{proof}

\subsection{The estimate of \( \lambda_{p,\ell} \) in the case of \( k + 1 \le \ell \le 5k\)} $\,$  \vskip 0.2cm

The aim of this subsection is to derive the precise asymptotic behavior of the eigenvalues \(\lambda_{p,\ell}\) for \(k+1\le \ell\le 5k\). The main result is stated in Proposition~\ref{MI:prop:transmatrix}. And we normalize \(v_{p,\ell}\) by \( \|v_{p,\ell} \|_{L^{\infty} (\Omega)} = 1 \).

Although $u_p$ may change sign, the translation derivatives remain exact
solutions of the linearized equation.  Indeed, for $p>1$ the map
$f_p(t)=(t^+)^p$ is $C^1$ with
\[
f_p'(t)=p(t^+)^{p-1}.
\]
Differentiating \eqref{1.1} in the distributional sense, and then using
elliptic regularity, gives
\begin{equation}\label{MI:eq:translation-exact}
\Delta^2(\partial_{x_i}u_p)
=
p(u_p^+)^{p-1}\partial_{x_i}u_p
\qquad \text{in}~ \Omega, \qquad i=1,\ldots,4.
\end{equation}
Thus the translation derivatives are exact solutions of the linearized equation in the interior, although in general they do not satisfy the boundary conditions.

Define the normalized cutoff translations
\begin{equation}\label{MI:eq:Ttest}
T_{p,j,i}
:=
\frac{p\e_{p,j}}{m_{p,j}}\,
\chi_j\partial_{x_i}u_p,
\qquad i=1,\dots,4,
\end{equation}
where \( \chi_j \) is the cutoff function introduced at the beginning of the proof of Proposition~\ref{MI:prop:firstk}. We have the following lemma about the variational upper bound for the \(4k\) translation branches.

\begin{Lem}
There exists a \(5k\)-dimensional subspace \(E_p\subset H_0^2(\Omega)\)
such that
\begin{equation*}
\sup_{0\ne\phi\in E_p}
\frac{a(\phi,\phi)}{b_p(\phi,\phi)}
\le1+C\e_p^2+o(p^{-1}).
\end{equation*}
In particular,
\begin{equation}\label{MI:eq:lambda5k-upper}
\lambda_{p,5k}
\le1+o(p^{-1}).
\end{equation}
\end{Lem}

\begin{proof}
Take
\[
E_p
=
\operatorname{span}\{U_{p,j},T_{p,j,i}:1\le j\le k,\ 1\le i\le4\}.
\]
The supports associated with different spikes are disjoint. For a fixed \(j\) and \(\boldsymbol{\vartheta}_j=(\vartheta_{1,j},\ldots,\vartheta_{4,j})\in\mathbb R^4\), define
\[
s_{p,j}:=\sum_{i=1}^4\vartheta_{i,j}\partial_{x_i}u_p,
\qquad
T_{p,j}(\boldsymbol{\vartheta}_j):=\frac{p\e_{p,j}}{m_{p,j}}\chi_js_{p,j}.
\]
By \eqref{MI:eq:translation-exact},
\begin{equation}\label{eq:0000}
    \mathcal{L}_ps_{p,j}=0
\qquad\text{in }\Omega.
\end{equation}
Testing this equation against \(\chi_j^2s_{p,j}\), we obtain
\[
a(\chi_js_{p,j},\chi_js_{p,j})
-b_p(\chi_js_{p,j},\chi_js_{p,j})
=O(p^{-2}|\boldsymbol{\vartheta}_j|^2),
\]
because all cutoff errors are supported on the fixed annulus, where Lemma~\ref{prop:expansionupwithdelta} gives
\[
|s_{p,j}|+|\nabla s_{p,j}|+|\Delta s_{p,j}|
\le C|\boldsymbol{\vartheta}_j|p^{-1}.
\]
After multiplying by \((p\e_{p,j}/m_{p,j})^2\),
\begin{equation}\label{MI:eq:T-Q-small}
a(T_{p,j}(\boldsymbol{\vartheta}_j),T_{p,j}(\boldsymbol{\vartheta}_j))
-b_p(T_{p,j}(\boldsymbol{\vartheta}_j),T_{p,j}(\boldsymbol{\vartheta}_j))
=O(\e_{p,j}^2|\boldsymbol{\vartheta}_j|^2).
\end{equation}
Moreover, choose $d$ sufficiently small that Lemma~\ref{MI:lem:gradient-tail} also applies with $2d$ in place of $d$; its proof then uses the source ball $B_{4d}(x_{p,j})$. Since $\chi_j$ is supported in $B_{2d}(x_{p,j})$, scaling gives
\[
 b_p(T_{p,j}(\boldsymbol{\vartheta}_j),T_{p,j}(\boldsymbol{\vartheta}_j))
 =\int_{B_{2d/\e_{p,j}}(0)}
 W_{p,j}(z)\chi_j^2(x_{p,j}+\e_{p,j}z)
 (\boldsymbol{\vartheta}_j\cdot\nabla Z_{p,j}(z))^2\,dz.
\]
By \eqref{MI:weight-tail} and \eqref{MI:eq:gradient-tail}, the absolute value of the integrand is bounded by $C|\boldsymbol{\vartheta}_j|^2(1+|z|)^{-8+\sigma}$. Dominated convergence applied to the entries of the associated $4\times4$ Gram matrix yields, uniformly in $\boldsymbol{\vartheta}_j\in\mathbb R^4$,
\begin{equation*}
 b_p(T_{p,j}(\boldsymbol{\vartheta}_j),T_{p,j}(\boldsymbol{\vartheta}_j))
 =\int_{\mathbb R^4}e^Z(\boldsymbol{\vartheta}_j\cdot\nabla Z)^2\,dz
 +o(1)
 \ge c_0|\boldsymbol{\vartheta}_j|^2
\end{equation*}
for all sufficiently large $p$, with $c_0>0$ independent of $p$, $j$ and $\boldsymbol{\vartheta}_j$.
The amplitude--translation mixed weighted product is negligible.  Indeed,
before the normalization in \eqref{MI:eq:Ttest},
\[
b_p(U_{p,j},\chi_j\partial_{x_i} u_p)
=
p\int_{\Omega}\chi_j^2(u_p^+)^p\partial_{x_i} u_p \,dx
=
-\frac{2p}{p+1}
\int_{\Omega}\chi_j\partial_{x_i}\chi_j (u_p^+)^{p+1} \,dx,
\]
which is $O((C/p)^{p+1})$, since the integral is supported on the cutoff annulus and $u_p=O(p^{-1})$ there.  Testing the equation \eqref{eq:0000} against \(\chi_j^2u_p\) shows in the same way that
\begin{equation}\label{eq:22222}
    a(U_{p,j},T_{p,j,i})-b_p(U_{p,j},T_{p,j,i})
=O(\e_{p,j}/p).
\end{equation}
For the amplitude direction, by the previous estimates, we have
\[
a(U_{p,j},U_{p,j})
-
b_p(U_{p,j},U_{p,j})
=
-64\pi^2e+o(1),
\]
and hence, for some \(\tilde{c}>0\),
\begin{equation}\label{eq:33333}
    a(U_{p,j},U_{p,j})
-
b_p(U_{p,j},U_{p,j})
\le -\tilde{c}
\end{equation}
for all sufficiently large \(p\).

At the \(j\)-th bubble scale, the five functions
\[
U_{p,j},\quad T_{p,j,1},\ldots,T_{p,j,4}
\]
converge, up to the harmless nonzero factor \(m_{p,j}\to\sqrt e\), to
\[
1,\quad
\partial_1Z,\ldots,\partial_4Z.
\]
These five limiting functions are linearly independent, so \(\dim E_p=5k\) for large \(p\).
Write $\phi=U(\boldsymbol\alpha)+T(\boldsymbol{\vartheta})\in E_p$, where
\[
 \boldsymbol\alpha=(\alpha_1,\ldots,\alpha_k)\in\mathbb R^k,
 \qquad
 \boldsymbol{\vartheta}=(\boldsymbol{\vartheta}_1,\ldots,\boldsymbol{\vartheta}_k)\in\mathbb R^{4k},
 \qquad \boldsymbol{\vartheta}_j\in\mathbb R^4,
\]
and
\[
 U(\boldsymbol\alpha):=\sum_{j=1}^k\alpha_jU_{p,j},
 \qquad
 T(\boldsymbol{\vartheta}):=\sum_{j=1}^kT_{p,j}(\boldsymbol{\vartheta}_j).
\]
By the pairwise disjointness of the supports and the uniform positivity of the local $b_p$-Gram matrices, there exists $c_1>0$ such that
\begin{equation}\label{eq:88888}
    b_p(\phi,\phi)\ge c_1(|\boldsymbol\alpha|^2+|\boldsymbol{\vartheta}|^2).
\end{equation}
Using \eqref{MI:eq:T-Q-small}, \eqref{eq:22222} and \eqref{eq:33333}
\[
a(\phi,\phi)-b_p(\phi,\phi)
\le-\tilde{c}|\boldsymbol\alpha|^2+C\e_p^2|\boldsymbol{\vartheta}|^2
+o(p^{-1})(|\boldsymbol\alpha|^2+|\boldsymbol{\vartheta}|^2)
+O(\e_p/p)|\boldsymbol\alpha||\boldsymbol{\vartheta}|.
\]
By Young's inequality, it follows that the mixed term can be absorbed into the negative amplitude contribution, and we obtain
\[
a(\phi,\phi)-b_p(\phi,\phi)
\le
\left(
C\varepsilon_p^2+o(p^{-1})
\right)
\bigl(|\boldsymbol\alpha|^2+|\boldsymbol{\vartheta}|^2\bigr).
\]
Combining this estimate with \eqref{eq:88888} yields,
\[a(\phi,\phi)\le(1+C\e_p^2+o(p^{-1}))b_p(\phi,\phi).\]
Therefore
\[
\sup_{0\ne\phi\in E_p}
\frac{a(\phi,\phi)}
{b_p(\phi,\phi)}
\le
1+C\varepsilon_p^2+o(p^{-1}).
\]
Then, the min-max characterization of the generalized eigenvalues gives
\[
\lambda_{p,5k}
\le
\sup_{0\ne\phi\in E_p}
\frac{a(\phi,\phi)}{b_p(\phi,\phi)}
\le
1+C\varepsilon_p^2+o(p^{-1}).
\]
Consequently,
\[
\lambda_{p,5k}
\le
1+o(p^{-1}),
\]
which completes the proof.
\end{proof}

\begin{Lem}\label{MI:prop:puretranslations}
For every $k+1\le\ell\le5k$,
\begin{equation}\label{MI:eq:near1}
 \lambda_{p,\ell}\longrightarrow1.
\end{equation}
Moreover, from any subsequence one can extract a further common subsequence along which, for all $k+1\le\ell\le5k$ and $j=1,\ldots,k$,
\begin{equation}\label{MI:eq:no-scale-firstblock}
 V_{p,\ell,j}\longrightarrow
 \sum_{i=1}^4\beta_{i,\ell,j}Y_i
 \qquad\text{in }C^3_{\mathrm{loc}}(\mathbb R^4).
\end{equation}
Along this common subsequence, the translation vectors
\[
 \boldsymbol\beta_\ell:=
 (\beta_{1,\ell,1},\ldots,\beta_{4,\ell,1},\ldots,
 \beta_{1,\ell,k},\ldots,\beta_{4,\ell,k})^T
 \in\mathbb R^{4k}
\]
are nonzero and mutually orthogonal. Hence
$\boldsymbol\beta_{k+1},\ldots,\boldsymbol\beta_{5k}$ form a basis of $\mathbb R^{4k}$.
\end{Lem}

\begin{proof}
By \eqref{MI:eq:kplus1-liminf}, \eqref{MI:eq:lambda5k-upper} and ordering, \eqref{MI:eq:near1} holds. Fix any subsequence. Proposition~\ref{MI:prop:tube-capture}, a finite diagonal extraction and Corollary~\ref{MI:cor:captured-class} yield a further common subsequence along which
\[
 V_{p,\ell,j}\longrightarrow
 \beta_{0,\ell,j}Y_0+\sum_{i=1}^4\beta_{i,\ell,j}Y_i
 \qquad\text{in }C^3_{\mathrm{loc}}(\mathbb R^4)
\]
for all $k+1\le\ell\le5k$ and $j=1,\ldots,k$. If $\beta_{0,\ell,j_0}\ne0$ for some $\ell$ and $j_0$, then Proposition~\ref{MI:prop:scale}
would give
\[
\lambda_{p,\ell}
=
1+\frac{10}{p}(1+o(1)),
\]
which contradicts \(\lambda_{p,\ell}\le1+o(p^{-1})\).  This proves
\eqref{MI:eq:no-scale-firstblock}.  Proposition~\ref{MI:prop:tube-capture}
shows that every \(\boldsymbol\beta_\ell\) is nonzero.

On the other hand, for \(\ell\ne m\), the eigenfunctions are \(b_p\)-orthogonal. Passing
to the spike scales and using \eqref{MI:eq:no-scale-firstblock}, we have
\[
0
=
\sum_{j=1}^k\sum_{i,h=1}^4
\beta_{i,\ell,j}\beta_{h,m,j}
\int_{\R^4}e^ZY_iY_h \,dz +o(1).
\]
Since
\[
\int_{\R^4}e^ZY_iY_h \,dz
=
\frac{\pi^2a}{120}\delta_{ih},
\]
we obtain
\[
\boldsymbol\beta_\ell\cdot\boldsymbol\beta_m=0.
\]
There are exactly \(4k\) nonzero mutually orthogonal vectors in
\(\R^{4k}\), so they form a basis.
\end{proof}

We next establish a sharper estimate for \( A_{p,j,\ell}\) for \(k+1\le\ell\le5k\). Set
\[
A_p^*:=\max_{1\le j \le k }|A_{p,j,\ell}|.
\]

\begin{Lem}
For \( k+1 \le \ell \le 5k \), it holds
\begin{equation}\label{MI:eq:A-stage1}
A_p^*=o(p^{-1}),
\end{equation}
and
\begin{equation}\label{MI:eq:A-prebootstrap}
A_p^*
=
o(\e_p)+o(|\lambda_{p,\ell}-1|).
\end{equation}
\end{Lem}

\begin{proof}
By Lemma~\ref{MI:prop:puretranslations}, if \( k+1 \le \ell \le 5k\), then \( \beta_{0,\ell,j} = 0\) for every \( j = 1,\dots,k\). Thus, similar to the proof of Proposition~\ref{MI:prop:scale}, set
\[
I_{p,j,\ell}:=
\int_{B_d(x_{p,j})}(u_p^+)^p(y)v_{p,\ell}(y) \,dy,
\]
we have
\begin{equation}\label{MI:eq:I-A-general}
I_{p,j,\ell}
=
O\left(\frac{A_p^*+\e_p}{p^2}\right),
\end{equation}
and
\begin{equation}\label{eq:jjgj111}
m_{p,j}A_{p,j,\ell}-pI_{p,j,\ell}
= o (p^{-1}).
\end{equation}
Combining \eqref{MI:eq:I-A-general} and \eqref{eq:jjgj111}, it yields
\[
|A_{p,j,\ell}|
\le
\frac CpA_p^*+\frac C p\e_p+o(p^{-1}).
\]
Taking the maximum over \(j\) and absorbing the \(p^{-1}A_p^*\) term
gives \eqref{MI:eq:A-stage1}.

We now use the \(P_j\)-identity to sharpen the estimate.  Insert the estimates
\eqref{MI:eq:eig-outer} and \eqref{dluo-1} for \(v_{p,\ell}\) and \(u_p\), respectively, into \( P_j (v_{p,\ell} , u_p)\). By \eqref{3.553.55}, the \(GG\)-term equals
\[
\frac{\lambda_{p,\ell}A_{p,j,\ell}C_{p,j}}{8\pi^2}
=
\frac{8\sqrt e+o(1)}pA_{p,j,\ell}.
\]
The translation contribution to \(P_j(v_{p,\ell},u_p)\) is
\[
-\frac{\lambda_{p,\ell}\e_{p,j}}2
\sum_{i=1}^4B_{p,i,j,\ell}
\left[
C_{p,j}\frac{\partial R(x_{p,j})}{\partial x_i}
+2\sum_{m\ne j}C_{p,m}D_{x_i}G(x_{p,m},x_{p,j})
\right],
\]
which is
\(o(\e_p/p)\) by the refined force estimate \eqref{daluo-gil} and the
boundedness of \(B_{p,i,j,\ell}\). Therefore
\begin{equation}\label{MI:eq:P-A}
P_j(v_{p,\ell},u_p)
=
\frac{8\sqrt e+o(1)}pA_{p,j,\ell}
+o\left(\frac{\e_p}{p}\right).
\end{equation}
Set
\[
J_{p,j,\ell}
:=
\int_{B_d(x_{p,j})}
(u_p^+)^{p-1}v_{p,\ell}
(x-x_{p,j})\cdot\nabla u_p \,dx.
\]
Since $u_p=O(p^{-1})$ on $\partial B_d(x_{p,j})$ and $\|v_{p,\ell}\|_{L^\infty(\Omega)}=1$, the nonlinear boundary term in \eqref{MI:eq:P-general} is $O((C/p)^p)$. For every fixed $c>0$, $(C/p)^p=o(e^{-cp})$. Hence Theorem~\ref{th1} gives $(C/p)^p=o(\e_p/p)$. Therefore,
\begin{equation}\label{eq:p1234}
    P_j(v_{p,\ell},u_p)
=
4I_{p,j,\ell}+o\left( \frac{\varepsilon_p}{p}\right)
-(\lambda_{p,\ell}-1)pJ_{p,j,\ell}.
\end{equation}
The same dominated-convergence argument as in \eqref{MI:eq:pJscale}, justified by Lemma~\ref{MI:lem:gradient-tail}, now has a zero limit because all scale coefficients vanish and the limiting translation integrands are odd. Thus
\[
J_{p,j,\ell}=o(p^{-2}).
\]
Hence, \eqref{MI:eq:I-A-general}, \eqref{MI:eq:P-A} and \eqref{eq:p1234} imply
\[
|A_{p,j,\ell}|
\le
\frac CpA_p^*
+o(\e_p)
+o(|\lambda_{p,\ell}-1|).
\]
Taking the maximum and absorbing the \(p^{-1}A_p^*\) term again proves
\eqref{MI:eq:A-prebootstrap}.
\end{proof}

By Theorem~\ref{th1},
\begin{equation*}
\e_{p,j}=e^{-p/8}\big(\rho_j+o(1)\big),
\qquad
\rho_j:=\exp\left(8\pi^2\Psi_{k,j}({\mathbf{x}}_\infty)-\frac12\ln(8\sqrt6)-\frac{13}{24}\right)>0.
\end{equation*}
We define
\begin{equation}\label{MI:Hcal}
\mathcal H:=\operatorname{diag}(\rho_1I_4,\ldots,\rho_kI_4),
\qquad
H_p:=\operatorname{diag}(\e_{p,1}I_4,\ldots,\e_{p,k}I_4),
\end{equation}
so that
\begin{equation}\label{MI:Hp-asymp}
H_p=e^{-p/8}\big(\mathcal H+o(1)\big).
\end{equation}
Then, based on the above results, we can obtain that the precise asymptotic behavior of \( \lambda_{p,\ell} \) for \( k+1 \le \ell \le 5k\).

\begin{Prop}\label{MI:prop:transmatrix}
Fix \(k+1\le\ell\le5k\) and an arbitrary subsequence.  After
passing to a further subsequence, let \(\boldsymbol\beta_\ell\ne0\) be the
translation coefficient vector of the local limits given by
Lemma~\ref{MI:prop:puretranslations}. Then, along this subsequence,
\begin{equation}\label{MI:eq:matrix-p}
(\lambda_{p,\ell}-1)\boldsymbol\beta_\ell
=
-80\sqrt6\,\pi^2
H_pD^2\Psi_k({\mathbf{x}}_\infty)H_p\boldsymbol\beta_\ell
+o(\e_p^2).
\end{equation}
In particular,
\begin{equation}\label{MI:eq:lambda-eps2}
\lambda_{p,\ell}-1=O(\e_p^2).
\end{equation}
\end{Prop}

\begin{proof}
We will apply the generalized identity \eqref{MI:eq:Q-general} to prove Proposition \ref{MI:prop:transmatrix}. On the left-hand side of \eqref{MI:eq:Q-general}, substituting the estimates
\eqref{MI:eq:eig-outer} and \eqref{dluo-1} for \(v_{p,\ell}\) and \(u_p\), respectively, into \(Q_{j}(v_{p,\ell},u_p)\), by \eqref{3.553.55}, \eqref{MI:eq:B-limit} and \eqref{MI:eq:near1}, we have
\begin{equation*}
    \begin{split}
        Q_{j} (v_{p,\ell},u_p) =& \lambda_{p,\ell} \sum_{s,m=1}^{k} A_{p,s,\ell} C_{p,m} Q_j \left( G (x_{p,s},x),G(x_{p,m},x)\right) \\
        &+ \frac{512 \pi^4 \sqrt{e}}{p} \sum_{s,m=1}^{k} \sum_{h=1}^{4} \varepsilon_{p,s} \beta_{h,\ell, s} Q_j \left( \partial_h G(x_{p,s},x) , G(x_{p,m},x)\right) + o \left( \frac{\varepsilon_p}{p}\right).
    \end{split}
\end{equation*}
By \eqref{aluo1}, \eqref{3.553.55} and \eqref{MI:eq:A-prebootstrap}, the \(GG\)-terms are
\[
o\left(\frac{\e_p}{p}\right)
+
o\left(\frac{|\lambda_{p,\ell}-1|}{p}\right).
\]
For the translation terms, from the Green identities in Proposition~\ref{lem2-1}:
\[
Q_{j}(G(x_{p,j},x),\partial_hG(x_{p,j},x))
=\frac12 \frac{\partial^2R(x_{p,j})}{\partial x_i \partial x_h},
\]
\[
Q_{j}(G(x_{p,j},x),\partial_hG(x_{p,s},x))
=D_{x_i}\partial_hG(x_{p,s},x_{p,j}),
\qquad s\ne j,
\]
and
\[
Q_{j}(G(x_{p,m},x),\partial_hG(x_{p,j},x))
=D_{x_i x_h}^2 G(x_{p,m},x_{p,j}),
\qquad m\ne j.
\]
We can complete
\begin{equation*}
    \begin{split}
& \sum_{s,m=1}^{k} \sum_{h=1}^{4} \varepsilon_{p,s} \beta_{h,\ell, s} Q_j \left( \partial_h G(x_{p,s},x) , G(x_{p,m},x)\right)\\
={}&
\sum_{h=1}^4
\varepsilon_{p,j}\beta_{h,\ell, j}
\left( \frac{1}{2} \frac{\partial^2R(x_{p,j})}{\partial x_i \partial x_h} + \sum_{\substack{m=1\\m\ne j}}^kD_{x_ix_h}^2G(x_{p,m},x_{p,j})\right) +
\sum_{\substack{s=1\\s\ne j}}^k
\sum_{h=1}^4
\varepsilon_{p,s}\beta_{h,\ell, s}
D_{x_i}\partial_hG(x_{p,s},x_{p,j}).
    \end{split}
\end{equation*}
Recall that
\[
\Psi_k(a)
=
\sum_{j=1}^k
\left(
R(a_j)
+
\sum_{m=1 ,m \ne j}^{k} G(a_j,a_m)
\right),
\]
we can check that
\[
\frac12
\big[D^2\Psi_k(\boldsymbol{{x}_p})H_p\boldsymbol{{\beta}_{\ell}}\big]_{j,i}
=\sum_{s,m=1}^{k} \sum_{h=1}^{4} \varepsilon_{p,s} \beta_{h,\ell, s} Q_j \left( \partial_h G(x_{p,s},x) , G(x_{p,m},x) \right).
\]
Consequently, it follows that
\begin{equation}\label{MI:eq:Q-left-trans}
Q_{j}(v_{p,\ell},u_p)
=
\frac{256\pi^4\sqrt e}{p}
\left[
D^2\Psi_k({{\mathbf{x}}_\infty})H_p\boldsymbol{{\beta}_{\ell}}
\right]_{j,i}
+o\left(\frac{\e_p}{p}\right)
+o\left(\frac{|\lambda_{p,\ell}-1|}{p}\right).
\end{equation}

On the right-hand side of \eqref{MI:eq:Q-general}, the nonlinear boundary term is again $O((C/p)^p)=o(\e_p/p)$. By \eqref{MI:weight-tail}, Lemma~\ref{MI:lem:gradient-tail} and $|V_{p,\ell,j}|\le1$,
\[
 |W_{p,j}(z)V_{p,\ell,j}(z)\partial_iZ_{p,j}(z)|
 \le C(1+|z|)^{-8+\sigma}
 \qquad\text{for } |z|\le d/\e_{p,j}.
\]
Consequently, scaling and dominated convergence give
\begin{align}\label{MI:eq:trans-source}
p\int_{B_d(x_{p,j})}
(u_p^+)^{p-1}v_{p,\ell}\frac{\partial u_p}{\partial {x_i}} \,dx
&=
\frac{m_{p,j}}{p\e_{p,j}}
\left[
\int_{\R^4}e^Z
\left(\sum_{h=1}^4\beta_{h,\ell,j}Y_h\right)\partial_iZ \,dz +o(1)
\right]\\
&=
-\frac{\sqrt e\,\pi^2a}{15p\e_{p,j}}
\bigl(\beta_{i,\ell,j}+o(1)\bigr),
\notag
\end{align}
where
\begin{equation*}
\int_{\R^4}e^ZY_h\partial_iZ \,dz
=-\frac{\pi^2a}{15}\delta_{ih},
\qquad a=8\sqrt6.
\end{equation*}
Equating \eqref{MI:eq:Q-left-trans} and
\((\lambda_{p,\ell}-1)\) times \eqref{MI:eq:trans-source}, and multiplying by
\(p\e_{p,j}/\sqrt e\), first gives the rough block equation
\[
(\lambda_{p,\ell}-1)\beta_{i,\ell,j}
=
-80\sqrt6\,\pi^2\e_{p,j}
\left[D^2\Psi_k( {\mathbf{x}}_\infty)H_p\boldsymbol{\beta_{\ell}}\right]_{j,i}
+o(\e_p^2)+o(|\lambda_{p,\ell}-1|).
\]
Since multiplication by \(H_p\) on the left multiplies the \(j\)-th block by \(\varepsilon_{p,j}\), the first term on the right-hand side is exactly
\[
-80\sqrt6\,\pi^2
\left[
H_pD^2\Psi_k({\mathbf{x}}_\infty)H_p\boldsymbol{\beta_{\ell}}
\right]_{j,i}.
\]
Collecting all \(j=1,\ldots,k\) and \(i=1,\ldots,4\), we obtain the vector relation
\begin{equation}\label{eq:gx}
    (\lambda_{p,\ell}-1)\boldsymbol{\beta_{\ell}}
=
-80\sqrt6\,\pi^2
H_pD^2\Psi_k({\mathbf{x}}_\infty)H_p\boldsymbol{\beta_{\ell}}
+
o(\varepsilon_p^2)
+
o(|\lambda_{p,\ell}-1|).
\end{equation}
By \eqref{MI:Hcal}, we have
\[
\|H_p\|=O(\varepsilon_p),
\]
and
\[
\left\|
H_pD^2\Psi_k({\mathbf{x}}_\infty)H_p\beta
\right\|
\le
C\varepsilon_p^2|\beta|.
\]
Here \(\|\cdot\|\) denotes the operator norm induced by the Euclidean norm. Taking Euclidean norms in \eqref{eq:gx} gives
\[
|\lambda_{p,\ell}-1|\,|\beta|
\le
C\varepsilon_p^2|\beta|
+
o(\varepsilon_p^2)
+
o(|\lambda_{p,\ell}-1|).
\]
Due to that \(\boldsymbol{\beta_{\ell}}\ne0\), thus we get \eqref{MI:eq:lambda-eps2}. Then by \eqref{MI:eq:A-prebootstrap}, which implies
\[
A_{p,j,\ell}=o(\varepsilon_p),
\qquad j=1,\dots,k.
\]
Returning to
\eqref{MI:eq:Q-left-trans}--\eqref{MI:eq:trans-source}, all terms involving
\(A_{p,j,\ell}\) and the source \(o(1)\) are now \(o(\e_p^2)\), and we obtain
\eqref{MI:eq:matrix-p}.
\end{proof}

\begin{Cor}\label{MI:cor:split}
Let
\begin{equation*}
M:=\mathcal H D^2\Psi_k({{\mathbf{x}}_\infty})\mathcal H,
\end{equation*}
where \(\mathcal H\) is defined in \eqref{MI:Hcal}.  Counting multiplicities,
\begin{equation}\label{MI:eq:multiset}
\left\{
 e^{p/4}(1-\lambda_{p,\ell}):
 k+1\le\ell\le5k
\right\}
\longrightarrow
80\sqrt6\,\pi^2\,\operatorname{spec}(M).
\end{equation}
Equivalently, for \( k+1 \le \ell \le 5k \), if \(\theta_1 \ge \theta_2 \ge \dots  \ge \theta_{4k} \) are the eigenvalues of \(M\), then it holds
\begin{equation}\label{MI:eq:trans-eig-asymp}
\lambda_{p,\ell}
=
1-80\sqrt6\,\pi^2e^{-p/4}\theta_{\ell-k}
+o(e^{-p/4}).
\end{equation}
\end{Cor}

\begin{proof}
Take an arbitrary subsequence. Since there are only \(4k\) translation branches, by a finite diagonal extraction, we may pass to a further common subsequence such that the limiting translation coefficient vectors \(\beta_\ell\) are simultaneously defined for all \(k+1\le\ell\le5k\), and
\[
\gamma_{p,\ell}:=e^{p/4}(1-\lambda_{p,\ell})
\]
converges to \( \gamma_{\ell}\) for every \(k+1\le\ell\le5k\). The boundedness of \(\gamma_{p,\ell}\) follows from
\eqref{MI:eq:lambda-eps2}. Lemma~\ref{MI:prop:puretranslations} shows
that the limiting vectors \(\boldsymbol\beta_{k+1},\dots,\boldsymbol\beta_{5k}\) are a
basis of \(\R^{4k}\).

From \eqref{MI:Hp-asymp},
\[
e^{p/8}H_p\to\mathcal H,
\]
dividing \eqref{MI:eq:matrix-p} by \(e^{-p/4}\) gives, for every limiting
branch,
\[
\gamma_{\ell}\boldsymbol\beta_\ell
=
80\sqrt6\,\pi^2
M\boldsymbol\beta_\ell.
\]
Thus every \(\boldsymbol\beta_\ell\) is an eigenvector of \(M\), with eigenvalue
\(\gamma_\ell/(80\sqrt6\pi^2)\). Since the \(4k\) limiting vectors
form a basis, these eigenvalues exhaust \(\operatorname{spec}(M)\) with multiplicity.
Since the original subsequence was arbitrary, the whole multiset
converges, proving \eqref{MI:eq:multiset}.
\end{proof}

\subsection{The estimate of \(\lambda_{p,5k+1}\)}

\begin{Prop}\label{MI:prop:aftertrans}
For all sufficiently large \(p\),
\begin{equation}\label{MI:eq:5kplus1}
\lambda_{p,5k+1}>1.
\end{equation}
More precisely, if along a subsequence
\(\lambda_{p,5k+1}\to1\), then
\begin{equation}\label{MI:eq:first-scale}
\lambda_{p,5k+1}
=
1+\frac{10}{p}(1+o(1)).
\end{equation}
\end{Prop}

\begin{proof}
First assume along a subsequence that \(\lambda_{p,5k+1}\to1\). Let \(v_{p,5k+1}\) be an associated eigenfunction, normalized by \( \|v_{p,5k+1}\|_{L^\infty(\Omega)}=1\). By Proposition~\ref{MI:prop:tube-capture} and Corollary~\ref{MI:cor:captured-class}, its local limits have the form
\[
V_{5k+1,j}=\beta_{0,5k+1,j}Y_0+\sum_{i=1}^4\beta_{i,5k+1,j}Y_i,
\]
with at least one nonzero coefficient. Define the translation coefficient vector of \(v_{p,5k+1}\) by
\[
\boldsymbol{\beta}
:=
(\beta_{1,5k+1,1},\ldots,\beta_{4,5k+1,1},
\ldots,
\beta_{1,5k+1,k},\ldots,\beta_{4,5k+1,k})^T
\in\mathbb R^{4k}.
\]
For every \(k+1\le \ell\le5k\), the eigenfunctions \(v_{p,5k+1}\) and \(v_{p,\ell}\) are \(b_p\)-orthogonal. By Lemma~\ref{MI:prop:puretranslations},
\[
V_{\ell,j}=\sum_{h=1}^4\beta_{h,\ell,j}Y_h.
\]
Passing to the bubble scales and using
\[
\int_{\mathbb R^4}e^ZY_0Y_h \,dz =0,
\qquad
\int_{\mathbb R^4}e^ZY_iY_h \,dz
=\frac{\pi^2a}{120}\delta_{ih},
\]
we obtain
\[
\boldsymbol{\beta}\cdot\boldsymbol{\beta_\ell}=0,
\qquad k+1\le\ell\le5k.
\]
But the \(4k\) translation vectors \(\boldsymbol{\beta_{\ell}}\) form a basis of \(\R^{4k}\), which follows that \(\boldsymbol\beta=0\). Hence some scale coefficient is nonzero, and
Proposition~\ref{MI:prop:scale} yields \eqref{MI:eq:first-scale}.

If \(\lambda_{p,5k+1}\le1\) along a subsequence, then
\eqref{MI:eq:kplus1-liminf} and ordering imply
\(\lambda_{p,5k+1}\to1\); the preceding asymptotics give
\(\lambda_{p,5k+1}>1\), a contradiction. Thus \eqref{MI:eq:5kplus1}
holds for all sufficiently large \(p\).
\end{proof}

\subsection{Proof of Theorem~\ref{thMorse} and Corollary~\ref{th1.1}}

\begin{proof}[\textbf{Proof of Theorem~\ref{thMorse}}]
By Lemma~\ref{MI:lem:semidefinite-weight}, the possible kernel of the
weighted form $b_p$ is a strictly positive subspace for $\mathcal Q_p$ and
therefore contributes neither to the Morse index nor to the augmented
Morse index.  Hence it is enough to count the finite generalized
eigenvalues in \eqref{MI:eq:eig} which lie below, or at, the threshold
$1$.

Proposition~\ref{MI:prop:firstk} gives exactly \(k\) eigenvalues which are
strictly below \(1\) and converge to \(0\).

By Corollary~\ref{MI:cor:split}, among the next \(4k\) eigenvalues,
every positive eigenvalue of
\[
M=\mathcal H D^2\Psi_k( {\mathbf{x}}_\infty)\mathcal H
\]
produces an eigenvalue \(\lambda_{p,\ell}<1\), whereas every negative
eigenvalue of \(M\) produces an eigenvalue \(\lambda_{p,\ell}>1\).
For an eigenvalue of $M$ equal to zero, the expansion does not determine whether the corresponding $\lambda_{p,\ell}$ is below, above or equal to $1$. These branches contribute at most $n_0(M)$ to either index.
Proposition~\ref{MI:prop:aftertrans} shows that
\(\lambda_{p,5k+1}>1\), hence no later eigenvalue contributes.

Since \(\mathcal H\) is positive definite, Sylvester's law of inertia gives
\[
n_{+}(M)
=
n_{+}\!\left(D^2\Psi_k({\mathbf{x}}_\infty)\right),
\]
and similarly for the negative and zero inertia.
This proves \eqref{morse-index-ineq}.

If \(D^2\Psi_k({\mathbf{x}}_\infty)\) is nondegenerate, the zero inertia
vanishes and
\[
m(u_p)=m_0(u_p)
=
k+n_{+}(D^2\Psi_k({\mathbf{x}}_\infty)).
\]
Since \(D^2\Psi_k\) is a \(4k\times4k\) matrix,
\[
n_{+}(D^2\Psi_k)
=
4k-n_{-}(D^2\Psi_k),
\]
and therefore
\[
m(u_p)
=
5k-m( {\mathbf{x}}_\infty,\Psi_k).
\]
\end{proof}

\begin{proof}[\textbf{Proof of Corollary~\ref{th1.1}}]
The \(k\) small eigenvalues satisfy \(\lambda_{p,\ell}\to0\).
The \(4k\) translation eigenvalues are separated from \(1\) by
\eqref{MI:eq:trans-eig-asymp}, since the effective Hessian is invertible.
The first eigenvalue after the translation block satisfies
\[
\lambda_{p,5k+1}
=
1+\frac{10}{p}(1+o(1))
\]
whenever it approaches \(1\), and all subsequent eigenvalues are no
smaller.
Thus \(1\) is not an eigenvalue of \eqref{MI:eq:eig}.
Equivalently,
\[
\ker\bigl(\Delta^2-p(u_p^+)^{p-1}\bigr)=\{0\}.
\]
\end{proof}

\appendix
\renewcommand{\theequation}{A.\arabic{equation}}

\setcounter{equation}{0}

\section{Proofs of (\ref{abb1-1})--(\ref{aluo41}) involving Green's function}\label{s6}
\setcounter{equation}{0}

In this appendix, we give the proofs of \eqref{abb1-1}--\eqref{aluo41} involving Green's function. The proof of  \eqref{1-1} is already quite detailed. The remaining proofs are analogous, so we will only sketch them here  for completeness. And all of the computations are based on Lemma \ref{Greenfun} and the basic results \eqref{jisuan}.

\begin{proof}[{\textbf{Proof of \eqref{abb1-1}}}]
By the bilinearity of $P_j(u,v)$, we have
\begin{align}
P_j\Bigr(G(x_{p,s},x),\partial_hG(x_{p,m},x)\Bigr)
={}&P_j\Bigr(S(x_{p,s},x),\partial_hH(x_{p,m},x)\Bigr)\notag+P_j\Bigr(H(x_{p,s},x),\partial_hH(x_{p,m},x)\Bigr)\notag\\
&+P_j\Bigr(G(x_{p,s},x),\partial_hS(x_{p,m},x)\Bigr).
\label{eq:P28-decomp}
\end{align}

We first consider the case $m=j$ and $s\neq j$. Since $x_{p,s}\neq x_{p,j}$, the function $S(x_{p,s},x)$ and $\partial_hH(x_{p,j},x)$ are smooth in a neighborhood of $x_{p,j}$.  Therefore, the first two terms on the right-hand side of \eqref{eq:P28-decomp} tend to zero as $\theta\to0$, and it remains to compute
$
P_j\Bigr(G(x_{p,s},x),\partial_hS(x_{p,j},x)\Bigr).
$

Since $\partial_hH(x_{p,j},x)$ is  smooth in $B_{\theta}(x_{p,j})$. Now, we expand $G(x_{p,s},x)$ at $x=x_{p,j}$, we have
\[
G(x_{p,s},x)
=
G(x_{p,s},x_{p,j})
+\theta\sum_{\ell=1}^4D_{x_\ell}G(x_{p,s},x_{p,j})\xi_\ell
+O(\theta^2),
\]
Thus, it follows that
\[
\frac{\partial}{\partial\nu} G(x_{p,s},x)
=
\sum_{\ell=1}^4D_{x_\ell}G(x_{p,s},x_{p,j})\xi_\ell+O(\theta),
\]
\[
\Delta_xG(x_{p,s},x)=\Delta_xG(x_{p,s},x_{p,j})+O(\theta),
\qquad
\frac{\partial}{\partial\nu}\Delta_xG(x_{p,s},x)=O(1),
\]
and
\[
\frac{\partial}{\partial\nu}\bigr\langle x-x_{p,j},\nabla_xG(x_{p,s},x)\bigr\rangle
=
\sum_{\ell=1}^4D_{x_\ell}G(x_{p,s},x_{p,j})\xi_\ell+O(\theta).
\]

Substituting these identities into the definition of $P_j$, we get

\begin{align}
&P_j\Bigr(G(x_{p,s},x),\partial_h  S(x_{p,j},x)\Bigr)
={}\displaystyle\int_{\partial B_\theta(x_{p,j})}
\Bigr[
-\theta \Delta_xG(x_{p,s},x)\Delta_x \big( \partial _h S(x_{p,j},x) \big)
\notag\\
&~~~~
-\theta \frac{\partial}{\partial\nu} \Delta_x \big( \partial _h S(x_{p,j},x) \big)\frac{\partial}{\partial\nu}G(x_{p,s},x)
-\theta\frac{\partial}{\partial\nu}\Delta_x G(x_{p,s},x) \frac{\partial}{\partial\nu} \big( \partial _h S(x_{p,j},x) \big)
\label{eq:P28-main}\\
&~~~~
+\Delta_x G(x_{p,s},x)\frac{\partial}{\partial\nu}\bigr\langle x-x_{p,j},\nabla_x\bigr(\partial_h S(x_{p,j},x)\bigr)\bigr\rangle
+\Delta_x \big( \partial _h S(x_{p,j},x) \big) \frac{\partial}{\partial\nu}\bigr\langle x-x_{p,j},\nabla_xG(x_{p,s},x)\bigr\rangle
\Bigr]\,d\sigma.\notag
\end{align}
According to \eqref{jisuan}, we know $\Delta_x\bigr(\partial_h S(x_{p,j},x)\bigr) = -\frac{\xi_h}{2\pi^2\theta^3}$, $\frac{\partial}{\partial\nu}\bigr\langle x-x_{p,j},\nabla_x\bigr(\partial_h S(x_{p,j},x)\bigr)\bigr\rangle = \frac{\xi_h}{8\pi^2\theta^2}$ and $\frac{\partial}{\partial\nu}\bigr(\partial_h S(x_{p,j},x)\bigr) = -\frac{\xi_h}{8\pi^2\theta^2}$, so we can easily prove that  the first term, the third term  and the fourth term on the right hand of \eqref{eq:P28-main} tend to zero.
For the second term, by \eqref{jisuan} we obtain
\begin{align*}
\displaystyle\int_{\partial B_\theta(x_{p,j})}
\theta\,\frac{\partial}{\partial\nu}\Delta_x\bigr(\partial_hS(x_{p,j},x)\bigr)\,\frac{\partial}{\partial\nu} G(x_{p,s},x)
=
\frac{3}{2\pi^2\theta^3}
\sum_{\ell=1}^4D_{x_\ell}G(x_{p,s},x_{p,j})
\displaystyle\int_{\partial B_\theta(x_{p,j})}\xi_h\xi_\ell\,d\sigma+o(1).
\end{align*}
For the fifth term, we have
\begin{align*}
&\displaystyle\int_{\partial B_\theta(x_{p,j})}
\Delta_x \big( \partial _h S(x_{p,j},x) \big) \frac{\partial}{\partial\nu}\bigr\langle x-x_{p,j},\nabla_xG(x_{p,s},x)\bigr\rangle \\
& \qquad \qquad =-\frac{1}{2\pi^2\theta^3}
\sum_{\ell=1}^4D_{x_\ell}G(x_{p,s},x_{p,j})
\displaystyle\int_{\partial B_\theta(x_{p,j})}\xi_h\xi_\ell\,d\sigma+o(1).
\end{align*}
Therefore, it follows that
\begin{align*}
P_j\bigr(G(x_{p,s},x),\partial_hS(x_{p,j},x)\bigr)
={}&
-\frac{2}{\pi^2\theta^3}
\sum_{\ell=1}^4D_{x_\ell}G(x_{p,s},x_{p,j})
\displaystyle\int_{\partial B_\theta(x_{p,j})}\xi_h\xi_\ell\,d\sigma+o(1).
\end{align*}
Since
$
\displaystyle\int_{\partial B_\theta(x_{p,j})}\xi_h\xi_\ell\,d\sigma
=
\frac{|S^3|}{4}\theta^3\delta_{h\ell}
=
\frac{\pi^2}{2}\theta^3\delta_{h\ell},
$
we get
\[
P_j\Bigr(G(x_{p,s},x),\partial_hS(x_{p,j},x)\Bigr)
=
-D_{x_h}G(x_{p,s},x_{p,j})+o(1).
\]
Hence, it holds
\[
P_j\Bigr(G(x_{p,s},x),\partial_hG(x_{p,j},x)\Bigr)
=
-D_{x_h}G(x_{p,s},x_{p,j}).
\]

Next, assume that $s=m=j$. By \eqref{eq:P28-decomp}, we obtain
\begin{align*}
P_j\Bigr(G(x_{p,j},x),\partial_hG(x_{p,j},x)\Bigr)
={}&P_j\Bigr(S(x_{p,j},x),\partial_hS(x_{p,j},x)\Bigr)
+P_j\Bigr(S(x_{p,j},x),\partial_hH(x_{p,j},x)\Bigr)\\
&+P_j\Bigr(H(x_{p,j},x),\partial_hS(x_{p,j},x)\Bigr)
+P_j\Bigr(H(x_{p,j},x),\partial_hH(x_{p,j},x)\Bigr).
\end{align*}
A direct computation gives
\[
P_j\Bigr(S(x_{p,j},x),\partial_hS(x_{p,j},x)\Bigr)=0.
\]
Moreover, since $H(x_{p,j},x)$ is smooth near $x_{p,j}$ and Lemma \ref{Greenfun}, we can compute directly that
\[
P_j\Bigr(S(x_{p,j},x),\partial_hH(x_{p,j},x)\Bigr)=o(1),
\qquad
P_j\Bigr(H(x_{p,j},x),\partial_hH(x_{p,j},x)\Bigr)=o(1),
\]
 Therefore
\[
P_j\Bigr(G(x_{p,j},x),\partial_hG(x_{p,j},x)\Bigr)
=
P_j\Bigr(H(x_{p,j},x),\partial_hS(x_{p,j},x)\Bigr)+o(1).
\]
Applying the computation above with $H(x_{p,j},x)$ in place of $G(x_{p,s},x)$, we obtain
\[
P_j\Bigr(H(x_{p,j},x),\partial_hS(x_{p,j},x)\Bigr)
=
-D_{x_h}H(x_{p,j},x_{p,j}).
\]
Hence
\[
P_j\Bigr(G(x_{p,j},x),\partial_hG(x_{p,j},x)\Bigr)
=
-D_{x_h}H(x_{p,j},x_{p,j}).
\]
Since $R(x)=H(x,x)$
and $H(x,y)=H(y,x)$, we have
\[
\partial_hR(x_{p,j})
=
D_{x_h}H(x_{p,j},x_{p,j})+\partial_hH(x_{p,j},x_{p,j})
=
2D_{x_h}H(x_{p,j},x_{p,j}),
\]
and therefore
\[
P_j\Bigr(G(x_{p,j},x),\partial_hG(x_{p,j},x)\Bigr)
=
-\frac12\,\partial_hR(x_{p,j}).
\]

Finally, if $m\neq j$, then $\partial_hG(x_{p,m},x)$ is smooth near $x_{p,j}$. Hence every term in \eqref{eq:P28-decomp} is $o(1)$, and so
\[
P_j\Bigr(G(x_{p,s},x),\partial_hG(x_{p,m},x)\Bigr)=0.
\]

\end{proof}
\begin{proof} [{\textbf{Proof of \eqref{aluo1}}}]
By the bilinearity of $Q_j(u,v)$, we have
\begin{align}
Q_j\Bigr(G(x_{p,m},x),G(x_{p,s},x)\Bigr)
={}&Q_j\Bigr(S(x_{p,m},x),H(x_{p,s},x)\Bigr)\notag+Q_j\Bigr(H(x_{p,m},x),H(x_{p,s},x)\Bigr)\notag\\
&+Q_j\Bigr(G(x_{p,m},x),S(x_{p,s},x)\Bigr).
\label{eq:Q29-decomp}
\end{align}

We first consider the case $s=j$ and $m\neq j$. Since $x_{p,m}\neq x_{p,j}$, the functions $S(x_{p,m},x)$ and $H(x_{p,j},x)$ are smooth in a neighborhood of $x_{p,j}$. Therefore, the first two terms on the right-hand side of \eqref{eq:Q29-decomp} tend to zero as $\theta\to0$, and it remains to compute
$
Q_j\Bigr(G(x_{p,m},x),S(x_{p,j},x)\Bigr).
$

Now expand $G(x_{p,m},x)$ at $x=x_{p,j}$, we have
\[
G(x_{p,m},x)
=
G(x_{p,m},x_{p,j})
+\theta\sum_{\ell=1}^4D_{x_\ell}G(x_{p,m},x_{p,j})\xi_\ell
+O(\theta^2),
\]
hence
\[
\partial_{x_i}G(x_{p,m},x)=D_{x_i}G(x_{p,m},x_{p,j})+O(\theta), \qquad
\frac{\partial}{\partial\nu}\partial_{x_i}G(x_{p,m},x)=O(1),
\]
\[
\Delta_xG(x_{p,m},x)=\Delta_xG(x_{p,m},x_{p,j})+O(\theta),
\qquad \;
\frac{\partial}{\partial\nu}\Delta_xG(x_{p,m},x)=O(1).
\]

Substituting these identities into the definition of $Q_j$, we get
\begin{align}
&Q_j\Bigr(G(x_{p,m},x),S(x_{p,j},x)\Bigr)
={}\displaystyle\int_{\partial B_\theta(x_{p,j})}
\Bigr[
\Delta_xG(x_{p,m},x)\Delta_xS(x_{p,j},x)\frac{x_i-x_{p,j,i}}{\theta}
\notag\\
&\qquad \qquad~~\qquad
+\frac{\partial}{\partial\nu}\Delta_xS(x_{p,j},x)\partial_{x_i}G(x_{p,m},x)
+\frac{\partial}{\partial\nu}\Delta_xG(x_{p,m},x)\partial_{x_i}S(x_{p,j},x)
\notag\\
&\qquad \qquad ~~ \qquad
-\Delta_xG(x_{p,m},x)\frac{\partial}{\partial\nu}\partial_{x_i}S(x_{p,j},x)
-\Delta_xS(x_{p,j},x)\frac{\partial}{\partial\nu}\partial_{x_i}G(x_{p,m},x)
\Bigr]\,d\sigma.
\label{eq:Q29-main}
\end{align}
By \eqref{jisuan}, we have $\Delta_x S(x_{p,j},x) = -\frac{1}{4\pi^2\theta^2}$. We can similarly get that $\frac{\partial}{\partial\nu}\partial_{x_i}S(x_{p,j},x) =O(\theta^{-2})$, $\partial_{x_i}S(x_{p,j},x) =O (\theta^{-1})$, so we can easily prove that  the first, third, fourth, fifth term on the right hand of \eqref{eq:Q29-main} tend to zero.
Therefore the only surviving contribution comes from the second term:
\begin{align*}
\displaystyle\int_{\partial B_\theta(x_{p,j})}
\frac{\partial}{\partial\nu}\Delta_xS(x_{p,j},x)\partial_{x_i}G(x_{p,m},x)\,d\sigma
={}
\displaystyle\int_{\partial B_\theta(x_{p,j})}
\frac1{2\pi^2\theta^3}
\Bigr(D_{x_i}G(x_{p,m},x_{p,j})+O(\theta)\Bigr)\,d\sigma.
\end{align*}
Since
$
|\partial B_\theta(x_{p,j})|=2\pi^2\theta^3,
$
we obtain
\[
Q_j\Bigr(G(x_{p,m},x),S(x_{p,j},x)\Bigr)
=
D_{x_i}G(x_{p,m},x_{p,j})+o(1).
\]
Hence, it holds
\[
Q_j\Bigr(G(x_{p,m},x),G(x_{p,j},x)\Bigr)
=
D_{x_i}G(x_{p,m},x_{p,j}),
\]
which proves the second line of \eqref{aluo1}. By symmetry, the third line follows.

Next, assume $s=m=j$. By \eqref{eq:Q29-decomp}, we get
\begin{align*}
Q_j\Bigr(G(x_{p,j},x),G(x_{p,j},x)\Bigr)
={}&Q_j\Bigr(S(x_{p,j},x),S(x_{p,j},x)\Bigr)
+Q_j\Bigr(S(x_{p,j},x),H(x_{p,j},x)\Bigr)\\
&+Q_j\Bigr(H(x_{p,j},x),S(x_{p,j},x)\Bigr)
+Q_j\Bigr(H(x_{p,j},x),H(x_{p,j},x)\Bigr).
\end{align*}
By Lemma \ref{Greenfun}, direct computations give
\[
Q_j\Bigr(S(x_{p,j},x),S(x_{p,j},x)\Bigr)=0,
~~\text{and}~~
Q_j\Bigr(H(x_{p,j},x),H(x_{p,j},x)\Bigr)=o(1).
\]
Applying the previous computation with $H(x_{p,j},x)$ in place of $G(x_{p,m},x)$, we obtain
\[
Q_j\Bigr(H(x_{p,j},x),S(x_{p,j},x)\Bigr)
=
D_{x_i}H(x_{p,j},x_{p,j}),
~~\text{and}~~
Q_j\Bigr(S(x_{p,j},x),H(x_{p,j},x)\Bigr)
=
D_{x_i}H(x_{p,j},x_{p,j}).
\]
Therefore
\[
Q_j\Bigr(G(x_{p,j},x),G(x_{p,j},x)\Bigr)
=
2D_{x_i}H(x_{p,j},x_{p,j}).
\]
Since $R(x)=H(x,x)$
and $H(x,y)=H(y,x)$, we have
\[
\partial_iR(x_{p,j})
=
D_{x_i}H(x_{p,j},x_{p,j})+\partial_iH(x_{p,j},x_{p,j})
=
2D_{x_i}H(x_{p,j},x_{p,j}),
\]
hence
\[
Q_j\Bigr(G(x_{p,j},x),G(x_{p,j},x)\Bigr)=\partial_iR(x_{p,j}).
\]

Finally, if $s,m\neq j$, then both factors are smooth near $x_{p,j}$, so every term in \eqref{eq:Q29-decomp} is $o(1)$. Therefore
\[
Q_j\Bigr(G(x_{p,m},x),G(x_{p,s},x)\Bigr)=0.
\]
\end{proof}
\begin{proof} [{\textbf{Proof of \eqref{aluo41}}}]
By the bilinearity of $Q_j(u,v)$, we have
\begin{align}
Q_j\Bigr(G(x_{p,m},x),\partial_hG(x_{p,s},x)\Bigr)
={}&Q_j\Bigr(S(x_{p,m},x),\partial_hH(x_{p,s},x)\Bigr)\notag+Q_j\Bigr(H(x_{p,m},x),\partial_hH(x_{p,s},x)\Bigr)\notag\\
&+Q_j\Bigr(G(x_{p,m},x),\partial_hS(x_{p,s},x)\Bigr).
\label{eq:Q210-decomp}
\end{align}
We first consider the case $s=j$ and $m\neq j$. Since $x_{p,m}\neq x_{p,j}$, the functions $S(x_{p,m},x)$ and $\partial_hH(x_{p,j},x)$ are smooth in a neighborhood of $x_{p,j}$. Therefore, the first two terms on the right-hand side of \eqref{eq:Q210-decomp} tend to zero as $\theta\to0$, and it remains to compute
$
Q_j\Bigr(G(x_{p,m},x),\partial_hS(x_{p,j},x)\Bigr).
$

Now expand $G(x_{p,m},x)$ at $x=x_{p,j}$, we have
\[
\partial_{x_i}G(x_{p,m},x)
=
D_{x_i}G(x_{p,m},x_{p,j})
+\theta\sum_{\ell=1}^4D^2_{x_ix_\ell}G(x_{p,m},x_{p,j})\xi_\ell
+O(\theta^2),
\]
\[
\frac{\partial}{\partial\nu}\partial_{x_i}G(x_{p,m},x)
=
\sum_{\ell=1}^4D^2_{x_ix_\ell}G(x_{p,m},x_{p,j})\xi_\ell+O(\theta),
\]
\[
\Delta_xG(x_{p,m},x)=\Delta_xG(x_{p,m},x_{p,j})+O(\theta),
\qquad
\frac{\partial}{\partial\nu}\Delta_xG(x_{p,m},x)=O(1).
\]

Substituting these identities into the definition of $Q_j$, we obtain
\begin{align}
&Q_j\Bigr(G(x_{p,m},x),\partial_hS(x_{p,j},x)\Bigr)
=\displaystyle\int_{\partial B_\theta(x_{p,j})}
\Bigr[
\Delta_xG(x_{p,m},x)\Delta_x\bigr(\partial_hS(x_{p,j},x)\bigr)\frac{x_i-x_{p,j,i}}{\theta}
\notag\\
& \qquad \qquad +\frac{\partial}{\partial\nu}\Delta_x\bigr(\partial_hS(x_{p,j},x)\bigr)\partial_{x_i}G(x_{p,m},x)
+\frac{\partial}{\partial\nu}\Delta_xG(x_{p,m},x)\partial_{x_i}\bigr(\partial_hS(x_{p,j},x)\bigr)
\notag\\
&\qquad  \qquad -\Delta_xG(x_{p,m},x)\frac{\partial}{\partial\nu}\partial_{x_i}\bigr(\partial_hS(x_{p,j},x)\bigr)
-\Delta_x\bigr(\partial_hS(x_{p,j},x)\bigr)\frac{\partial}{\partial\nu}\partial_{x_i}G(x_{p,m},x)
\Bigr]\,d\sigma.
\label{eq:Q210-main}
\end{align}
Similar to previous proofs, we can  obtain that the third term on the right hand of \eqref{eq:Q210-main} tend to zero. And the first and fourth terms do not vanish separately; rather, their leading contributions cancel. Thus the leading contribution on the right hand of \eqref{eq:Q210-main} comes from the second and fifth terms. For the second term, we have
\begin{align*}
\displaystyle\int_{\partial B_\theta(x_{p,j})}
\frac{\partial}{\partial\nu}\Delta_x\bigr(\partial_hS(x_{p,j},x)\bigr)\partial_{x_i}G(x_{p,m},x)\,d\sigma
=
\frac{3}{2\pi^2\theta^3}
\sum_{\ell=1}^4D^2_{x_ix_\ell}G(x_{p,m},x_{p,j})
\displaystyle\int_{\partial B_\theta(x_{p,j})}\xi_h\xi_\ell\,d\sigma+o(1).
\end{align*}
For the fifth term, since
$
-\Delta_x\bigr(\partial_hS(x_{p,j},x)\bigr)
=
\frac{\xi_h}{2\pi^2\theta^3},
$
we get
\begin{align*}
-\displaystyle\int_{\partial B_\theta(x_{p,j})}
\Delta_x\bigr(\partial_hS(x_{p,j},x)\bigr)\frac{\partial}{\partial\nu}\partial_{x_i}G(x_{p,m},x)\,d\sigma =
\frac{1}{2\pi^2\theta^3}
\sum_{\ell=1}^4D^2_{x_ix_\ell}G(x_{p,m},x_{p,j})
\displaystyle\int_{\partial B_\theta(x_{p,j})}\xi_h\xi_\ell\,d\sigma+o(1).
\end{align*}
Therefore, it follows that
\begin{align*}
Q_j\Bigr(G(x_{p,m},x),\partial_hS(x_{p,j},x)\Bigr)
={}&
\frac{2}{\pi^2\theta^3}
\sum_{\ell=1}^4D^2_{x_ix_\ell}G(x_{p,m},x_{p,j})
\displaystyle\int_{\partial B_\theta(x_{p,j})}\xi_h\xi_\ell\,d\sigma+o(1).
\end{align*}
Since
$
\displaystyle\int_{\partial B_\theta(x_{p,j})}\xi_h\xi_\ell\,d\sigma
=
\frac{|S^3|}{4}\theta^3\delta_{h\ell}
=
\frac{\pi^2}{2}\theta^3\delta_{h\ell},
$
we obtain
\[
Q_j\Bigr(G(x_{p,m},x),\partial_hS(x_{p,j},x)\Bigr)
=
D^2_{x_ix_h}G(x_{p,m},x_{p,j})+o(1).
\]
Hence
\[
Q_j\Bigr(G(x_{p,m},x),\partial_hG(x_{p,j},x)\Bigr)
=
D^2_{x_ix_h}G(x_{p,m},x_{p,j}).
\]

Next, assume $m=j$ and $s\neq j$. Since $\partial_hG(x_{p,s},x)$ is smooth near $x_{p,j}$, repeating the computation in the proof of \eqref{aluo1} with $\partial_hG(x_{p,s},x)$ in place of $G(x_{p,m},x)$, we get
\[
Q_j\Bigr(G(x_{p,j},x),\partial_hG(x_{p,s},x)\Bigr)
=
D_{x_i}\partial_hG(x_{p,s},x_{p,j}).
\]

Next, assume $s=m=j$. By \eqref{eq:Q210-decomp}, we have
\begin{align*}
Q_j\Bigr(G(x_{p,j},x),\partial_hG(x_{p,j},x)\Bigr)
={}&Q_j\Bigr(S(x_{p,j},x),\partial_hS(x_{p,j},x)\Bigr)
+Q_j\Bigr(S(x_{p,j},x),\partial_hH(x_{p,j},x)\Bigr)\\
&+Q_j\Bigr(H(x_{p,j},x),\partial_hS(x_{p,j},x)\Bigr)
+Q_j\Bigr(H(x_{p,j},x),\partial_hH(x_{p,j},x)\Bigr).
\end{align*}
By Lemma \ref{Greenfun}, we can similarly obtain that
$
Q_j\Bigr(S(x_{p,j},x),\partial_hS(x_{p,j},x)\Bigr)=0.
$
Moreover, we obtain
$
Q_j\Bigr(H(x_{p,j},x),\partial_hH(x_{p,j},x)\Bigr)=o(1).
$
Applying the previous computation with $H(x_{p,j},x)$ in place of $G(x_{p,m},x)$, we obtain
\[
Q_j\Bigr(H(x_{p,j},x),\partial_hS(x_{p,j},x)\Bigr)
=
D^2_{x_ix_h}H(x_{p,j},x_{p,j}),
\]
and repeating the computation in the proof of \eqref{aluo1} with $\partial_hH(x_{p,j},x)$ in place of $G(x_{p,m},x)$, we get
\[
Q_j\Bigr(S(x_{p,j},x),\partial_hH(x_{p,j},x)\Bigr)
=
D_{x_i}\partial_hH(x_{p,j},x_{p,j}).
\]
Therefore
\[
Q_j\Bigr(G(x_{p,j},x),\partial_hG(x_{p,j},x)\Bigr)
=
D_{x_i}\partial_hH(x_{p,j},x_{p,j})
+
D^2_{x_ix_h}H(x_{p,j},x_{p,j}).
\]
On the other hand, we get
\[
\partial_{ih}^2R(x_{p,j})
=
D^2_{x_ix_h}H(x_{p,j},x_{p,j})
+2D_{x_i}\partial_hH(x_{p,j},x_{p,j})
+\partial_{ih}^2H(x_{p,j},x_{p,j}).
\]
Since $H(x,y)=H(y,x)$, we also have
\[
\partial_{ih}^2H(x_{p,j},x_{p,j})
=
D^2_{x_ix_h}H(x_{p,j},x_{p,j}),
\]
thus, it holds
\[
\partial_{ih}^2R(x_{p,j})
=
2\Bigr(
D^2_{x_ix_h}H(x_{p,j},x_{p,j})
+
D_{x_i}\partial_hH(x_{p,j},x_{p,j})
\Bigr),
\]
and therefore, we have
\[
Q_j\Bigr(G(x_{p,j},x),\partial_hG(x_{p,j},x)\Bigr)
=
\frac12\,\partial_{ih}^2R(x_{p,j}).
\]

Finally, if $s,m\neq j$, then both factors are smooth near $x_{p,j}$, so every term in \eqref{eq:Q210-decomp} is $o(1)$. Therefore
\[
Q_j\Bigr(G(x_{p,m},x),\partial_hG(x_{p,s},x)\Bigr)=0.
\]
\end{proof}

\section{ The estimate of \eqref{AAA} }\label{76}
\setcounter{equation}{0}
\renewcommand{\theequation}{B.\arabic{equation}}

\setcounter{equation}{0}

\begin{proof}
 By \eqref{6-4-2}, \(\eta_0\) solves
$
\Delta^2\eta_0-e^Z\eta_0
=
-\frac12 Z^2e^Z
~ \text{in } \mathbb R^4
$,
thus $\Delta^2\eta_0=e^Z\left(\eta_0-\frac12 Z^2\right) \in L^1({\mathbb R^4})$, and it follows that
\[
A=\lim_{R\to\infty}\int_{B_R}\Delta^2\eta_0\,dx .
\]
We define $\phi(y):=\eta_0(\kappa y)$ with $\displaystyle \kappa:=\sqrt{8\sqrt6}$, then $\phi(y)$ satisfies
\begin{equation*}
\Delta^2\phi - V(y)\phi = -S(y)
 ~\text{in }\mathbb R^4,
\end{equation*}
where
\begin{equation*}
V(y):=\frac{384}{ (1+|y|^2)^4 },
~~~
S(y):=
\frac{3072\log^2(1+|y|^2)}{(1+|y|^2)^4}.
\end{equation*}
Moreover, it's not hard to find that
\begin{equation*}
A=\lim_{R\to\infty}\int_{B_R}\Delta^2\phi(y)\,dy .
\end{equation*}
We now reduce the computation to the radial part. Let
\[
\bar{\phi}(r):=
\frac1{|S^3|}
\int_{S^3}\phi(r\theta)\,d\theta,
~\text{with} ~ r=|y|.
\]
Since \(V\) and \(S\) are radial, \(\bar{\phi}\) satisfies
\begin{equation}\label{eq:radial-equation}
\Delta^2\bar{\phi}-V(r)\bar{\phi}=-S(r).
\end{equation}
Furthermore, by spherical coordinate transformation, we know that
\[
\int_{B_R}\Delta^2\phi\,dy
=
\int_{B_R}\Delta^2\bar{\phi}\,dy ~~~\text{holds for every \(R>0\)}.
\]
Introduce the Emden-Fowler transformation \( r=e^t\), \(v(t):=\bar{\phi}(e^t) \), and then we have
\begin{equation}\label{eq:bilap-cylinder}
\Delta\bar{\phi}
=
e^{-2t}(v''+2v'), ~~\text{and}~~\Delta^2\bar{\phi}
=
e^{-4t}(v^{(4)}-4v'').
\end{equation}
Since
$$e^{4t}V(e^t)
=
\frac{384e^{4t}}{(1+e^{2t})^4}
=
24\operatorname{sech}^4 t,
~ ~
e^{4t}S(e^t)
=
\frac{3072e^{4t}\log^2(1+e^{2t})}{(1+e^{2t})^4}
=
192\operatorname{sech}^4 t\log^2(1+e^{2t}),
$$ multiplying \eqref{eq:radial-equation} by \(e^{4t}\), we obtain
\begin{equation}\label{eq:one-dimensional}
Lv
=
-192\operatorname{sech}^4 t\log^2(1+e^{2t}),
\end{equation}
where
\begin{equation*}
L:=D_t^4-4D_t^2-24\operatorname{sech}^4t.
\end{equation*}
Let
\[
W(y):=Z\!\left(\kappa y\right)
=-4\log(1+|y|^2).
\]
The scaling mode associated with \(W\) is
\[
\psi_0(t)
:=
4+y\cdot\nabla W(y)\big|_{|y|=e^t}
=
4\frac{1-e^{2t}}{1+e^{2t}}
=
-4\tanh t.
\]
A direct computation gives
\begin{equation}\label{eq:kernel}
L\psi_0=0.
\end{equation}
For any smooth functions \(u,v\), the following Lagrange identity holds
\begin{equation}\label{eq:lagrange}
uLv-vLu
=
\frac{d}{dt}B[u,v],
\end{equation}
where
\begin{equation*}
B[u,v]
:=
uv'''-u'v''+u''v'-u'''v
-4(uv'-u'v).
\end{equation*}
Taking \(u=\psi_0\) in \eqref{eq:lagrange} and combining \eqref{eq:kernel}, we get
\begin{equation}\label{eq:lagrange-integrated}
\int_{-T}^{T}\psi_0(t)Lv(t)\,dt
=
B[\psi_0,v](T)-B[\psi_0,v](-T).
\end{equation}
On one hand, since \(\bar{\phi}\) is smooth at the origin, by Taylor's expansion and $\bar{\phi}$ is an even function with $\eta_0(0)=0$, there exist a constant \(c\) such that
\[
\bar{\phi}(r)=cr^2+O(r^4)
\quad\text{as }r\to0.
\]
Then by elliptic  regular estimates (refer to \cite{GGS}), we get
\[
v(t)=O(e^{2t})
~~ \text{and}~~
v^{(k)}(t)=O(e^{2t})
 ~~ \text{for}~k\ge1,~~ \text{as}~~
 t\to-\infty.
\]
By \(\psi_0(t)\to4\) and \(\psi_0^{(k)}(t)=O(e^{2t})\) for \(k\ge1\) as \(t\to-\infty\), we obtain
\begin{equation}\label{eq:left-boundary}
\lim_{T\to\infty}B[\psi_0,v](-T)=0.
\end{equation}
On the other hand, by \eqref{boundw0}, together with elliptic estimates, we have
\[
v^{(k)}(t)=O(e^{\tau t}),
\qquad k=0,1,2,3,
\qquad t\to+\infty.
\]
Since
$
\psi_0(t)\to -4,~
\text{and}
~\psi_0^{(k)}(t)=O(e^{-2t})
~ \text{for } k\ge1 ~\text{as}~ t\to+\infty,
$
all terms in \(B[\psi_0,v]\) involving derivatives of \(\psi_0\) vanish as  \(t \rightarrow+\infty\). Thus
\begin{equation}\label{eq:right-boundary}
\lim_{t\to+\infty}B[\psi_0,v](t)
=
-4\lim_{t\to+\infty}\bigl(v'''(t)-4v'(t)\bigr).
\end{equation}
 From \eqref{eq:bilap-cylinder}, we have
$
\partial_r\Delta\bar{\phi}
=
e^{-3t}\bigl(v'''(t)-4v'(t)\bigr).
$
Therefore, by the divergence theorem and \(|S^3|=2\pi^2\), it follows that
\begin{equation}\label{B.9}
A
=
\lim_{R\to\infty}\int_{B_R}\Delta^2\bar{\phi}\,dy  =
\lim_{R\to\infty}\int_{\partial B_R}
\partial_r\Delta\bar{\phi}\,d\sigma
=
2\pi^2
\lim_{t\to+\infty}
\bigl(v'''(t)-4v'(t)\bigr).
\end{equation}
Combining \eqref{eq:lagrange-integrated}, \eqref{eq:left-boundary}, \eqref{eq:right-boundary}, and \eqref{B.9}, we obtain
\begin{equation}\label{B10}
\int_{-\infty}^{+\infty}\psi_0(t)Lv(t)\,dt
=
-\frac{2A}{\pi^2}.
\end{equation}
It remains to compute the integral explicitly. By \eqref{eq:one-dimensional} and the formula
\(\psi_0(t)=-4\tanh t\), we have
\[
\psi_0(t)Lv(t)
=
768\tanh t\,\operatorname{sech}^4t
\log^2(1+e^{2t}).
\]
Hence, \eqref{B10} becomes
\begin{equation}\label{eq:int-t}
-\frac{2A}{\pi^2}
=
768\int_{-\infty}^{+\infty}
\tanh t\,\operatorname{sech}^4t
\log^2(1+e^{2t})\,dt.
\end{equation}
Set \(x=e^{2t}\), then
$
dt=\frac{dx}{2x},
~
\tanh t=\frac{x-1}{x+1},
~
\operatorname{sech}^4t
=
\frac{16x^2}{(1+x)^4}.
$
and  we can obtain from \eqref{eq:int-t} that
\begin{equation*}
-\frac{2A}{\pi^2}
=
768\int_0^\infty
\frac{x-1}{x+1}
\frac{16x^2}{(1+x)^4}
\log^2(1+x)
\frac{dx}{2x} \notag
=
6144
\int_0^\infty
\frac{x(x-1)}{(1+x)^5}
\log^2(1+x)\,dx .
\end{equation*}
Now let \(m=1+x\), then
$
\frac{x(x-1)}{(1+x)^5}
=
\frac{(m-1)(m-2)}{m^5}
=
m^{-3}-3m^{-4}+2m^{-5},
$
and we have
\begin{equation}\label{eq:int-m}
-\frac{2A}{\pi^2}
=
6144
\int_1^\infty
\left(
m^{-3}-3m^{-4}+2m^{-5}
\right)
(\log m)^2\,dm .
\end{equation}
For \(p>1\), we have
$
\int_1^\infty m^{-p}(\log m)^2\,dm
=
\frac{2}{(p-1)^3}.
$
Therefore, it holds
\begin{equation}\label{B1}
    \int_1^\infty m^{-3}(\log m)^2\,dm=\frac14, \quad
    \int_1^\infty m^{-4}(\log m)^2\,dm=\frac{2}{27},
~\text{ and }~
    \int_1^\infty m^{-5}(\log m)^2\,dm=\frac1{32}.
\end{equation}
Consequently, substituting \eqref{B1} into \eqref{eq:int-m}, we obtain
\begin{equation}
-\frac{2A}{\pi^2}
=
6144
\left(
\frac14
-
3\cdot\frac{2}{27}
+
2\cdot\frac1{32}
\right) \notag
=
\frac{1664}{3}.
\end{equation}
Thus, it holds
\[
A
=
-\frac{\pi^2}{2}\cdot\frac{1664}{3}
=
-\frac{832}{3}\pi^2.
\]
\end{proof}

\vskip 0.5cm

\section{Some explicit bubble integrals used in the Morse index computation}\label{MI:appendix-integrals}
\setcounter{equation}{0}
\renewcommand{\theequation}{C.\arabic{equation}}

For completeness, we collect here the elementary radial integrals used in Section~\ref{section4}. Here we verify all constants used in the proof.
Put
\[
s:=\frac{r^2}{a},
\qquad
a=8\sqrt6,
\qquad
|S^3|=2\pi^2.
\]
Then
\[
e^Z=(1+s)^{-4},
\qquad
Y_0=\frac{1-s}{1+s},
\qquad
z\cdot\nabla Z=-\frac{8s}{1+s},
\]
and the four-dimensional radial measure becomes
\begin{equation}\label{MI:int:eq:measure}
2\pi^2r^3\dd r
=
\pi^2a^2s\dd s.
\end{equation}

We first calculate the integral \(\int_{\R^4}e^ZY_0\). Using \eqref{MI:int:eq:measure}, we have
\[
\int_{\R^4}e^ZY_0
=
\pi^2a^2
\int_0^\infty
\frac{s(1-s)}{(1+s)^5}\dd s.
\]
Since
\[
\int_0^\infty\frac{s}{(1+s)^5}\dd s
=
B(2,3)=\frac1{12},
\qquad
\int_0^\infty\frac{s^2}{(1+s)^5}\dd s
=
B(3,2)=\frac1{12}.
\]
Thus
\[
\int_{\R^4}e^ZY_0=0.
\]

For the integral \(\int_{\R^4}Ze^ZY_0\). Note that \(Z=-4\log(1+s)\), it follows that
\[
\int_{\R^4}Ze^ZY_0
=
-4\pi^2a^2
\int_0^\infty
\frac{s(1-s)\log(1+s)}{(1+s)^5}\dd s.
\]
Differentiating the beta identity
\[
\int_0^\infty \frac{s^{\alpha-1}}{(1+s)^\beta}\dd s
=
\frac{\Gamma(\alpha)\Gamma(\beta-\alpha)}
{\Gamma(\beta)}
\]
with respect to \(\beta\) gives
\[
\int_{\R^4}Ze^ZY_0=64\pi^2.
\]

Then, we compute the integral \( \int_{\R^4}e^ZY_0\,z\cdot\nabla Z\). A direct computation yields
\[
\int_{\R^4}e^ZY_0\,z\cdot\nabla Z
=
-8\pi^2a^2
\int_0^\infty
\frac{s^2(1-s)}{(1+s)^6}\dd s
=
-8\pi^2a^2
\left[
B(3,3)-B(4,2)
\right].
\]
From
\[
B(3,3)=\frac1{30},
\qquad
B(4,2)=\frac1{20},
\]
we have
\[
\int_{\R^4}e^ZY_0\,z\cdot\nabla Z
=
\frac{2\pi^2a^2}{15}
=
\frac{256\pi^2}{5}.
\]

Next, we calculate the integral \( \int_{\R^4}z_i e^ZY_i \). For \(i,h\in\{1,\dots,4\}\), rotational symmetry gives zero when
\(i\ne h\), while
\[
\int_{\R^4}z_i^2F(r)\dd z
=
\frac{|S^3|}{4}\int_0^\infty r^5F(r)\dd r.
\]
Therefore
\[
\int_{\R^4}z_i e^ZY_i
=
\int_{\R^4}
\frac{z_i^2e^Z}{a+r^2}\dd z
=
\frac{\pi^2a^2}{4}
\int_0^\infty
\frac{s^2}{(1+s)^5}\dd s
=
\frac{\pi^2a^2}{4}B(3,2)
=
8\pi^2.
\]

And for the integral \( \int_{\R^4}e^ZY_i\partial_iZ \). Since
\[
\partial_iZ=-\frac{8z_i}{a+r^2},
\]
we can obtain that
\[
\int_{\R^4}e^ZY_i\partial_iZ
=
-8
\int_{\R^4}
\frac{z_i^2e^Z}{(a+r^2)^2}\dd z
=
-2\pi^2a
\int_0^\infty
\frac{s^2}{(1+s)^6}\dd s
=
-2\pi^2a B(3,3)
=
-\frac{\pi^2a}{15}.
\]
Thus
\[
\int_{\R^4}e^ZY_h\partial_iZ
=
-\frac{\pi^2a}{15}\delta_{ih}.
\]

Similar to the above computation, we can check that
\[
\int_{\R^4}e^ZY_0^2
=
\pi^2a^2
\int_0^\infty
\frac{s(1-s)^2}{(1+s)^6}\dd s
=
\frac{\pi^2a^2}{30}
=
\frac{64\pi^2}{5}
\quad
\text{and}
\quad
\int_{\R^4}e^ZY_iY_h
=
\frac{\pi^2a}{120}\delta_{ih}
=
\frac{\pi^2\sqrt6}{15}\delta_{ih}.
\]

\vskip 0.5cm

\noindent\textbf{Acknowledgments} The authors are grateful to  thank Prof. Peng Luo  for  his generous academic guidance and many insightful discussions throughout this study. Shuying Tian was supported by the National Natural Science Foundation of China (Grant Nos. 12671239).

\vskip 0.2cm

\noindent\textbf{Data Availability}

\vskip 0.1cm

  Data sharing is not applicable to this article as no new data were created or analyzed in this
study.

\vskip 0.2cm

\noindent\textbf{Declarations}	

\vskip 0.1cm

 The authors declare that they have no conflict of interest.

\renewcommand\refname{References}
\renewenvironment{thebibliography}[1]{%
  \section*{\refname}
  \list{\arabic{enumi}}{%
    \def\makelabel##1{\hss{##1}}
    \topsep=0mm
    \parsep=0mm
    \partopsep=0mm
    \itemsep=0mm
    \labelsep=1ex
    \itemindent=0mm
    \settowidth\labelwidth{\small[#1]}%
    \leftmargin\labelwidth \advance\leftmargin\labelsep
    \advance\leftmargin -\itemindent
    \usecounter{enumi}
  }%
  \small
  \def\newblock{\ }
  \sloppy
  \clubpenalty4000
  \widowpenalty4000
  \sfcode`\.=1000\relax
}{%
\endlist
}


\end{document}